\documentclass{amsart}
\usepackage{setspace}
\usepackage{a4}
\usepackage{amssymb,amsmath,amsthm,latexsym}
\usepackage{amsfonts}
\usepackage{amsfonts}
\usepackage{graphicx}
\usepackage{textcomp}
\usepackage{cite}
\usepackage{enumerate}
\usepackage[mathscr]{euscript}
\usepackage{mathtools}
\newtheorem{theorem}{Theorem}[section]

\newtheorem{corollary}[theorem] {Corollary}
\newtheorem{definition}[theorem]{Definition}

\newtheorem{lemma} [theorem]{Lemma}

\newtheorem{proposition}[theorem]{Proposition}
\newtheorem{remark}[theorem]{Remark}

\title{This is the title}
\usepackage{amssymb}
\usepackage{amsmath}
\usepackage{tikz}
\usepackage{hyperref}

\begin{document}

\begin{center}
{\bf{\large}EXTENSION OF FRAMES AND BASES - II}\\
K. MAHESH KRISHNA  AND P. SAM JOHNSON \\
Department of Mathematical and Computational Sciences\\ 
National Institute of Technology Karnataka (NITK), Surathkal\\
Mangaluru 575 025, India  \\
Emails: kmaheshak@gmail.com, kmaheshakma16f02@nitk.edu.in, \\
       nitksam@gmail.com, sam@nitk.edu.in \\
       
Date: \today
\end{center}

\hrule
\vspace{0.5cm}
\textbf{Abstract}:  Operator-valued frame ($G$-frame), as a generalization of frame is introduced by Kaftal, Larson, and Zhang  in \textit{Trans. Amer. Math. Soc.}, 361(12):6349-6385,   2009 and by Sun in \textit{J. Math. Anal. Appl.}, 322(1):437-452, 2006.
It has been further extended in the paper arXiv:1810.01629 [math.OA] 3 October 2018, so as to have a rich theory on operator-valued frames for Hilbert spaces as well as for Banach spaces.  The continuous version has been studied in this paper when the indexing set is a measure space. We study duality, similarity, orthogonality and stability of this extension. Several characterizations are given including a notable characterization when the measure space is a locally compact group. Variation formula, dimension formula and trace formula are derived when the Hilbert space is finite dimensional.


\tableofcontents
\section{Introduction}\label{INTRODUCTION}
Let $ \mathcal{H}$,  $ \mathcal{H}_0 $ be Hilbert spaces,  $ \mathcal{B}(\mathcal{H}, \mathcal{H}_0)$ be the Banach space of all bounded linear operators from $ \mathcal{H}$ to $ \mathcal{H}_0 $ and $ \mathcal{B}(\mathcal{H})\coloneqq \mathcal{B}(\mathcal{H}, \mathcal{H})$. Letter $ \mathbb{J}$ denotes an   indexing set and $\mathbb{K}$ denotes the   field of scalars ($ \mathbb{R}$ or $ \mathbb{C}$). 

\begin{definition}\cite{DUFFIN1, OLE1}\label{OLE}
A collection $ \{x_j\}_{j \in \mathbb{J}}$ in  a Hilbert space $ \mathcal{H}$ is said to be a (discrete) 
\begin{enumerate}[\upshape(i)]
\item   frame for $\mathcal{H}$ if there exist $ a, b >0$ such that
\begin{equation*}
a\|h\|^2\leq\sum_{j \in \mathbb{J}}|\langle h, x_j \rangle|^2 \leq b\|h\|^2  ,\quad \forall h \in \mathcal{H}.
\end{equation*}
\item   Bessel sequence for $\mathcal{H}$ if there exists $  b >0$ such that

\begin{equation*}
\sum_{j \in \mathbb{J}}|\langle h, x_j \rangle|^2 \leq b\|h\|^2  ,\quad \forall h \in \mathcal{H}.
\end{equation*}
\end{enumerate}
\end{definition}

 We refer \cite{OLE1, YOUNG, HEIL, CASAZZAKUTYNIOK, HANKORNELSONLARSONWEBER, HANLARSON, CASAZZAART, CERDA, FEICHTINGERLUEF, DAISUN, BALANEQUI, FRANKPAULSENTIBALLI} for more details on frames (and a well studied class of frames) and Bessel sequences in Hilbert spaces.
 Most  general version of  Definition \ref{OLE} is
 \begin{definition}\cite{KAFTALLARSONZHANG1, SUN1}\label{KAFTAL}
 A collection  $ \{A_j\}_{j \in \mathbb{J}} $  in $ \mathcal{B}(\mathcal{H}, \mathcal{H}_0)$ is said to be 	an operator-valued
 \begin{enumerate}[\upshape(i)]
 \item  frame in $ \mathcal{B}(\mathcal{H}, \mathcal{H}_0)$ if the series $ \sum_{j\in \mathbb{J}} A_j^*A_j$  converges in the strong-operator topology on $ \mathcal{B}(\mathcal{H})$ to a  bounded positive invertible operator.
 \item  Bessel sequence in $ \mathcal{B}(\mathcal{H}, \mathcal{H}_0)$ if the series $ \sum_{j\in \mathbb{J}} A_j^*A_j$  converges in the strong-operator topology on $ \mathcal{B}(\mathcal{H})$ to a  bounded  positive operator.
 \end{enumerate}	
\end{definition}
 We refer \cite{KAFTALLARSONZHANG1, SUN1, HANLIMENG, MENG, SUNSTAB} for more details on operator-valued frames and Bessel sequences in Hilbert spaces.
 In \cite{MAHESHKRISHNASAMJOHNSON} we defined the following two definitions.
 \begin{definition}\cite{MAHESHKRISHNASAMJOHNSON}
Let $ \{\tau_j\}_{j\in \mathbb{J}}$ be a set of vectors in  a Hilbert space $\mathcal{H}$. A set of vectors   $ \{x_j\}_{j\in \mathbb{J}}$  in   $\mathcal{H}$ is said to be a
 \begin{enumerate}[\upshape(i)]
 \item  frame  with respect to   (w.r.t.)   $ \{\tau_j\}_{j\in \mathbb{J}}$  if there are  $  c,  d  >0$ such that
 \begin{enumerate}[\upshape(a)]
 \item the map  $\mathcal{H} \ni  h \mapsto \sum_{j\in\mathbb{J}}\langle h,  x_j\rangle\tau_j \in \mathcal{H} $ is a well-defined  bounded positive  invertible operator.
 \item $  \sum_{j \in \mathbb{J}}|\langle h,x_j\rangle|^2  \leq c\|h\|^2 , \forall h \in \mathcal{H}; 
 \sum_{j \in \mathbb{J}}|\langle h,\tau_j\rangle|^2 \leq d\|h\|^2 , \forall h \in \mathcal{H}.$
 \end{enumerate}
 \item   Bessel sequence    w.r.t.  $ \{\tau_j\}_{j\in \mathbb{J}}$   if there are  $  c,  d  >0$ such that
 \begin{enumerate}[\upshape(a)]
 \item the map  $\mathcal{H} \ni  h \mapsto \sum_{j\in\mathbb{J}}\langle h,  x_j\rangle\tau_j \in \mathcal{H} $ is a well-defined  bounded positive  operator.
 \item $  \sum_{j \in \mathbb{J}}|\langle h,x_j\rangle|^2  \leq c\|h\|^2 , \forall h \in \mathcal{H}; 
 \sum_{j \in \mathbb{J}}|\langle h,\tau_j\rangle|^2 \leq d\|h\|^2 , \forall h \in \mathcal{H}.$
 \end{enumerate}
 \end{enumerate}
 \end{definition}
 \begin{definition}\cite{MAHESHKRISHNASAMJOHNSON}\label{DEFOVF}
 Define 
 $ L_j : \mathcal{H}_0 \ni h \mapsto e_j\otimes h \in  \ell^2(\mathbb{J}) \otimes \mathcal{H}_0$,  where $\{e_j\}_{j \in \mathbb{J}} $ is  the standard orthonormal basis for $\ell^2(\mathbb{J})$,  for each $ j \in \mathbb{J}$. A collection $ \{A_j\}_{j \in \mathbb{J}} $  in $ \mathcal{B}(\mathcal{H}, \mathcal{H}_0)$ is said to be an operator-valued  	
 \begin{enumerate}[\upshape(i)]
\item frame  in $ \mathcal{B}(\mathcal{H}, \mathcal{H}_0) $  with respect to a collection  $ \{\Psi_j\}_{j \in \mathbb{J}}  $ in $ \mathcal{B}(\mathcal{H}, \mathcal{H}_0) $ if 
 \begin{enumerate}[\upshape(a)]
 \item the series $ \sum_{j\in \mathbb{J}} \Psi_j^*A_j$  converges in the strong-operator topology  on $ \mathcal{B}(\mathcal{H})$ to a  bounded positive invertible operator,
 \item both $ \sum_{j\in \mathbb{J}} L_jA_j$, $ \sum_{j\in \mathbb{J}} L_j\Psi_j$ converge in the strong-operator topology on $ \mathcal{B}(\mathcal{H},\ell^2(\mathbb{J}) \otimes \mathcal{H}_0 )$ to  bounded operators.
 \end{enumerate}
 \item Bessel sequence  in $ \mathcal{B}(\mathcal{H}, \mathcal{H}_0) $  with respect to a collection  $ \{\Psi_j\}_{j \in \mathbb{J}}  $ in $ \mathcal{B}(\mathcal{H}, \mathcal{H}_0) $ if 
 \begin{enumerate}[\upshape(a)]
 \item the series $ \sum_{j\in \mathbb{J}} \Psi_j^*A_j$  converges in the strong-operator topology  on $ \mathcal{B}(\mathcal{H})$ to a  bounded positive  operator,
 \item both $ \sum_{j\in \mathbb{J}} L_jA_j$, $ \sum_{j\in \mathbb{J}} L_j\Psi_j$ converge in the strong-operator topology on $ \mathcal{B}(\mathcal{H},\ell^2(\mathbb{J}) \otimes \mathcal{H}_0 )$ to  bounded operators.
 \end{enumerate}
 \end{enumerate}
 \end{definition}
 
  All of our vector-valued integrals are in the weak-sense (i.e., they are Gelfand-Pettis integral and we refer \cite{TALAGRAND, MUSIAL, HILDEBRANDT, MUSIAL1, PETTIS, MASANI, EDGAR, MASANINIEMI} for more details). $ \Omega$ denotes a measure space with positive measure $\mu$.

Continuous frame, as a generalization of frames was introduced independently by Ali, Antoine, Gazeau \cite{ALI1}  and Kaiser \cite{KAISER1}.
\begin{definition}\cite{ALI1, KAISER1} \label{SEQUENTIALSINGLEKAISERALI}
A set of vectors   $ \{x_\alpha\}_{\alpha\in \Omega}$  in  $\mathcal{H}$ is said to be a continuous frame    for $\mathcal{H}$ if 
\begin{enumerate}[\upshape(i)]
\item for each  $h \in \mathcal{H}$,  the  map   $\Omega \ni \alpha \mapsto\langle  h, x_\alpha\rangle \in \mathbb{K}$ is measurable,
\item there exist $a,b>0$ such that 

\begin{align*}
a\|h\|^2\leq \int_{\Omega}|\langle h, x_\alpha\rangle|^2 \,d\mu(\alpha)\leq  b \|h\|^2,\quad \forall h \in \mathcal{H}.
\end{align*}
\end{enumerate}	
\end{definition}
We refer \cite{KAISER1, ALI1, GABARDOHANADV, GROHS, FORNASIERRAUHUT, RAHIMINAJATIDEHGHAN, IVERSON1, IVERSON2} for more details on continuous frames. We also refer \cite{FORNASIERRAUHUT, FREEMANSPEEGLE, ALIBOOK} for connections between continuous frames and discrete frames.

\section{Extension of continuous operator-valued frames}\label{MK}
In order to set continuous version of  Definition \ref{DEFOVF},  we want existence of certain operators, for which  we use the following definition.
\begin{definition}\cite{ABDOLLAHPOURFAROUGHI, HANLARSONLIULIU} \label{FIRSTDEFINITION}
Let $ \Omega$ be a measure space with positive measure $\mu$.  A collection $ \{A_\alpha\}_{\alpha \in \Omega} $  in $ \mathcal{B}(\mathcal{H}, \mathcal{H}_0)$ is said to be   \textit{continuous operator-valued Bessel}	if 
\begin{enumerate}[\upshape(i)]
\item for each  $h \in \mathcal{H}$, the map  $\Omega \ni \alpha \mapsto A_\alpha h\in \mathcal{H}_0$ is measurable,
\item there exists $  b >0$ such that
\begin{equation}\label{BESSELCONDITION}
\int_{\Omega}\|A_\alpha h\|^2\,d\mu(\alpha) \leq b\|h\|^2  ,\quad \forall h \in \mathcal{H}.
\end{equation}
\end{enumerate}
\end{definition}
Let $ \{A_\alpha\}_{\alpha \in \Omega} $ and  $ \{\Psi_\alpha\}_{\alpha \in \Omega} $  in $ \mathcal{B}(\mathcal{H}, \mathcal{H}_0)$  be  continuous operator-valued Bessel with bounds $b$ and $d$, respectively. Continuity of norm and polarization identity reveal that the map $\Omega \ni \alpha \mapsto \langle A_\alpha h, \Psi_\alpha g\rangle \in \mathbb{K}$ is measurable, for each fixed $h,g \in \mathcal{H}$. Cauchy-Schwarz inequality and Inequality (\ref{BESSELCONDITION}) now tell that this map is in $\mathcal{L}^2(\Omega,\mathbb{K})$, explicitly,

\begin{align*}
\left|\int_{\Omega}\langle A_\alpha h, \Psi_\alpha g \rangle\,d\mu(\alpha)  \right|&\leq \int_{\Omega}|\langle A_\alpha h, \Psi_\alpha g \rangle|\,d\mu(\alpha)\leq \int_{\Omega}\| A_\alpha h\|\|\Psi_\alpha g \|\,d\mu(\alpha)\\
&\leq \left(\int_{\Omega}\|A_\alpha h\|^2\,d\mu(\alpha)\right)^\frac{1}{2}\left(\int_{\Omega}\|\Psi_\alpha g\|^2\,d\mu(\alpha)\right)^\frac{1}{2}\leq \sqrt{bd} \|h\|\|g\|.
\end{align*}
Previous inequalities also show that for each fixed $h \in \mathcal{H}$,   the map 

\begin{align*}
\zeta_h:\mathcal{H} \ni g \mapsto \int_{\Omega}\langle A_\alpha h, \Psi_\alpha g \rangle\,d\mu(\alpha) \in \mathbb{K}
\end{align*}
is a conjugate-linear bounded functional with $\|\zeta_h\|_{\text{op}}\leq \sqrt{bd}\|h\|$ (where $\|\cdot\|_{\text{op}}$ denotes the operator-norm). Let $\int_{\Omega}\Psi_\alpha^*A_\alpha h\,d\mu(\alpha)$ be that unique element (which comes from Riesz representation theorem)  of $\mathcal{H}$ such that 
\begin{align*}
 \int_{\Omega}\langle A_\alpha h, \Psi_\alpha g \rangle\,d\mu(\alpha)&=\zeta_h g=\left\langle \int_{\Omega}\Psi_\alpha^*A_\alpha h\,d\mu(\alpha), g\right\rangle, ~ \forall g \in \mathcal{H} \text{ and }\\
  \|\zeta_h\|_{\text{op}}&=\left\|\int_{\Omega}\Psi_\alpha^*A_\alpha h\,d\mu(\alpha)\right\|.
\end{align*}
By varying $h \in \mathcal{H}$, we get the map

\begin{align*}
S_{A,\Psi}:\mathcal{H} \ni h \mapsto S_{A,\Psi}h\coloneqq \int_{\Omega}\Psi_\alpha^*A_\alpha h\,d\mu(\alpha)\in \mathbb{K}.
\end{align*}
Above map is a bounded linear operator,  $S_{A,\Psi}^*=S_{\Psi,A}$ and $S_{A,A}$ is positive. Indeed,

\begin{align*}
 \|S_{A,\Psi}\|&=\sup_{h\in \mathcal{H}, \|h\|\leq 1}\|S_{A,\Psi}h\|=\sup_{h\in \mathcal{H}, \|h\|\leq 1}\left\|\int_{\Omega}\Psi_\alpha^*A_\alpha h\,d\mu(\alpha)\right\|\\
 &=\sup_{h\in \mathcal{H}, \|h\|\leq 1}\|\zeta_h\|_{\text{op}}\leq \sup_{h\in \mathcal{H}, \|h\|\leq 1}\sqrt{bd}\|h\|=\sqrt{bd}.
\end{align*}
and 

\begin{align*}
\langle S_{A,\Psi}h, g\rangle&=\int_{\Omega}\langle A_\alpha h, \Psi_\alpha g \rangle\,d\mu(\alpha)=\int_{\Omega}\overline{\langle  \Psi_\alpha g, A_\alpha h \rangle}\,d\mu(\alpha)\\
&=\overline{\int_{\Omega}\langle  \Psi_\alpha g, A_\alpha h \rangle\,d\mu(\alpha)}=\overline{\langle S_{\Psi, A}g, h\rangle}=\langle h, S_{\Psi, A}g\rangle,\quad \forall h,g \in \mathcal{H},\\
\langle S_{A,A}h, h\rangle&=\int_{\Omega}\|A_\alpha h\|^2 \,d\mu(\alpha)\geq 0,\quad \forall h \in \mathcal{H}.
\end{align*}
We further note that Inequality (\ref{BESSELCONDITION}) gives that 

\begin{equation*}
\theta_A:\mathcal{H} \ni h \mapsto \theta_Ah \in \mathcal{L}^2(\Omega, \mathcal{H}_0) , \quad\theta_Ah: \Omega \ni \alpha \mapsto A_\alpha h \in \mathcal{H}_0
\end{equation*}
 is a well-defined bounded linear operator whose adjoint is 
\begin{equation*}
\theta_A^*:\mathcal{L}^2(\Omega, \mathcal{H}_0)\ni f \mapsto \int_{\Omega}A_\alpha^*f(\alpha) \,d\mu(\alpha)\in \mathcal{H},
\end{equation*}
where the integral  is in the weak-sense. In fact, 

\begin{align*}
\|\theta_Ah\|^2=\int_{\Omega}\|\theta_Ah(\alpha)\|^2\,d\mu(\alpha)=\int_{\Omega}\|A_\alpha h\|^2\,d\mu(\alpha)\leq b \|h\|^2,\quad \forall h \in \mathcal{H}
\end{align*}
and 

\begin{align*}
\langle \theta_Ah, f\rangle&=\int_{\Omega} \langle\theta_Ah(\alpha), f(\alpha)\rangle\,d\mu(\alpha)=\int_{\Omega} \langle A_\alpha h, f(\alpha)\rangle\,d\mu(\alpha)\\
&=\int_{\Omega} \langle h,  A_\alpha^* f(\alpha)\rangle\,d\mu(\alpha)=\langle h, \theta_A^* f\rangle,\quad \forall h\in \mathcal{H}, \forall f \in \mathcal{L}^2(\Omega, \mathcal{H}_0).
\end{align*}
We next observe that Condition (ii) in Definition \ref{FIRSTDEFINITION} holds if and only if the map $\theta_A$ is a well-defined bounded linear operator. With this knowledge we are ready to define the continuous version of Definition \ref{DEFOVF}.
\begin{definition}\label{1}
  A collection $ \{A_\alpha\}_{\alpha \in \Omega} $  in $ \mathcal{B}(\mathcal{H}, \mathcal{H}_0)$ is said to be a \textit{ continuous operator-valued frame} (in short,  continuous (ovf)) in $ \mathcal{B}(\mathcal{H}, \mathcal{H}_0) $  with respect to a collection  $ \{\Psi_\alpha\}_{\alpha \in \Omega}  $ in $ \mathcal{B}(\mathcal{H}, \mathcal{H}_0) $ if 
\begin{enumerate}[\upshape(i)]
\item for each  $h \in \mathcal{H}$, both  maps   $\Omega \ni \alpha \mapsto A_\alpha h\in \mathcal{H}_0$ and $\Omega \ni\alpha \mapsto \Psi_\alpha h\in \mathcal{H}_0$ are measurable,
\item the map (we call as frame operator) $S_{A,\Psi} : \mathcal{H} \ni h \mapsto \int_{\Omega}\Psi_\alpha^*A_\alpha h  \,d\mu(\alpha)\in \mathcal{H}$ (the integral is in the weak-sense) is a well-defined bounded positive invertible operator,
\item  both maps (we call as analysis operator and its adjoint as synthesis operator) $ \theta_A:\mathcal{H} \ni h \mapsto \theta_Ah \in \mathcal{L}^2(\Omega, \mathcal{H}_0) $, $\theta_Ah: \Omega \ni \alpha \mapsto A_\alpha h \in \mathcal{H}_0$ and  $\theta_\Psi:\mathcal{H} \ni h \mapsto \theta_\Psi h \in \mathcal{L}^2(\Omega, \mathcal{H}_0) $, $\theta_\Psi h: \Omega \ni \alpha \mapsto \Psi_\alpha h \in \mathcal{H}_0 $ are well-defined bounded linear operators.
\end{enumerate}
 We note that $\theta_A^*:\mathcal{L}^2(\Omega, \mathcal{H}_0)\ni f \mapsto \int_{\Omega}A_\alpha^*f(\alpha) \,d\mu(\alpha)\in \mathcal{H} $, $\theta_\Psi^*:\mathcal{L}^2(\Omega, \mathcal{H}_0)\ni f \mapsto \int_{\Omega}\Psi_\alpha^*f(\alpha) \,d\mu(\alpha)\in \mathcal{H}  $ (both integrals are in the weak-sense). Notions of frame bounds, Parseval frame are similar to the same in Definition 2.1 in \cite{MAHESHKRISHNASAMJOHNSON}.

 Whenever $ \{A_\alpha \}_{\alpha \in \Omega}$  is a continuous operator-valued frame w.r.t. $ \{\Psi_\alpha \}_{\alpha \in \Omega}$ we write $(\{A_\alpha\}_{\alpha \in \Omega}, \{\Psi_\alpha\}_{\alpha \in \Omega}) $ is continuous (ovf).
 
For fixed  $ \Omega$, $\mathcal{H}, \mathcal{H}_0 $ and $ \{\Psi_\alpha \}_{\alpha \in \Omega}$, the set of all continuous  operator-valued frames in $ \mathcal{B}(\mathcal{H}, \mathcal{H}_0) $ with respect to collection  $ \{\Psi_\alpha \}_{\alpha \in \Omega}$ is denoted by $ \mathscr{F}_\Psi.$
\end{definition}
\begin{remark}
Fundamental difference of continuous frames with discrete one is that we are not allowed to use orthonormal bases (indexed by $\Omega$).
\end{remark}
 If the condition \text{\upshape(ii)} in Definition \ref{1} is replaced by ``the map  $S_{A,\Psi} : \mathcal{H} \ni h \mapsto \int_{\Omega}\Psi_\alpha^*A_\alpha h  \,d\mu(\alpha)\in \mathcal{H}$ is a well-defined bounded positive  operator (not necessarily invertible)", then we say $ \{A_\alpha\}_{\alpha\in \Omega}$   w.r.t. $ \{\Psi_\alpha\}_{\alpha\in\Omega}$ is Bessel.

We note that (ii) in Definition \ref{1} implies that  there are real $ a,b >0$ such that for all $h \in \mathcal{H}$,

\begin{align*}
a\|h\|^2\leq \langle S_{A,\Psi}h, h\rangle =\left\langle\int_\Omega\Psi^*_\alpha A_\alpha h \,d\mu(\alpha), h\right \rangle   = \int_\Omega\langle \Psi^*_\alpha A_\alpha h, h\rangle  \,d\mu(\alpha) =\int_\Omega\langle A_\alpha h, \Psi_\alpha h\rangle  \,d\mu(\alpha)\leq b\|h\|^2,
\end{align*}
and  (iii) implies there exist $c,d>0$ such that for all $h \in \mathcal{H}$,

\begin{align*}
\|\theta_Ah\|^2=\langle\theta_Ah,\theta_Ah \rangle=\int_\Omega\langle \theta_Ah(\alpha),\theta_Ah(\alpha) \rangle \,d \mu(\alpha) =\int_\Omega\langle A_\alpha h,A_\alpha h\rangle  \,d\mu(\alpha)=\int_\Omega\| A_\alpha h\|^2 \, d\mu(\alpha)\leq c\|h\|^2;\\
\|\theta_\Psi h\|^2=\langle\theta_\Psi h,\theta_\Psi h \rangle=\int_\Omega\langle \theta_\Psi h(\alpha),\theta_\Psi h(\alpha) \rangle  \,d \mu(\alpha) =\int_\Omega\langle \Psi _\alpha h,\Psi _\alpha h\rangle  \,d\mu(\alpha)=\int_\Omega\|\Psi_\alpha h\|^2 \,d\mu(\alpha)\leq d\|h\|^2.
\end{align*}

We note the following. 
\begin{enumerate}[(i)]
\item If  $\{A_\alpha\}_{\alpha \in \Omega} $ is a continuous  (ovf) w.r.t. $ \{\Psi_\alpha\}_{\alpha \in \Omega} $,   then $ \{\Psi_\alpha\}_{\alpha \in \Omega} $ is  a continuous   (ovf) w.r.t. $ \{A_\alpha\}_{\alpha \in \Omega} $. 
\item $\{h \in \mathcal{H}: A_\alpha h=0, \forall \alpha \in \Omega\} =\{0\}= \{h \in \mathcal{H}: \Psi_\alpha h=0, \forall \alpha \in \Omega\} $, and $\overline{\operatorname{span}}\cup_{\alpha\in \Omega} A^*_\alpha(\mathcal{H}_0)=\mathcal{H}= \overline{\operatorname{span}}\cup_{\alpha\in \Omega} \Psi^*_\alpha(\mathcal{H}_0). $
\item $ S_{A, \Psi}=  S_{\Psi, A}$. 
\item If  $ \{A_\alpha\}_{\alpha \in \Omega}, $ $ \{B_\alpha\}_{\alpha \in \Omega}  \in \mathscr{F}_\Psi $, then $\{A_\alpha+B_\alpha\}_{\alpha \in \Omega} \in \mathscr{F}_\Psi $, and $\{\alpha A_\alpha\}_{\alpha \in \Omega} \in \mathscr{F}_\Psi, \forall \alpha >0. $
\item If $(\{A_\alpha\}_{\alpha\in \Omega}, \{\Psi_\alpha\}_{\alpha\in \Omega}) $ is tight continuous (ovf) with bound $ a, $ then $ S_{A,\Psi}=aI_\mathcal{H}.$
\end{enumerate}
 
\begin{proposition}
Let $(\{A_\alpha\}_{\alpha\in \Omega}, \{\Psi_\alpha\}_{\alpha\in \Omega}) $ be  a continuous (ovf)   in $ \mathcal{B}(\mathcal{H}, \mathcal{H}_0)$  with an upper frame  bound $b$. If $\{\alpha\}$ is measurable and $\Psi_\alpha^*A_\alpha\geq 0, \forall \alpha \in \Omega ,$ then $ \mu(\{\alpha\})\|\Psi_\alpha^*A_\alpha\|\leq b, \forall \alpha \in \Omega.$
\end{proposition}
 \begin{proof}
For each $ h \in \mathcal{H}, \alpha \in \Omega$ we get $ \mu(\{\alpha\})\langle\Psi_\alpha^*A_\alpha h,h\rangle=\int_{\{\alpha\}}\langle\Psi_\beta^*A_\beta h,h\rangle\,d\mu(\beta) \leq \int_{\{\alpha\}}\langle\Psi_\beta^*A_\beta h,h\rangle\,d\mu(\beta)+\int_{\Omega\setminus\{\alpha\}}\langle\Psi_\beta^*A_\beta h,h\rangle\,d\mu(\beta)=\int_{\Omega}\langle\Psi_\beta^*A_\beta h,h\rangle\,d\mu(\beta) \leq b \langle h,h \rangle$ and hence $ \mu(\{\alpha\})\|\Psi_\alpha^*A_\alpha\|=\mu(\{\alpha\})\sup_{h\in \mathcal{H},\|h\|=1}$ $\langle\Psi_\alpha^*A_\alpha h,h\rangle \leq b, \forall \alpha \in \Omega.$
 \end{proof}

\begin{proposition}
Let $ (\{A_\alpha\}_{\alpha\in \Omega}, \{\Psi_\alpha\}_{\alpha\in \Omega}) $  be a continuous   (ovf) in    $\mathcal{B}(\mathcal{H}, \mathcal{H}_0)$. Then the bounded left-inverses of 
\begin{enumerate}[\upshape(i)]
\item $ \theta_A$ are precisely  $S_{A,\Psi}^{-1}\theta_\Psi^*+U(I_{\mathcal{L}^2(\Omega,\mathcal{H}_0)}-\theta_AS_{A,\Psi}^{-1}\theta_\Psi^*)$, where $U \in \mathcal{B}(\mathcal{L}^2(\Omega,\mathcal{H}_0), \mathcal{H})$.
\item $ \theta_\Psi$ are precisely  $S_{A,\Psi}^{-1}\theta_A^*+V(I_{\mathcal{L}^2(\Omega,\mathcal{H}_0)}-\theta_\Psi S_{A,\Psi}^{-1}\theta_A^*)$, where $V\in \mathcal{B} (\mathcal{L}^2(\Omega,\mathcal{H}_0), \mathcal{H})$.
\end{enumerate}	
\end{proposition}
\begin{proof}
Similar to the proof of Proposition 2.29 in \cite{MAHESHKRISHNASAMJOHNSON}.	
 \end{proof} 

\begin{proposition}\label{2.2}
For every $ \{A_\alpha\}_{\alpha \in \Omega}  \in \mathscr{F}_\Psi$,
 \begin{enumerate}[ \upshape (i)]
\item $ \theta_A^* \theta_A h= \int_{\Omega}A_\alpha^*A_\alpha h \,d\mu(\alpha)$, $\forall h \in \mathcal{H}$.
\item $ S_{A, \Psi} = \theta_\Psi^*\theta_A=\theta_A^*\theta_\Psi =S_{\Psi,A}.$
\item $(\{A_\alpha\}_{\alpha\in \Omega}, \{\Psi_\alpha\}_{\alpha\in \Omega}) $  is Parseval  if and only if $  \theta_\Psi^*\theta_A =I_\mathcal{H}.$ 
\item $(\{A_\alpha\}_{\alpha\in \Omega}, \{\Psi_\alpha\}_{\alpha\in \Omega}) $  is Parseval  if and only if $ \theta_A\theta_\Psi^* $ is idempotent.
\item $ P_{A, \Psi}\coloneqq\theta_AS_{A,\Psi}^{-1}\theta_\Psi^* $ is idempotent and $ P_{\Psi, A}=P_{A, \Psi}^*$.
\item $\theta_A $ and $ \theta_\Psi$ are injective and  their ranges are closed.
\item  $\theta_A^* $ and $ \theta_\Psi^*$ are surjective.
\end{enumerate}
\end{proposition}
\begin{proof} Let $ h, g \in \mathcal{H}.$ We observe 
$$  \left\langle  \theta_A^* \theta_Ah, g \right\rangle= \left\langle\theta_Ah,\theta_Ag \right\rangle=\int_{\Omega}\langle\theta_Ah(\alpha),\theta_Ag(\alpha) \rangle \,d\mu(\alpha)=\int_{\Omega}\langle A_\alpha h,  A_\alpha g \rangle \,d\mu(\alpha) =\left\langle\int_{\Omega}A_\alpha^*A_\alpha h \,d\mu(\alpha), g \right\rangle ,$$ 
and 
\begin{align*}
\langle S_{A, \Psi}h, g \rangle &= \left\langle \int_{\Omega}\Psi_\alpha^*A_\alpha h \,d\mu(\alpha), g\right \rangle = \int_{\Omega}\langle \Psi_\alpha^*A_\alpha h ,g \rangle \,d\mu(\alpha)\\
&=\int_{\Omega}\langle A_\alpha h ,\Psi_\alpha g \rangle \,d\mu(\alpha) =\langle  \theta_Ah, \theta_\Psi g\rangle=\langle\theta_\Psi^*\theta_Ah ,g \rangle ,
\end{align*}
hence we get (i) and (ii). Arguments for other statements are similar to the proof of Proposition 2.30 in  \cite{MAHESHKRISHNASAMJOHNSON}.
\end{proof}

\begin{proposition}
A collection $ \{A_\alpha\}_{\alpha \in \Omega} $   in $ \mathcal{B}(\mathcal{H}, \mathcal{H}_0)$ is  a continuous  (ovf)  w.r.t. $ \{\Psi_\alpha\}_{\alpha \in \Omega}  $ in $ \mathcal{B}(\mathcal{H}, \mathcal{H}_0) $ if and only if there exist $ a, b, c , d>0$ such that 
\begin{enumerate}[\upshape(i)]
\item for each  $h \in \mathcal{H}$, both  maps   $\Omega \ni \alpha \mapsto A_\alpha h\in \mathcal{H}_0$, $\Omega \ni\alpha \mapsto \Psi_\alpha h\in \mathcal{H}_0$ are measurable,
\item $\int_{\Omega}\Psi_\alpha^*A_\alpha h\,d\mu(\alpha)=\int_{\Omega}A_\alpha^*\Psi_\alpha h\,d\mu(\alpha), \forall h \in \mathcal{H},$
\item $a\|h\|^2\leq\int_{\Omega}\langle A_\alpha h, \Psi_\alpha h\rangle \,d\mu(\alpha) \leq b\|h\|^2, \forall h \in \mathcal{H},$	
\item $\int_{\Omega}\| A_\alpha h\|^2 \,d\mu(\alpha)\leq c\|h\|^2, \forall h \in \mathcal{H}; \int_{\Omega} \|\Psi_\alpha h\|^2  \,d\mu(\alpha)\leq d\|h\|^2, \forall h \in \mathcal{H}$.
\end{enumerate}
\end{proposition}
\begin{definition}\label{RIESZOVF}
 A continuous  (ovf)  $(\{A_\alpha \}_{\alpha\in \Omega}, \{\Psi_\alpha\}_{\alpha\in \Omega})$  in $\mathcal{B}(\mathcal{H}, \mathcal{H}_0)$ is said to be a Riesz continuous  (ovf)  if $ P_{A,\Psi}=  I_{\mathcal{L}^2(\Omega, \mathcal{H}_0)}$. 
 \end{definition}

 \begin{proposition}\label{RIESZOVFCHARACTERIZATIONPROPOSITION}
 A continuous  (ovf) $ (\{A_\alpha\}_{\Omega}, \{\Psi_\alpha\}_{\Omega}) $ in $ \mathcal{B}(\mathcal{H}, \mathcal{H}_0)$ is a Riesz continuous (ovf)  if and only if  $\theta_A(\mathcal{H})=\mathcal{L}^2(\Omega, \mathcal{H}_0)$ if and only if $\theta_\Psi(\mathcal{H})=\mathcal{L}^2(\Omega, \mathcal{H}_0).$
 \end{proposition}
 \begin{proof}
 Similar to the proof of Proposition  2.36 in \cite{MAHESHKRISHNASAMJOHNSON}.
 \end{proof} 
Following is the dilation result in discrete setting (for dilation results in Hilbert spaces we refer Theorem 2.38 in \cite{MAHESHKRISHNASAMJOHNSON} and \cite{KASHINKULIKOVA, HANLARSON, CZAJA, KRIVOSHEIN}).
 \begin{theorem}\label{OPERATORDILATION}\cite{MAHESHKRISHNASAMJOHNSON}
 	Let $(\{A_j\}_{j\in \mathbb{J}},\{\Psi_j\}_{j\in \mathbb{J}} )$ be  a Parseval (ovf) in $ \mathcal{B}(\mathcal{H}, \mathcal{H}_0)$ such that $ \theta_A(\mathcal{H})=\theta_\Psi(\mathcal{H})$ and $ P_{A,\Psi}$ is projection. Then there exist a Hilbert space $ \mathcal{H}_1 $ which contains $ \mathcal{H}$ isometrically and  bounded linear operators $B_j,\Phi_j:\mathcal{H}_1\rightarrow \mathcal{H}_0, \forall j \in \mathbb{J} $ such that $(\{B_j\}_{j\in \mathbb{J}} ,\{\Phi_j\}_{j\in \mathbb{J}})$ is an orthonormal (ovf) in $ \mathcal{B}(\mathcal{H}_1, \mathcal{H}_0)$ and $B_j|_{\mathcal{H}}=  A_j,\Phi_j|_{\mathcal{H}}=\Psi_j, \forall j \in \mathbb{J}$. 
 \end{theorem}
 
 We remark here that we  don't know any  result corresponding to Theorem \ref{OPERATORDILATION} when the indexing set is a measure space. 
 \begin{definition}
 A continuous  (ovf)   $ (\{B_\alpha\}_{\alpha\in \Omega}, \{\Phi_\alpha\}_{\alpha\in \Omega}) $  in $\mathcal{B}(\mathcal{H}, \mathcal{H}_0)$ is said to be a dual of a continuous  (ovf) $(\{A_\alpha\}_{\alpha\in \Omega}, \{\Psi_\alpha\}_{\alpha\in \Omega})$ in $\mathcal{B}(\mathcal{H}, \mathcal{H}_0)$  if $ \theta_\Phi^*\theta_A= \theta_B^*\theta_\Psi=I_{\mathcal{H}}$. The `continuous operator-valued frame' $(\{\widetilde{A}_\alpha\coloneqq A_\alpha S_{A,\Psi}^{-1}\}_{\alpha\in \Omega}, \{\widetilde{\Psi}_\alpha\coloneqq\Psi_\alpha S_{A,\Psi}^{-1}\}_{\alpha \in \Omega} )$, which is a `dual' of $( \{A_\alpha\}_{\alpha\in \Omega}, \{\Psi_\alpha\}_{\alpha\in \Omega} )$ is called as the canonical dual of $( \{A_\alpha\}_{\alpha\in \Omega}, \{\Psi_\alpha\}_{\alpha\in \Omega} )$.
 \end{definition}

 \begin{proposition}
 Let $( \{A_\alpha\}_{\alpha\in \Omega}, \{\Psi_\alpha\}_{\alpha\in \Omega} )$ be a continuous  (ovf) in $ \mathcal{B}(\mathcal{H}, \mathcal{H}_0).$  If $ h \in \mathcal{H}$ has representation  $ h=\int_{\Omega}A_\alpha^*f(\alpha) \,d \mu (\alpha)= \int_{\Omega}\Psi_\alpha^*g(\alpha) \,d \mu (\alpha), $ for some measurable  $ f,g : \Omega \rightarrow \mathcal{H}_0$, then 
 $$ \int_{\Omega}\langle f(\alpha),g(\alpha)\rangle\, d \mu (\alpha)=\int_{\Omega}\langle \widetilde{\Psi}_\alpha h,\widetilde{A}_\alpha h\rangle \,d \mu (\alpha)+\int_{\Omega}\langle f(\alpha)-\widetilde{\Psi}_\alpha h,g(\alpha)-\widetilde{A}_\alpha h\rangle \,d \mu (\alpha). $$
 \end{proposition}
 \begin{proof}
 \begin{align*}
 \text{	Right side } &=\int_{\Omega}\langle \widetilde{\Psi}_\alpha h,\widetilde{A}_\alpha h\rangle \, d \mu (\alpha)+\int_{\Omega}\langle f(\alpha), g(\alpha)\rangle \,d \mu (\alpha) -\int_{\Omega}\langle f(\alpha), \widetilde{A}_\alpha h\rangle\, d \mu (\alpha)\\
 &\quad-\int_{\Omega}\langle \widetilde{\Psi}_\alpha h, g(\alpha)\rangle \, d \mu (\alpha)+\int_{\Omega}\langle \widetilde{\Psi}_\alpha h,\widetilde{A}_\alpha h\rangle \,d \mu (\alpha)\\
 &=2\int_{\Omega}\langle \widetilde{\Psi}_\alpha h,\widetilde{A}_\alpha h\rangle \,d \mu (\alpha)+ \int_{\Omega}\langle f(\alpha) , g(\alpha)\rangle \,d \mu (\alpha)-\int_{\Omega}\langle f(\alpha),A_\alpha S_{A,\Psi}^{-1}h\rangle \,d \mu (\alpha)\\
 &\quad-\int_{\Omega}\langle \Psi_\alpha S_{A,\Psi}^{-1}h, g(\alpha)\rangle \,d \mu (\alpha)\\
 &= 2\left\langle\int_{ \Omega}S_{A,\Psi}^{-1}A_\alpha^*\Psi_\alpha S_{A,\Psi}^{-1}h \,d \mu (\alpha), h \right\rangle+ \int_{\Omega}\langle f(\alpha), g(\alpha)\rangle \,d \mu (\alpha)\\
 &\quad-\left\langle \int_{\Omega}A_\alpha^*f(\alpha) \,d \mu (\alpha),S_{A,\Psi}^{-1}h\right \rangle -\left\langle S_{A,\Psi}^{-1}h , \int_{\Omega}\Psi_\alpha^*g(\alpha) \,d \mu (\alpha)\right \rangle\\
 &=2 \langle S_{A,\Psi}^{-1}h,h \rangle + \int_{ \Omega}\langle f(\alpha), g(\alpha)\rangle \,d \mu (\alpha) -\langle h, S_{A,\Psi}^{-1}h\rangle-\langle S_{A,\Psi}^{-1}h, h\rangle=\text{Left side.}
 \end{align*}
 \end{proof} 
\begin{theorem}\label{CANONICALDUALFRAMEPROPERTYOPERATORVERSION}
 Let $(\{A_\alpha \}_{\alpha \in \Omega},\{\Psi_\alpha \}_{\alpha \in \Omega})$ be a continuous  (ovf) with frame bounds $ a$ and $ b.$ Then the following statements are true.
 \begin{enumerate}[\upshape(i)]
\item The canonical dual (ovf) of the canonical dual (ovf)  of $(\{A_\alpha \}_{\alpha \in \Omega},\{\Psi_\alpha \}_{\alpha \in \Omega})$ is itself.
\item$ \frac{1}{b}, \frac{1}{a}$ are frame bounds for the canonical dual of $(\{A_\alpha \}_{\alpha \in \Omega},\{\Psi_\alpha \}_{\alpha \in \Omega})$.
\item If $ a, b $ are optimal frame bounds for $(\{A_\alpha \}_{\alpha \in \Omega},\{\Psi_\alpha \}_{\alpha \in \Omega})$, then $ \frac{1}{b}, \frac{1}{a}$ are optimal  frame bounds for its canonical dual.
\end{enumerate} 
\end{theorem} 
\begin{proof}
We note that

\begin{align*}
 \left \langle \int_{\Omega} \widetilde{\Psi}^*_\alpha \widetilde{A}_\alpha h \,d \mu (\alpha), g\right \rangle =\int_{\Omega}\langle S_{A, \Psi}^{-1} \Psi_\alpha^*A_\alpha S_{A, \Psi}^{-1}h, g\rangle \,d \mu (\alpha)=\int_{\Omega}\langle \Psi_\alpha^*A_\alpha S_{A, \Psi}^{-1}h,S_{A, \Psi}^{-1}g \rangle \,d \mu (\alpha)\\
 = \left \langle \int_{\Omega} \Psi_\alpha^*A_\alpha( S_{A, \Psi}^{-1}h)  \,d \mu (\alpha), S_{A, \Psi}^{-1}g\right \rangle =\langle S_{A, \Psi}(S_{A, \Psi}^{-1}h), S_{A, \Psi}^{-1}g\rangle =\langle S_{A, \Psi}^{-1}h, g\rangle ,\quad \forall h, g \in \mathcal{H}.
\end{align*}
Therefore the frame operator for the canonical dual  $(\{A_\alpha S_{A,\Psi}^{-1}\}_{\alpha \in \Omega}, \{\Psi_\alpha S_{A,\Psi}^{-1}\}_{\alpha\in \Omega} )$ is $ S_{A,\Psi}$. Remainings are similar to the proof of Theorem 2.41 in \cite{MAHESHKRISHNASAMJOHNSON}.
\end{proof}
\begin{proposition}\label{DUALOVFCHARACTERIZATION}
Let  $ (\{A_\alpha\}_{\alpha\in \Omega}, \{\Psi_\alpha\}_{\alpha\in \Omega}) $ and $ (\{B_\alpha\}_{\alpha\in \Omega}, \{\Phi_\alpha\}_{\alpha\in \Omega}) $ be continuous operator-valued frames in  $\mathcal{B}(\mathcal{H}, \mathcal{H}_0)$. Then the following are equivalent.
 \begin{enumerate}[\upshape(i)]
 \item   $ (\{B_\alpha\}_{\alpha\in \Omega}, \{\Phi_\alpha\}_{\alpha\in \Omega}) $ is a dual of $ (\{A_\alpha\}_{\alpha\in \Omega}, \{\Psi_\alpha\}_{\alpha\in \Omega}) $. 
 \item $ \int_{\Omega}\Phi_\alpha^*A_\alpha h \,d \mu(\alpha)= \int_{\Omega }B_\alpha^*\Psi_\alpha h\,d \mu(\alpha)=h, \forall h \in \mathcal{H}$.
 \end{enumerate}
 \end{proposition}
 \begin{proof}
 $\langle \theta_\Phi^*\theta_Ah, g \rangle = \langle \theta_Ah, \theta_\Phi g \rangle = \int_{\Omega }\langle \theta_Ah(\alpha), \theta_\Phi g (\alpha)\rangle \,d \mu(\alpha)= \int_{\Omega }\langle A_\alpha h, \Phi_ \alpha g\rangle \,d \mu(\alpha) =\langle \int_{\Omega}\Phi_\alpha^*A_\alpha h \,d \mu(\alpha), g\rangle ,$ $\forall h,g \in  \mathcal{H}$. Similarly $\langle \theta_B^*\theta_\Psi h, g \rangle=\langle\int_{\Omega }B_\alpha^*\Psi_\alpha h\,d \mu(\alpha) , g \rangle , $ $\forall h,g \in  \mathcal{H}$.
 \end{proof}
 \begin{theorem}
  If $ (\{A_\alpha\}_{\alpha\in \Omega}, \{\Psi_\alpha\}_{\alpha\in \Omega}) $  is a Riesz continuous  (ovf)  in  $\mathcal{B}(\mathcal{H}, \mathcal{H}_0)$, then it  has unique dual. 
 \end{theorem}
 \begin{proof}
 Let  $ (\{B_\alpha\}_{\alpha\in \Omega}, \{\Phi_\alpha\}_{\alpha\in \Omega})$ and  $ (\{C_\alpha\}_{\alpha\in \Omega}, \{\Xi_\alpha\}_{\alpha\in \Omega})$ be  continuous operator-valued frames such that  both are duals  of  $ (\{A_\alpha\}_{\alpha\in \Omega}, \{\Psi_\alpha\}_{\alpha\in \Omega}) $. Then $ \theta_\Psi^*\theta_B=I_\mathcal{H}=\theta_A^*\theta_\Phi=\theta_\Psi^*\theta_C=\theta_A^*\theta_\Xi$ $\Rightarrow $ $\theta_\Psi^*(\theta_B-\theta_C)=0=\theta_A^*(\theta_\Phi-\theta_\Xi)$ $\Rightarrow $ $ \theta_AS_{A,\Psi}^{-1}\theta_\Psi^*(\theta_B-\theta_C)=P_{A,\Psi}(\theta_B-\theta_C)=I_\mathcal{H}(\theta_B-\theta_C)=0=\theta_\Psi S_{\Psi, A}^{-1}\theta_A^*(\theta_\Phi-\theta_\Xi)=P_{\Psi,A}(\theta_\Phi-\theta_\Xi)=I_\mathcal{H}(\theta_\Phi-\theta_\Xi)$ $\Rightarrow $ $\theta_B=\theta_C, \theta_\Phi=\theta_\Xi$ $\Rightarrow $ $B_\alpha h=\theta_Bh(\alpha)=\theta_Ch(\alpha)=C_\alpha h, \Phi_\alpha h=\theta_\Phi h(\alpha)=\theta_\Xi h(\alpha)=\Xi_\alpha h$, $\forall h \in \mathcal{H}, \forall \alpha \in \Omega$.
 \end{proof}
 \begin{proposition}
Let $(\{A_\alpha\}_{\alpha\in \Omega}, \{\Psi_\alpha\}_{\alpha\in \Omega}) $  be a continuous  (ovf) in    $\mathcal{B}(\mathcal{H}, \mathcal{H}_0)$. If  $ (\{B_\alpha\}_{\alpha\in \Omega}, \{\Phi_\alpha\}_{\alpha\in \Omega})$ is dual of $(\{A_\alpha\}_{\alpha\in \Omega}, \{\Psi_\alpha\}_{\alpha\in \Omega}) $, then there exist continuous  Bessel  $ \{C_\alpha\}_{\alpha\in \Omega}$ and $ \{\Xi_\alpha\}_{\alpha\in \Omega} $ (w.r.t. themselves) in  $\mathcal{B}(\mathcal{H}, \mathcal{H}_0)$ such that $ B_\alpha=A_\alpha S_{A,\Psi}^{-1}+C_\alpha, \Phi_\alpha=\Psi_\alpha S_{A,\Psi}^{-1}+\Xi_\alpha,\forall \alpha \in \Omega$,  and $\theta_C(\mathcal{H})\perp \theta_\Psi(\mathcal{H}),\theta_\Xi(\mathcal{H})\perp \theta_A(\mathcal{H})$. Converse holds if $\theta_\Xi^*\theta_C \geq 0$.
 \end{proposition}
 \begin{proof}
 Similar to the proof of Proposition 2.44 in \cite{MAHESHKRISHNASAMJOHNSON}.
 \end{proof}
 \begin{definition}
 A continuous  (ovf)  $ (\{B_\alpha\}_{\alpha\in \Omega}, \{\Phi_\alpha\}_{\alpha\in \Omega}) $ in $\mathcal{B}(\mathcal{H}, \mathcal{H}_0)$ is said to be orthogonal to a continuous  (ovf)  $ (\{A_\alpha\}_{\alpha\in \Omega}, \{\Psi_\alpha\}_{\alpha\in \Omega}) $ in $\mathcal{B}(\mathcal{H}, \mathcal{H}_0)$ if $ \theta_\Phi^*\theta_A= \theta_B^*\theta_\Psi=0.$
 \end{definition}
 
 \begin{proposition}\label{ORTHOGONALOVFCHARACTERIZATION}
 Let  $ (\{A_\alpha\}_{\alpha\in \Omega}, \{\Psi_\alpha\}_{\alpha\in \Omega}) $ and $ (\{B_\alpha\}_{\alpha\in \Omega}, \{\Phi_\alpha\}_{\alpha\in \Omega}) $ be continuous  operator-valued frames in  $\mathcal{B}(\mathcal{H}, \mathcal{H}_0)$. Then the following are equivalent.
 \begin{enumerate}[\upshape(i)]
 \item $ (\{B_\alpha\}_{\alpha\in \Omega}, \{\Phi_\alpha\}_{\alpha\in \Omega}) $ is orthogonal to  $ (\{A_\alpha\}_{\alpha\in \Omega}, \{\Psi_\alpha\}_{\alpha\in \Omega}) $. 
 \item  $ \int_{\Omega}\Phi_\alpha^*A_\alpha h \,d \mu(\alpha)= \int_{\Omega }B_\alpha^*\Psi_\alpha h\,d \mu(\alpha)=0, \forall h \in \mathcal{H}$.
 \end{enumerate}
\end{proposition}

\begin{proposition}
Two orthogonal continuous operator-valued frames  have common dual continuous (ovf).	
\end{proposition} 
\begin{proof}
Similar to the proof of Proposition 2.48 in \cite{MAHESHKRISHNASAMJOHNSON}.
\end{proof}
 \begin{proposition}
Let $ (\{A_\alpha\}_{\alpha\in \Omega}, \{\Psi_\alpha\}_{\alpha\in \Omega}) $ and $ (\{B_\alpha\}_{\alpha\in \Omega}, \{\Phi_\alpha\}_{\alpha\in \Omega}) $ be  two Parseval continuous  operator-valued frames in   $\mathcal{B}(\mathcal{H}, \mathcal{H}_0)$ which are  orthogonal. If $C,D,E,F \in \mathcal{B}(\mathcal{H})$ are such that $ C^*E+D^*F=I_\mathcal{H}$, then  $ (\{A_\alpha C+B_\alpha D\}_{\alpha\in \Omega}, \{\Psi_\alpha E+\Phi_\alpha F\}_{\alpha\in \Omega}) $ is a  Parseval continuous  (ovf) in  $\mathcal{B}(\mathcal{H}, \mathcal{H}_0)$. In particular,  if scalars $ c,d,e,f$ satisfy $\bar{c}e+\bar{d}f =1$, then $ (\{cA_\alpha+dB_\alpha\}_{\alpha\in \Omega}, \{e\Psi_\alpha+f\Phi_\alpha\}_{\alpha\in \Omega}) $ is  a  Parseval continuous  (ovf).
\end{proposition}   
\begin{proof}
For all $h \in \mathcal{H}$ and $\alpha \in \Omega$  we see $ \theta_{AC+BD} h (\alpha) = (A_\alpha C+B_\alpha D)h=A_\alpha (Ch)+B_\alpha (Dh)=\theta_{A} (Ch) (\alpha)+\theta_{B} (Dh) (\alpha)=(\theta_{A} (Ch) +\theta_{B} (Dh)) (\alpha)=(\theta_{A} C +\theta_{B} D) h (\alpha)$ $\Rightarrow$ $\theta_{AC+BD}=\theta_AC+\theta_BD$. Similarly $ \theta_{\Psi E+\Phi F}=\theta_\Psi E+\theta_\Phi F $. Other arguments are similar to that in the proof of Proposition 2.49 in \cite{MAHESHKRISHNASAMJOHNSON}.
\end{proof}
\begin{definition}
Two continuous operator-valued frames $(\{A_\alpha\}_{\alpha\in \Omega},\{\Psi_\alpha\}_{\alpha\in \Omega} )$  and $ (\{B_\alpha\}_{\alpha\in \Omega}, \{\Phi_\alpha\}_{\alpha\in \Omega} )$   in $ \mathcal{B}(\mathcal{H}, \mathcal{H}_0)$  are called 
disjoint if $(\{A_\alpha\oplus B_\alpha\}_{\alpha \in \Omega},\{\Psi_\alpha\oplus \Phi_\alpha\}_{\alpha \in \Omega})$ is continuous  (ovf) in $ \mathcal{B}(\mathcal{H}\oplus \mathcal{H}, \mathcal{H}_0).$   
\end{definition}
\begin{proposition}
If $(\{A_\alpha\}_{\alpha\in \Omega},\{\Psi_\alpha\}_{\alpha\in \Omega} )$  and $ (\{B_\alpha\}_{\alpha\in \Omega}, \{\Phi_\alpha\}_{\alpha\in \Omega} )$  are  orthogonal continuous operator-valued frames  in $ \mathcal{B}(\mathcal{H}, \mathcal{H}_0)$, then  they  are disjoint. Further, if both $(\{A_\alpha\}_{\alpha\in \Omega},\{\Psi_\alpha\}_{\alpha\in \Omega} )$  and $ (\{B_\alpha\}_{\alpha\in \Omega}$, $ \{\Phi_\alpha\}_{\alpha\in \Omega} )$ are  Parseval, then $(\{A_\alpha\oplus B_\alpha\}_{\alpha \in \Omega},\{\Psi_\alpha\oplus \Phi_\alpha\}_{\alpha \in \Omega})$ is Parseval.
\end{proposition}
\begin{proof}
For all $h\oplus g \in \mathcal{H}\oplus\mathcal{H} $, $\theta_{A\oplus B}(h\oplus g)(\alpha)=(A_\alpha\oplus B_\alpha)(h\oplus g)=A_\alpha h+ B_\alpha g=\theta_Ah(\alpha)+\theta_Bg(\alpha)=(\theta_Ah+\theta_Bg)(\alpha)$ and for all $f \in \mathcal{L}^2(\Omega, \mathcal{H}_0)$, $\langle \theta_{\Psi\oplus \Phi}^*f, h\oplus g\rangle =\langle f,\theta_{\Psi\oplus \Phi}(h\oplus g)\rangle=\langle \theta_\Psi^*f, h \rangle+\langle\theta_\Phi^*f,  g \rangle=\langle \theta_\Psi^*f\oplus\theta_\Phi^*f, h\oplus g\rangle .$ Thus  $S_{A\oplus B,\Psi\oplus\Phi}(h\oplus g)=\theta^*_{\Psi\oplus\Phi}\theta_{A\oplus B}(h\oplus g)=\theta^*_{\Psi\oplus\Phi}(\theta_Ah+\theta_Bg)=\theta_\Psi^*(\theta_Ah+\theta_Bg)\oplus\theta_\Phi^*(\theta_Ah+\theta_Bg)=(S_{A ,\Psi}+0)\oplus (0+S_{ B,\Phi})=S_{A ,\Psi}\oplus S_{ B,\Phi}$, which is bounded positive invertible with $S_{A\oplus B,\Psi\oplus\Phi}^{-1}=S_{A ,\Psi}^{-1}\oplus S_{ B,\Phi}^{-1}$. 
\end{proof}

\section{Characterizations of the extension}\label{CHARACTERIZATIONSOF THE EXTENSION}

\begin{theorem}\label{OPERATORCHARACTERIZATIONHILBERT2}
Let $\{A_\alpha\}_{\alpha\in\Omega},\{\Psi_\alpha\}_{\alpha\in\Omega}$ be in $ \mathcal{B}(\mathcal{H}, \mathcal{H}_0)$ such that 
 for each  $h \in \mathcal{H}$, both  maps   $\Omega \ni \alpha \mapsto A_\alpha h\in \mathcal{H}_0$, $\Omega \ni\alpha \mapsto \Psi_\alpha h\in \mathcal{H}_0$ are measurable. Then  $ (\{A_\alpha\}_{\alpha\in \Omega}, \{\Psi_\alpha\}_{\alpha\in \Omega})$  is a continuous (ovf) with bounds  $a $ and  $ b$ (resp. continuous Bessel with bound $ b$)
\begin{enumerate}[\upshape(i)]
\item   if and only if 
$$U:\mathcal{L}^2(\Omega, \mathcal{H}_0) \ni f\mapsto\int_{\Omega}A_\alpha^*f(\alpha)\,d \mu(\alpha) \in \mathcal{H}, ~\text{and} ~ V:\mathcal{L}^2(\Omega, \mathcal{H}_0) \ni g\mapsto \int_{\Omega}\Psi_\alpha^*g(\alpha)\,d \mu(\alpha) \in \mathcal{H} $$ 
are well-defined, $ U,V \in \mathcal{B}(\mathcal{L}^2(\Omega, \mathcal{H}_0),\mathcal{H})$  such that  $ aI_\mathcal{H}\leq VU^*\leq bI_\mathcal{H}$ (resp. $ 0\leq VU^*\leq bI_\mathcal{H}$).
\item    if and only if 
$$U:\mathcal{L}^2(\Omega, \mathcal{H}_0) \ni f\mapsto\int_{\Omega}A_\alpha^*f(\alpha)\,d \mu(\alpha) \in \mathcal{H}, ~\text{and} ~ S: \mathcal{H} \ni x\mapsto Sx \in \mathcal{L}^2(\Omega, \mathcal{H}_0), Sx: \Omega \ni \alpha \mapsto \Psi_\alpha x \in \mathcal{H}_0 $$ 
are well-defined, $ U \in \mathcal{B}(\mathcal{L}^2(\Omega, \mathcal{H}_0),\mathcal{H})$, $ S \in \mathcal{B}(\mathcal{H}, \mathcal{L}^2(\Omega, \mathcal{H}_0))$ such that  $ aI_\mathcal{H}\leq S^*U^*\leq bI_\mathcal{H}$ (resp. $ 0\leq S^*U^*\leq bI_\mathcal{H}$).
\item  if and only if  
$$R:   \mathcal{H} \ni h\mapsto  Rh\in  \mathcal{L}^2(\Omega, \mathcal{H}_0), Rh: \Omega \ni \alpha \mapsto A_\alpha h \in \mathcal{H}_0, ~\text{and} ~ V: \mathcal{L}^2(\Omega, \mathcal{H}_0) \ni g\mapsto \int_{\Omega}\Psi_\alpha^*g(\alpha)\,d \mu(\alpha)\in \mathcal{H} $$ 
are well-defined, $ R \in \mathcal{B}(\mathcal{H}, \mathcal{L}^2(\Omega, \mathcal{H}_0))$, $ V \in \mathcal{B}(\mathcal{L}^2(\Omega, \mathcal{H}_0),\mathcal{H})$ such that  $ aI_\mathcal{H}\leq VR\leq bI_\mathcal{H}$ (resp. $ 0\leq VR\leq bI_\mathcal{H}$).
\item  if and only if  $$ \mathcal{H} \ni h\mapsto  Rh\in  \mathcal{L}^2(\Omega, \mathcal{H}_0), Rh: \Omega \ni \alpha \mapsto A_\alpha h \in \mathcal{H}_0, ~\text{and} ~ S: \mathcal{H} \ni x\mapsto Sx \in \mathcal{L}^2(\Omega, \mathcal{H}_0), Sx: \Omega \ni \alpha \mapsto \Psi_\alpha x \in \mathcal{H}_0 $$
are well-defined, $ R,S \in \mathcal{B}(\mathcal{H},\mathcal{L}^2(\Omega, \mathcal{H}_0))$  such that  $ aI_\mathcal{H}\leq S^*R\leq bI_\mathcal{H}$ (resp. $ 0\leq S^*R\leq bI_\mathcal{H}$). 
\end{enumerate}

\end{theorem} 
\begin{proof}
We argue only for (i), in frame situation.  $(\Rightarrow)$ Now $U=\theta_A^*$, $V=\theta_\Psi^*$ and $VU^*=\theta_\Psi^*\theta_A=S_{A,\Psi}$.

$(\Leftarrow)$ Now $\theta_A=U^*$, $\theta_\Psi=V^*$ and $S_{A,\Psi}=\theta_\Psi^*\theta_A=VU^*$.
\end{proof}

Let $ \{A_\alpha\}_{\alpha\in\Omega}, \{\Psi_\alpha\}_{\alpha\in\Omega}$ be in $ \mathcal{B}(\mathcal{H},\mathcal{H}_0).$ For each fixed $\alpha \in \Omega$, suppose $ \{e_{\alpha,\beta}\}_{\beta\in\Omega_\alpha}$ is an   orthonormal basis for $ \mathcal{H}_0.$ From Riesz representation theorem,  we get  unique $u_{\alpha,\beta},v_{\alpha,\beta}  \in \mathcal{H}$ such that  $ \langle A_\alpha h, e_{\alpha,\beta}\rangle=\langle h, u_{\alpha,\beta}\rangle, \langle \Psi_\alpha h, e_{\alpha,\beta}\rangle=\langle h, v_{\alpha,\beta}\rangle,\forall h \in \mathcal{H}, \forall \beta \in \Omega_\alpha , \forall \alpha \in \Omega $. Now $A_\alpha h=\sum_{\beta \in \Omega_\alpha}\langle  A_\alpha h,e_{\alpha ,\beta}\rangle e_{\alpha ,\beta}=\sum_{\beta \in\Omega_\alpha}\langle  h,u_{\alpha ,\beta}\rangle e_{\alpha ,\beta}$, $ \Psi_\alpha h= \sum_{\beta \in\Omega_\alpha}\langle  h,v_{\alpha ,\beta}\rangle e_{\alpha ,\beta}, \forall h \in \mathcal{H},\forall \alpha \in \Omega.$ We next  find the adjoints of $ A_\alpha$'s and $\Psi_\alpha $'s in terms of $\{u_{\alpha,\beta}\}_{\beta\in \Omega_\alpha} $ and $\{v_{\alpha,\beta}\}_{\beta\in \Omega_\alpha} $. For all $h \in \mathcal{H}$, $ \langle h, A_\alpha^*y\rangle=\langle A_\alpha h, y\rangle =\sum_{\beta\in \Omega_\alpha}\langle  h,u_{\alpha,\beta}\rangle \langle e_{\alpha,\beta}, y \rangle$ $=\langle h , \sum_{\beta\in \Omega_\alpha} \langle y,e_{\alpha,\beta} \rangle u_{\alpha,\beta}\rangle,
\langle h, \Psi_\alpha ^*y\rangle=\langle \Psi_\alpha h, y\rangle =\sum_{\beta\in \Omega_\alpha}\langle  h,v_{\alpha,\beta}\rangle \langle e_{\alpha,\beta},y  \rangle=\langle h , \sum_{\beta\in \Omega_\alpha} \langle y,e_{\alpha,\beta} \rangle v_{\alpha,\beta}\rangle ,\forall y\in \mathcal{H}_0.$ Therefore $ A_\alpha^*y=\sum_{\beta\in \Omega_\alpha} \langle y,e_{\alpha,\beta} \rangle u_{\alpha,\beta}$, $ \Psi_\alpha^*z=\sum_{\beta\in \Omega_\alpha} \langle z,e_{\alpha,\beta} \rangle v_{\alpha,\beta},  \forall y,z \in \mathcal{H}_0, \forall  \alpha\in \Omega.$ Evaluation of these at $e_{\alpha,\beta_0}$ gives $ u_{\alpha,\beta_0}=A_\alpha^*e_{\alpha,\beta_0}, v_{\alpha,\beta_0}=\Psi_\alpha^*e_{\alpha,\beta_0}, \forall \beta_0 \in  \Omega_\alpha, \alpha\in \Omega.$ 
\begin{theorem}\label{SEQUENTIAL CHARACTERIZATION}
Let $ \{A_\alpha\}_{\alpha\in\Omega}, \{\Psi_\alpha\}_{\alpha\in\Omega}$ be in $ \mathcal{B}(\mathcal{H},\mathcal{H}_0).$ Suppose $ \{e_{\alpha, \beta}\}_{\beta\in\Omega_\alpha }$ is an   orthonormal basis for $ \mathcal{H}_0,$ for each $\alpha \in \Omega.$ Let  $ u_{\alpha,\beta}=A_\alpha^*e_{\alpha,\beta}, v_{\alpha,\beta}=\Psi_\alpha^*e_{\alpha,\beta}, \forall \beta \in  \Omega_\alpha, \forall \alpha\in \Omega.$ Then $(\{A_\alpha\}_{\alpha\in\Omega}, \{\Psi_\alpha\}_{\alpha\in\Omega})$ is  a continuous 
\begin{enumerate}[\upshape(i)]
\item    (ovf) in $ \mathcal{B}(\mathcal{H},\mathcal{H}_0)$  with bounds $a $ and $ b$  if and only if  for each  $h \in \mathcal{H}$, both  maps   $\Omega \ni \alpha \mapsto  \sum_{\beta \in\Omega_\alpha}\langle  h,u_{\alpha ,\beta}\rangle e_{\alpha ,\beta}\in \mathcal{H}_0$, $\Omega \ni\alpha \mapsto \sum_{\beta \in\Omega_\alpha}\langle  h,v_{\alpha ,\beta}\rangle e_{\alpha ,\beta}\in \mathcal{H}_0$ are measurable  and there exist $ c,d >0$ such that the map 
$$ T: \mathcal{H} \ni h \mapsto\int_{\Omega}\sum_{\beta \in \Omega_\alpha}\langle h, u_{\alpha,\beta}\rangle v_{\alpha,\beta} \,d\mu(\alpha)\in  \mathcal{H} $$
is a well-defined bounded positive invertible operator such that $ a\|h\|^2 \leq \langle Th,h \rangle \leq b\|h\|^2, \forall h \in \mathcal{H} $, and 
$$  \int_{\Omega}\sum_{\beta \in \Omega_\alpha}|\langle h, u_{\alpha,\beta}\rangle |^2 \,d\mu(\alpha)\leq c\|h\|^2 ,~ \forall h \in \mathcal{H}; \quad \int_{\Omega}\sum_{\beta \in \Omega_\alpha} |\langle h, v_{\alpha,\beta}\rangle |^2\,d\mu(\alpha)\leq d\|h\|^2 ,~ \forall h \in \mathcal{H}.$$
\item  Bessel  in $ \mathcal{B}(\mathcal{H},\mathcal{H}_0)$  with bound  $ b$  if and only if for each  $h \in \mathcal{H}$, both  maps   $\Omega \ni \alpha \mapsto  \sum_{\beta \in\Omega_\alpha}\langle  h,u_{\alpha ,\beta}\rangle e_{\alpha ,\beta}$ $\in \mathcal{H}_0$, $\Omega \ni\alpha \mapsto \sum_{\beta \in\Omega_\alpha}\langle  h,v_{\alpha ,\beta}\rangle e_{\alpha ,\beta}\in \mathcal{H}_0$ are measurable  and there exist $ c,d >0$ such that the map 
$$ T: \mathcal{H} \ni h \mapsto\int_{\Omega}\sum_{\beta \in\Omega_\alpha}\langle h, u_{\alpha,\beta}\rangle v_{\alpha,\beta} \,d\mu(\alpha)\in  \mathcal{H} $$
is a well-defined bounded positive  operator such that $ 0 \leq \langle Th,h \rangle \leq b\|h\|^2, \forall h \in \mathcal{H} $, and 
$$  \int_{\Omega}\sum_{\beta \in\Omega_\alpha}|\langle h, u_{\alpha,\beta}\rangle |^2\,d\mu(\alpha)\leq c\|h\|^2 ,~ \forall h \in \mathcal{H}; \quad  \int_{\Omega}\sum_{\beta \in\Omega_\alpha} |\langle h, v_{\alpha,\beta}\rangle |^2\,d\mu(\alpha)\leq d\|h\|^2 ,~ \forall h \in \mathcal{H}.$$ 
\item   (ovf)  in $ \mathcal{B}(\mathcal{H},\mathcal{H}_0)$  with bounds $a $ and $ b$  if and only if for each  $h \in \mathcal{H}$, both  maps   $\Omega \ni \alpha \mapsto  \sum_{\beta \in\Omega_\alpha}\langle  h,u_{\alpha ,\beta}\rangle e_{\alpha ,\beta}\in \mathcal{H}_0$, $\Omega \ni\alpha \mapsto \sum_{\beta \in\Omega_\alpha}\langle  h,v_{\alpha ,\beta}\rangle e_{\alpha ,\beta}\in \mathcal{H}_0$ are measurable  and there exist $ c,d, r >0$ such that 
 $$\left \|\int_{\Omega}\sum_{\beta \in\Omega_\alpha}\langle h, u_{\alpha,\beta}\rangle v_{\alpha,\beta}\,d\mu(\alpha)\right\|\leq r\|h\|,~\forall h \in \mathcal{H}   ;$$ 
 $$\int_{\Omega}\sum_{\beta \in\Omega_\alpha}\langle h, u_{\alpha,\beta}\rangle v_{\alpha,\beta} \,d\mu(\alpha)=\int_{\Omega}\sum_{\beta \in\Omega_\alpha}\langle h, v_{\alpha,\beta}\rangle u_{\alpha,\beta}\,d\mu(\alpha) ,~\forall h \in \mathcal{H} ;$$
 $$a\|h\|^2\leq \int_{\Omega}\sum_{\beta \in\Omega_\alpha}\langle h, u_{\alpha,\beta}\rangle \langle  v_{\alpha,\beta} , h\rangle \,d\mu(\alpha) \leq b\|h\|^2 ,~ \forall h \in \mathcal{H}, ~\text{and} $$
 $$  \int_{\Omega}\sum_{\beta \in\Omega_\alpha}|\langle h, u_{\alpha,\beta}\rangle |^2 \,d\mu(\alpha)\leq c\|h\|^2 ,~ \forall h \in \mathcal{H}; \quad \int_{\Omega}\sum_{\beta \in\Omega_\alpha} |\langle h, v_{\alpha,\beta}\rangle |^2\,d\mu(\alpha)\leq d\|h\|^2 ,~ \forall h \in \mathcal{H}.$$
\item Bessel  in $ \mathcal{B}(\mathcal{H},\mathcal{H}_0)$  with bound  $ b$ if and only  if for each  $h \in \mathcal{H}$, both  maps   $\Omega \ni \alpha \mapsto  \sum_{\beta \in\Omega_\alpha}\langle  h,u_{\alpha ,\beta}\rangle e_{\alpha ,\beta}$ $\in \mathcal{H}_0$, $\Omega \ni\alpha \mapsto \sum_{\beta \in\Omega_\alpha}\langle  h,v_{\alpha ,\beta}\rangle e_{\alpha ,\beta}$ $\in \mathcal{H}_0$ are measurable  and there exist $ c,d, r >0$ such that 
$$\left \|\int_{\Omega}\sum_{\beta \in\Omega_\alpha}\langle h, u_{\alpha,\beta}\rangle v_{\alpha,\beta}\,d\mu(\alpha)\right\|\leq r\|h\|,~\forall h \in \mathcal{H} ;$$
$$ \int_{\Omega}\sum_{\beta \in\Omega_\alpha}\langle h, u_{\alpha,\beta}\rangle v_{\alpha,\beta}\,d\mu(\alpha) =\int_{\Omega}\sum_{\beta \in\Omega_\alpha}\langle h, v_{\alpha,\beta}\rangle u_{\alpha,\beta}\,d\mu(\alpha) ,~\forall h \in \mathcal{H} ;$$
 $$ 0 \leq \int_{\Omega}\sum_{\beta \in\Omega_\alpha}\langle h, u_{\alpha,\beta}\rangle \langle  v_{\alpha,\beta} , h\rangle\,d\mu(\alpha) \leq b\|h\|^2 ,~ \forall h \in \mathcal{H}, ~\text{and} $$
 $$  \int_{\Omega}\sum_{\beta \in\Omega_\alpha}|\langle h, u_{\alpha,\beta}\rangle |^2 \,d\mu(\alpha)\leq c\|h\|^2 ,~ \forall h \in \mathcal{H}; \quad  \int_{\Omega}\sum_{\beta \in\Omega_\alpha} |\langle h, v_{\alpha,\beta}\rangle|^2\,d\mu(\alpha)\leq d\|h\|^2 ,~ \forall h \in \mathcal{H}.$$ 
\end{enumerate}
\end{theorem}
\begin{proof}
\begin{enumerate}[\upshape(i)]
\item For all $ h\in \mathcal{H},$

\begin{align*}
\int_{\Omega}\sum\limits_{\beta \in\Omega_\alpha}\langle h, u_{\alpha,\beta}\rangle  v_{\alpha,\beta}\,d\mu(\alpha)&= \int_{\Omega}\sum\limits_{\beta \in\Omega_\alpha}\langle A_\alpha h, e_{\alpha,\beta}\rangle  \Psi_\alpha^*e_{\alpha,\beta} \,d\mu(\alpha)\\
&=\int_{\Omega}\Psi_\alpha^*\left(\sum\limits_{\beta \in\Omega_\alpha}\langle A_\alpha h, e_{\alpha,\beta}\rangle  e_{\alpha,\beta}\right) \,d\mu(\alpha)=\int_{\Omega}\Psi_\alpha^*A_\alpha h\,d\mu(\alpha),
\end{align*}

\begin{align*}
\int_{\Omega}\langle A_\alpha h, \Psi_\alpha h\rangle\,d\mu(\alpha) &=\int_{\Omega} \left \langle\sum\limits_{\beta \in\Omega_\alpha}\langle h, A_\alpha^*e_{\alpha,\beta} \rangle e_{\alpha,\beta}, \sum\limits_{\gamma \in\Omega_\alpha} \langle h, \Psi_\alpha^*e_{\alpha,\gamma}\rangle e_{\alpha,\gamma} \right\rangle\,d\mu(\alpha)\\
&=\int_{\Omega}\left \langle\sum\limits_{\beta \in\Omega_\alpha}\langle h, u_{\alpha,\beta}\rangle e_{\alpha,\beta}, \sum\limits_{\gamma \in \Omega_\alpha} \langle h, v_{\alpha,\gamma}\rangle e_{\alpha,\gamma} \right\rangle\,d\mu(\alpha)\\
&=\int_{\Omega}\sum\limits_{\beta \in\Omega_\alpha}\langle h, u_{\alpha,\beta}\rangle \langle v_{\alpha,\beta}, h\rangle\,d\mu(\alpha), 
\end{align*}

\begin{align*}
\left\|\theta_Ah\right\|^2&=\int_{\Omega}\|A_\alpha h\|^2\,d\mu(\alpha)=\int_{\Omega}\sum\limits_{\beta \in\Omega_\alpha}|\langle h, u_{\alpha,\beta}\rangle|^2\,d\mu(\alpha); \\
\left\|\theta_\Psi h\right\|^2&=\int_{\Omega}\sum\limits_{\beta \in\Omega_\alpha}|\langle h, v_{\alpha,\beta}\rangle|^2\,d\mu(\alpha).
\end{align*}
 \item Similar to (i).
 \item $\mathcal{H} \ni h \mapsto \int_{\Omega}\Psi_\alpha^*A_\alpha h  \,d\mu(\alpha)\in \mathcal{H}$   exists and is bounded  positive  invertible if and only if  there exist $ c,d, r  >0$ such that 
 $  \|\int_{\Omega}\sum_{\beta\in \Omega_\alpha}\langle h, u_{\alpha,\beta}\rangle v_{\alpha,\beta}\,d\mu(\alpha)\|\leq r\|h\|,\forall h \in \mathcal{H}$, $\int_{\Omega}\sum_{\beta\in \Omega_\alpha}\langle h, u_{\alpha,\beta}\rangle v_{\alpha,\beta}\,d\mu(\alpha) =\int_{\Omega}\sum_{\beta\in \Omega_\alpha}\langle h, v_{\alpha,\beta}\rangle u_{\alpha,\beta}\,d\mu(\alpha) ,\forall h \in \mathcal{H}$ and $
 a\|h\|^2\leq \int_{\Omega}\sum_{\beta\in\Omega_\alpha}\langle h, u_{\alpha,\beta}\rangle \langle  v_{\alpha,\beta} , h\rangle \,d\mu(\alpha) \leq b\|h\|^2 , \forall h \in \mathcal{H} $. Also, $ \theta_A:\mathcal{H} \ni h \mapsto \theta_Ah \in \mathcal{L}^2(\Omega, \mathcal{H}_0) $, $\theta_Ah: \Omega \ni \alpha \mapsto A_\alpha h \in \mathcal{H}_0$, (resp. $\theta_\Psi:\mathcal{H} \ni h \mapsto \theta_\Psi h \in \mathcal{L}^2(\Omega, \mathcal{H}_0) $, $\theta_\Psi h: \Omega \ni \alpha \mapsto \Psi_\alpha h \in \mathcal{H}_0 $) exists and is bounded if and only if there exists $ c>0$ (resp. $ d>0$) such that $\int_{\Omega}\sum_{\beta\in \Omega_\alpha}|\langle h, u_{\alpha,\beta}\rangle |^2 \,d\mu(\alpha)\leq c\|h\|^2 , \forall h \in \mathcal{H} $ (resp. $ \int_{\Omega}\sum_{\beta\in \Omega_\alpha}|\langle h, v_{\alpha,\beta}\rangle |^2\,d\mu(\alpha) \leq d\|h\|^2 , \forall h \in \mathcal{H}$).
 \item Similar to (iii).
\end{enumerate}
\end{proof}

\textbf{Similarity}\label{SIMILARITYCOMPOSITIONANDTENSORPRODUCT}
\begin{definition}
 A continuous  (ovf)  $(\{B_\alpha\}_{\alpha\in \Omega},  \{\Phi_\alpha\}_{\alpha\in \Omega})$  in $ \mathcal{B}(\mathcal{H}, \mathcal{H}_0)$    is said to be right-similar  to a continuous  (ovf)  $(\{A_\alpha\}_{\alpha\in \Omega},   \{\Psi_\alpha\}_{\alpha\in \Omega})$ in $ \mathcal{B}(\mathcal{H}, \mathcal{H}_0)$  if there exist invertible  $ R_{A,B}, R_{\Psi, \Phi} \in \mathcal{B}(\mathcal{H})$   such that $B_\alpha=A_\alpha R_{A,B} , \Phi_\alpha=\Psi_\alpha R_{\Psi, \Phi} $, $\forall \alpha \in \Omega$.
\end{definition}

\begin{proposition}\label{RIGHTSIMILARITYPROPOSITIONOPERATORVERSION}
Let $ \{A_\alpha\}_{\alpha\in \Omega}\in \mathscr{F}_\Psi$  with frame bounds $a, b,$  let $R_{A,B}, R_{\Psi, \Phi} \in \mathcal{B}(\mathcal{H})$ be positive, invertible, commute with each other, commute with $ S_{A, \Psi}$, and let $B_\alpha=A_\alpha R_{A,B} , \Phi_\alpha=\Psi_\alpha R_{\Psi, \Phi},  \forall \alpha \in \Omega.$ Then 
\begin{enumerate}[\upshape(i)]
\item $ \{B_\alpha\}_{\alpha\in \Omega}\in \mathscr{F}_\Phi$ and $ \frac{a}{\|R_{A,B}^{-1}\|\|R_{\Psi,\Phi}^{-1}\|}\leq S_{B, \Phi} \leq b\|R_{A,B}R_{\Psi,\Phi}\|.$ Assuming that $ (\{A_\alpha\}_{\alpha\in \Omega},\{\Psi_\alpha\}_{\alpha\in \Omega})$ is Parseval, then $(\{B_\alpha\}_{\alpha\in \Omega},  \{\Phi_\alpha\}_{\alpha\in \Omega})$ is Parseval  if and only if   $ R_{\Psi, \Phi}R_{A,B}=I_\mathcal{H}.$  
\item $ \theta_B=\theta_A R_{A,B}, \theta_\Phi=\theta_\Psi R_{\Psi,\Phi}, S_{B,\Phi}=R_{\Psi,\Phi}S_{A, \Psi}R_{A,B},  P_{B,\Phi}=P_{A, \Psi}.$
\end{enumerate}
\end{proposition}
\begin{proof}
For all $h, g \in \mathcal{H}$,

\begin{align*}
&\langle R_{\Psi,\Phi}S_{A, \Psi}R_{A,B}h, g\rangle =\langle S_{A, \Psi}R_{A,B}h, R_{\Psi,\Phi}^*g\rangle=\langle S_{A, \Psi}R_{A,B}h, R_{\Psi,\Phi}g\rangle\\
&=\int_{\Omega}\langle  A_\alpha R_{A,B}h, \Psi_\alpha R_{\Psi,\Phi}g\rangle\,d\mu(\alpha)
=\int_{\Omega}\langle  (\Psi_\alpha R_{\Psi,\Phi})^*A_\alpha R_{A,B}h, g\rangle\,d\mu(\alpha)=\langle  S_{B,\Phi}h, g\rangle.
\end{align*}
 \end{proof}
\begin{lemma}\label{SIM}
 Let $ \{A_\alpha\}_{\alpha\in \Omega}\in \mathscr{F}_\Psi,$ $ \{B_\alpha\}_{\alpha\in \Omega}\in \mathscr{F}_\Phi$ and   $B_\alpha=A_\alpha R_{A,B} ,\Phi_\alpha=\Psi_\alpha R_{\Psi, \Phi},  \forall \alpha \in \Omega$, for some invertible $ R_{A,B} ,R_{\Psi, \Phi} \in \mathcal{B}(\mathcal{H}).$ Then 
  $ \theta_B=\theta_A R_{A,B}, \theta_\Phi=\theta_\Psi R_{\Psi,\Phi}, S_{B,\Phi}=R_{\Psi,\Phi}^*S_{A, \Psi}R_{A,B},  P_{B,\Phi}=P_{A, \Psi}.$ Assuming that $ (\{A_\alpha\}_{\alpha\in \Omega},\{\Psi_\alpha\}_{\alpha\in \Omega})$ is Parseval, then $(\{B_\alpha\}_{\alpha\in \Omega},  \{\Phi_\alpha\}_{\alpha\in \Omega})$ is Parseval  if and only if   $ R_{\Psi, \Phi}^*R_{A,B}=I_\mathcal{H}.$ 
 \end{lemma}
\begin{proof}
$\theta_Bh(\alpha)=B_\alpha h=(A_\alpha R_{A,B})h=A_\alpha (R_{A,B}h)=\theta_A(R_{A,B}h)(\alpha)=(\theta_AR_{A,B})h(\alpha),\forall  \alpha \in \Omega, \forall h \in \mathcal{H}$ $\Rightarrow$ $\theta_B=\theta_AR_{A,B}  , \forall \alpha \in \Omega $. Similarly $\theta_\Phi=\theta_\Psi R_{\Psi,\Phi}$.
\end{proof}
 
\begin{theorem}\label{RIGHTSIMILARITY}
Let $ \{A_\alpha\}_{\alpha\in \Omega}\in \mathscr{F}_\Psi,$ $ \{B_\alpha\}_{\alpha\in \Omega}\in \mathscr{F}_\Phi.$ The following are equivalent.
\begin{enumerate}[\upshape(i)]
\item $B_\alpha=A_\alpha R_{A,B} , \Phi_\alpha=\Psi_\alpha R_{\Psi, \Phi} ,  \forall \alpha \in \Omega,$ for some invertible  $ R_{A,B} ,R_{\Psi, \Phi} \in \mathcal{B}(\mathcal{H}). $
\item $\theta_B=\theta_AR_{A,B}' , \theta_\Phi=\theta_\Psi R_{\Psi, \Phi}' $ for some invertible  $ R_{A,B}' ,R_{\Psi, \Phi}' \in \mathcal{B}(\mathcal{H}). $
\item $P_{B,\Phi}=P_{A,\Psi}.$
\end{enumerate}
If one of the above conditions is satisfied, then  invertible operators in  $ \operatorname{(i)}$ and  $ \operatorname{(ii)}$ are unique and are given by $R_{A,B}=S_{A,\Psi}^{-1}\theta_\Psi^*\theta_B, R_{\Psi, \Phi}=S_{A,\Psi}^{-1}\theta_A^*\theta_\Phi.$
In the case that $(\{A_\alpha\}_{\alpha\in \Omega},  \{\Psi_\alpha\}_{\alpha\in \Omega})$ is Parseval, then $(\{B_\alpha\}_{\alpha\in \Omega},  \{\Phi_\alpha\}_{\alpha\in \Omega})$ is  Parseval if and only if $R_{\Psi, \Phi}^*R_{A,B} $  is the identity operator if and only if $R_{A,B}R_{\Psi, \Phi}^* $  is the identity operator. 
\end{theorem}
\begin{proof}
(ii) $\Rightarrow$ (i)  $ B_\alpha h=\theta_Bh(\alpha)=\theta_Bh(\alpha)=(\theta_AR_{A,B}')h(\alpha)=\theta_A(R_{A,B}'h)(\alpha)=A_\alpha(R_{A,B}'h)=(A_\alpha R_{A,B}')h$, $\forall \alpha \in \Omega, \forall h \in \mathcal{H}$ $\Rightarrow$ $B_\alpha=A_\alpha R_{A,B}', \forall \alpha \in \Omega$. Similarly $ \Phi_\alpha=\Psi_\alpha R_{\Psi, \Phi}' ,  \forall \alpha \in \Omega$. Other arguments are similar to that in the proof of Theorem 4.4 in \cite{MAHESHKRISHNASAMJOHNSON}.
\end{proof}
\begin{corollary}
For any given continuous (ovf) $ (\{A_\alpha\}_{\alpha \in \Omega} , \{\Psi_\alpha\}_{\alpha \in \Omega})$, the canonical dual of $ (\{A_\alpha\}_{\alpha \in \Omega} $, $ \{\Psi_\alpha\}_{\alpha \in\Omega}  )$ is the only dual continuous (ovf) that is right-similar to $ (\{A_\alpha\}_{\alpha \in \Omega} , \{\Psi_\alpha\}_{\alpha \in \Omega} )$.
\end{corollary}
\begin{proof}
Similar to the proof of Corollary 4.5 in \cite{MAHESHKRISHNASAMJOHNSON}.
\end{proof}
\begin{corollary}
Two right-similar continuous operator-valued frames cannot be orthogonal.
\end{corollary}
\begin{proof}
Similar to the proof of Corollary 4.6  in \cite{MAHESHKRISHNASAMJOHNSON}.
\end{proof}
\begin{remark}
 For every continuous (ovf) $(\{A_\alpha\}_{\alpha \in \Omega},\{\Psi_\alpha\}_{\alpha \in \Omega})$, each  of `continuous operator-valued frames'  $( \{A_\alpha S_{A, \Psi}^{-1}\}_{\alpha \in \Omega}, \{\Psi_\alpha\}_{\alpha \in \Omega}),$   $( \{A_\alpha S_{A, \Psi}^{-1/2}\}_{\alpha \in \Omega}, \{\Psi_\alpha S_{A,\Psi}^{-1/2}\}_{\alpha \in \Omega}),$ and  $ (\{A_\alpha \}_{\alpha \in \Omega}, \{\Psi_\alpha S_{A,\Psi}^{-1}\}_{\alpha \in \Omega})$ is a Parseval continuous  (ovf) which is right-similar to  $ (\{A_\alpha\}_{\alpha \in \Omega} , \{\Psi_\alpha\}_{\alpha \in \Omega}  ).$  Thus every continuous (ovf) is right-similar to  Parseval continuous   operator-valued frames.
\end{remark}

\section{Continuous frames and  representations of locally compact  groups} \label{FRAMESANDDISCRETEGROUPREPRESENTATIONS}
Let $ G$ be a locally compact group, $\mu_G$  be a left-invariant  Haar measure on $G$    (we refer \cite{FOLLAND, HEWITTROSS, DIESTELSPALSBURY, NACHBIN, AMBROSE} for locally compact groups and  Haar measures).  Let $\lambda $ be the left regular representation of $ G$ defined by $ \lambda_gf(x)=f(g^{-1}x), \forall  g, x \in G, \forall f \in\mathcal{B}(\mathcal{L}^2(G,\mathcal{H}_0)) $;  $\rho $ be the right regular representation of $ G$ defined by $ \rho_gf(x)=\Delta_G(g)^{1/2}f(xg), \forall g, x\in G,\forall f \in \mathcal{B}(\mathcal{L}^2(G,\mathcal{H}_0))$, where $\Delta_G$ is the modular function associated with $G$ \cite{TAKESAKI2}. 

 \begin{definition}
Let $ \pi$ be a unitary representation of a locally compact group $ G$ on  a Hilbert space $ \mathcal{H}.$ An operator $ A$ in $ \mathcal{B}(\mathcal{H}, \mathcal{H}_0)$ is called a continuous  operator-valued  frame generator (resp. a  Parseval frame generator) w.r.t. an operator $ \Psi$ in $ \mathcal{B}(\mathcal{H}, \mathcal{H}_0)$ if $(\{A_g\coloneqq A \pi_{g^{-1}}\}_{g\in G}, \{\Psi_g\coloneqq \Psi \pi_{g^{-1}}\}_{g\in G})$ is a continuous (ovf) (resp. a Parseval continuous (ovf)) in $ \mathcal{B}(\mathcal{H}, \mathcal{H}_0)$ (where the measure on $G$ is a left invariant  Haar measure $\mu_G$). In this case, we write $ (A,\Psi)$ is a continuous operator-valued  frame generator for $\pi$.
\end{definition}

\begin{proposition}\label{REPRESENATIONLEMMA}
Let $ (A,\Psi)$ and $ (B,\Phi)$ be   continuous operator-valued frame generators    in $\mathcal{B}(\mathcal{H},  \mathcal{H}_0)$ for a unitary representation $ \pi$ of a locally compact group $G$ on $ \mathcal{H}.$ Then
\begin{enumerate}[\upshape(i)]
\item $ \theta_A\pi_g=\lambda_g\theta_A,  \theta_\Psi \pi_g=\lambda_g\theta_\Psi,  \forall g \in G.$
\item $ \theta_A^*\theta_B,   \theta_\Psi^*\theta_\Phi,\theta_A^*\theta_\Phi$ are in the commutant $ \pi(G)'$ of $ \pi(G)''.$ Further, $ S_{A,\Psi} \in \pi(G)'$ and $(AS_{A, \Psi}^{-{1/2}} , \Psi S_{A, \Psi}^{-{1/2}})$ is a Parseval frame generator. 
\end{enumerate}
\end{proposition}
\begin{proof} 
\begin{enumerate}[\upshape(i)]
\item  For all $h\in \mathcal{H}$ and $f\in \mathcal{L}^2(G, \mathcal{H}_0)$,

\begin{align*}
\langle \lambda_g\theta_Ah, f\rangle &=\int_{G}\langle \lambda_g\theta_Ah(\alpha), f(\alpha)\rangle \,d\mu_G(\alpha)=\int_{G}\langle \theta_Ah(g^{-1}\alpha), f(\alpha)\rangle \,d\mu_G(\alpha)\\
&=\int_{G}\langle A_{g^{-1}\alpha}h, f(\alpha)\rangle \,d\mu_G(\alpha)
=\int_{G}\langle A\pi_{(g^{-1}\alpha)^{-1}}h, f(\alpha)\rangle \,d\mu_G(\alpha)\\
&=\left\langle  \pi_gh,\int_{G}(A\pi_{\alpha^{-1}})^* f(\alpha)\,d\mu_G(\alpha)\right\rangle 
=\left\langle  \pi_gh,\int_{G}A_\alpha^* f(\alpha)\,d\mu_G(\alpha)\right\rangle
\\
&=\langle \pi_gh, \theta_A^*f\rangle =\langle \theta_A\pi_gh, f\rangle 
\end{align*}
$\Rightarrow\lambda_g\theta_A=\theta_A\pi_g$. Similarly $ \theta_\Psi \pi_g=\lambda_g \theta_\Psi.$
\item $\theta_A^*\theta_B\pi_g=\theta_A^*\lambda_g\theta_B=(\lambda_{g^{-1}}\theta_A)^*\theta_B=(\theta_A\pi_{g^{-1}})^*\theta_B=\pi_g\theta_A^*\theta_B$ and for all $x,y \in \mathcal{H}$

\begin{align*}
\left\langle \int_{G} (\Psi S_{A,\Psi}^{-\frac{1}{2}}\pi_{g^{-1}})^*(A S_{A,\Psi}^{-\frac{1}{2}}\pi_{g^{-1}})x \, d \mu_G(g),y  \right\rangle &= \int_{G} \langle(\Psi S_{A,\Psi}^{-\frac{1}{2}}\pi_{g^{-1}})^*(A S_{A,\Psi}^{-\frac{1}{2}}\pi_{g^{-1}})x,y  \rangle\, d \mu_G(g)\\
&= \int_{G}\langle \pi_gS_{A,\Psi}^{-\frac{1}{2}}\Psi^* AS_{A,\Psi}^{-\frac{1}{2}}\pi_{g^{-1}}x,y  \rangle\, d \mu_G(g)\\
&=\int_{G}\langle S_{A,\Psi}^{-\frac{1}{2}}\pi_g\Psi^* A\pi_{g^{-1}}S_{A,\Psi}^{-\frac{1}{2}}x,y  \rangle\, d \mu_G(g)\\
&=\int_{G}\langle (\Psi\pi_{g^{-1}})^* (A\pi_{g^{-1}})S_{A,\Psi}^{-\frac{1}{2}}x,S_{A,\Psi}^{-\frac{1}{2}}y  \rangle\, d \mu_G(g)\\
&=\left\langle\int_{ G}\Psi_g^*A_g(S_{A,\Psi}^{-\frac{1}{2}}x)\, d \mu_G(g),S_{A,\Psi}^{-\frac{1}{2}}y\right\rangle\\
&=\langle S_{A,\Psi}(S_{A,\Psi}^{-\frac{1}{2}}x),S_{A,\Psi}^{-\frac{1}{2}}y \rangle =\langle x,y\rangle
\end{align*}
$\Rightarrow $ $\int_{G} (\Psi S_{A,\Psi}^{-\frac{1}{2}}\pi_{g^{-1}})^*(A S_{A,\Psi}^{-\frac{1}{2}}\pi_{g^{-1}})x \, d \mu(g)=x, \forall x \in \mathcal{H}$ and  hence the last part.
\end{enumerate}
\end{proof}

\begin{theorem}\label{gc1}
Let $ G$ be a locally compact group  with identity $ e$ and $( \{A_g\}_{g\in G},  \{\Psi_g\}_{g\in G})$ be a Parseval  continuous (ovf) in $ \mathcal{B}(\mathcal{H},\mathcal{H}_0).$ Then there is a  unitary representation $ \pi$  of $ G$ on  $ \mathcal{H}$  for which 
$$ A_g=A_e\pi_{g^{-1}}, ~\Psi_g=\Psi_e\pi_{g^{-1}}, \quad\forall  g \in G$$
 if and only if 
$$A_{gp}A_{gq}^*=A_pA_q^* ,~ A_{gp}\Psi_{gq}^*=A_p\Psi_q^*,~ \Psi_{gp}\Psi_{gq}^*=\Psi_p\Psi_q^*, \quad \forall g,p,q \in G.$$
\end{theorem} 
\begin{proof}
$(\Rightarrow)$ Similar to the proof of `only if' part of Theorem 5.3 in \cite{MAHESHKRISHNASAMJOHNSON}.

 $ (\Leftarrow)$ We claim   the following three equalities among them  we derive the second,  two others are similar.
 For all $ g \in G,$
 
 \begin{align*}
 \lambda_g\theta_A\theta_A^*=\theta_A\theta_A^*\lambda_g , ~ \lambda_g \theta_A\theta_\Psi^*=\theta_A\theta_\Psi^*\lambda_g ,~
 \lambda_g\theta_\Psi\theta_\Psi^*=\theta_\Psi\theta_\Psi^*\lambda_g.
 \end{align*}
Let $u,v \in \mathcal{L}^2(G, \mathcal{H}_0)$. Then 

 \begin{align*}
 \langle \lambda_g\theta_A\theta_\Psi^*\lambda_g^*u, v\rangle &= \langle \theta_\Psi^*\lambda_g^*u, \theta_A^*\lambda_g^*v\rangle\\
 &=\left\langle\int_{G}\Psi_\alpha^*\lambda_g^*u(\alpha)\,d\mu_G(\alpha) ,\int_{G}A_\beta^*\lambda_g^*v(\beta) \,d\mu_G(\beta)\right \rangle\\
 &=\left\langle\int_{G}\Psi_\alpha^*\lambda_{g^{-1}}u(\alpha)\,d\mu_G(\alpha) ,\int_{G}A_\beta^*\lambda_{g^{-1}}v(\beta) \,d\mu_G(\beta)\right \rangle\\
 &=\left\langle\int_{G}\Psi_\alpha^*u(g\alpha)\,d\mu_G(\alpha) ,\int_{G}A_\beta^*v(g\beta) \,d\mu_G(\beta)\right \rangle\\
 &=\left\langle\int_{G}\Psi_{g^{-1}p}^*u(p)\,d\mu_G(g^{-1}p) ,\int_{G}A_{g^{-1}q}^*v(q) \,d\mu_G(g^{-1}q)\right \rangle\\
 &=\left\langle\int_{G}\Psi_{g^{-1}p}^*u(p)\,d\mu_G(p) ,\int_{G}A_{g^{-1}q}^*v(q) \,d\mu_G(q)\right \rangle\\
 &=\int_{G}\int_{G}\langle A_{g^{-1}q}\Psi_{g^{-1}p}^*u(p), v(q)\rangle \,d\mu_G(q)\,d\mu_G(p)\\
 &=\int_{G}\int_{G}\langle A_{q}\Psi_{p}^*u(p), v(q)\rangle \,d\mu_G(q)\,d\mu_G(p)\\
 &=\left\langle\int_{G}\Psi_p^*u(p)\,d\mu_G(p) ,\int_{G}A_q^*v(q) \,d\mu_G(q)\right \rangle=\langle \theta_\Psi^*u, \theta_A^*v\rangle=\langle \theta_A\theta_\Psi^*u, v\rangle. 
 \end{align*}
 Define $ \pi : G \ni g  \mapsto \pi_g\coloneqq \theta_\Psi^*\lambda_g\theta_A  \in \mathcal{B}(\mathcal{H}).$ Using the fact that frame is Parseval,  $ \pi_g\pi_h=\theta_\Psi^*\lambda_g \theta_A \theta_\Psi^*\lambda_h\theta_A =\theta_\Psi^*\theta_A \theta_\Psi^*\lambda_g \lambda_h\theta_A = \theta_\Psi^*\lambda_{gh}\theta_A =\pi_{gh}$ for all $g, h \in G,$ and $\pi_g\pi_g^*=\theta_\Psi^*\lambda_g\theta_A\theta_A^*\lambda_{g^{-1}} \theta_\Psi=\theta_\Psi^*\theta_A\theta_A^*\lambda_g \lambda_{g^{-1}} \theta_\Psi=I_\mathcal{H},  \pi_g^*\pi_g=\theta_A^*\lambda_{g^{-1}}\theta_\Psi\theta_\Psi^*\lambda_{g}\theta_A=\theta_A^*\lambda_{g^{-1}} \lambda_{g}\theta_\Psi\theta_\Psi^*\theta_A=I_\mathcal{H} $ for all $ g \in G$. We next prove that, for each fixed $h \in \mathcal{H}$, the map $ \phi_h:G \ni g \mapsto \pi_gh \in \mathcal{H}$ is continuous. So, let $h \in \mathcal{H}$ be fixed. Then $\theta_Ah $ is fixed. Since $\lambda$ is a unitary  representation, the map $G \ni g \mapsto \lambda_g(\theta_Ah) \in \mathcal{L}^2(G,\mathbb{K}) $ is continuous. Continuity of $\theta_\Psi^*$ now gives that the map $G \ni g \mapsto \theta_\Psi^*(\lambda_g(\theta_Ah)) \in \mathcal{H} $ is continuous, i.e., $\phi_h$ is continuous. This proves $ \pi$ is a unitary representation. We now   prove  $ A_g=A_e\pi_{g^{-1}}, \Psi_g=\Psi_e\pi_{g^{-1}}  $ for all $ g \in G$. For all $ h \in \mathcal{H}$ and for all $f \in \mathcal{L}^2(G,\mathcal{H}_0)$,
 
 \begin{align*}
 \langle A_e\pi_{g^{-1}}h, f\rangle &= \langle \pi_{g^{-1}}h,A_e^* f\rangle= \langle \theta_\Psi^*\lambda_{g^{-1}}\theta_Ah,A_e^* f\rangle\\
 &=\left\langle \int_{G}\Psi_\alpha^*\lambda_{g^{-1}}\theta_Ah(\alpha)\,d\mu_G(\alpha),A_e^*f \right\rangle 
 =\left\langle \int_{G}\Psi_\alpha^*\theta_Ah(g\alpha)\,d\mu_G(\alpha),A_e^*f \right\rangle\\
 &=\left\langle \int_{G}\Psi_\alpha^*A_{g\alpha}h\,d\mu_G(\alpha),A_e^*f \right\rangle
 =\left\langle \int_{G}\Psi_{g^{-1}\beta}^*A_{\beta}h\,d\mu_G(g^{-1}\beta),A_e^*f \right\rangle\\
 &=\left\langle \int_{G}\Psi_{g^{-1}\beta}^*A_{\beta}h\,d\mu_G(\beta),A_e^*f \right\rangle
 =\int_{G}\langle A_{g^{-1}g}\Psi_{g^{-1}\beta}^*A_{\beta}h, f \rangle \,d\mu_G(\beta)\\
 &=\int_{G}\langle A_{g}\Psi_{\beta}^*A_{\beta}h, f \rangle \,d\mu_G(\beta)
 =\left\langle\int_{G}\Psi_{\beta}^*A_{\beta}h \,d\mu_G(\beta), A_g^*f\right\rangle \\
 &=\langle h,A_g^*f\rangle =\langle A_gh,f\rangle,
 \end{align*}
and

\begin{align*}
\langle \Psi_e\pi_{g^{-1}}h, f\rangle &= \langle \pi_{g^{-1}}h,\Psi_e^* f\rangle= \langle \theta_\Psi^*\lambda_{g^{-1}}\theta_Ah,\Psi_e^* f\rangle\\
&=\left\langle \int_{G}\Psi_\alpha^*\lambda_{g^{-1}}\theta_Ah(\alpha)\,d\mu_G(\alpha),\Psi_e^*f \right\rangle 
=\left\langle \int_{G}\Psi_\alpha^*\theta_Ah(g\alpha)\,d\mu_G(\alpha),\Psi_e^*f \right\rangle\\
&=\left\langle \int_{G}\Psi_\alpha^*A_{g\alpha}h\,d\mu_G(\alpha),\Psi_e^*f \right\rangle
=\left\langle \int_{G}\Psi_{g^{-1}\beta}^*A_{\beta}h\,d\mu_G(g^{-1}\beta),\Psi_e^*f \right\rangle\\
&=\left\langle \int_{G}\Psi_{g^{-1}\beta}^*A_{\beta}h\,d\mu_G(\beta),\Psi_e^*f \right\rangle
=\int_{G}\langle \Psi_{g^{-1}g}\Psi_{g^{-1}\beta}^*A_{\beta}h, f \rangle \,d\mu_G(\beta)\\
&=\int_{G}\langle \Psi_{g}\Psi_{\beta}^*A_{\beta}h, f \rangle \,d\mu_G(\beta)
=\left\langle\int_{G}\Psi_{\beta}^*A_{\beta}h \,d\mu_G(\beta), \Psi_g^*f\right\rangle \\
&=\langle h,\Psi_g^*f\rangle =\langle \Psi_gh,f\rangle.
\end{align*}
\end{proof}

\begin{corollary}
Let $ G$ be a locally compact group  with identity $ e$ and $( \{A_g\}_{g\in G},  \{\Psi_g\}_{g\in G})$ be a continuous (ovf) in $ \mathcal{B}(\mathcal{H},\mathcal{H}_0).$ Then there is a  unitary representation $ \pi$  of $ G$ on  $ \mathcal{H}$  for which
\begin{enumerate}[\upshape(i)]
\item  $ A_g=A_eS_{A,\Psi}^{-1}\pi_{g^{-1}}S_{A,\Psi}, \Psi_g=\Psi_e\pi_{g^{-1}}  $ for all $ g \in G$  if and only if $A_{gp}S_{A,\Psi}^{-2}A_{gq}^*=A_pS_{A,\Psi}^{-2}A_q^* , A_{gp}S_{A,\Psi}^{-1}\Psi_{gq}^*$ $=A_pS_{A,\Psi}^{-1}\Psi_q^*, \Psi_{gp}\Psi_{gq}^*=\Psi_p\Psi_q^*$ for all $ g,p,q \in G.$
\item $ A_g=A_eS_{A,\Psi}^{-1/2}\pi_{g^{-1}}S_{A,\Psi}^{1/2}, \Psi_g=\Psi_eS_{A,\Psi}^{-1/2}\pi_{g^{-1}}S_{A,\Psi}^{1/2}  $ for all $ g \in G$  if and only if $A_{gp}S_{A,\Psi}^{-1}A_{gq}^*=A_pS_{A,\Psi}^{-1}A_q^* , A_{gp}S_{A,\Psi}^{-1}\Psi_{gq}^*=A_pS_{A,\Psi}^{-1}\Psi_q^*, \Psi_{gp}S_{A,\Psi}^{-1}\Psi_{gq}^*=\Psi_pS_{A,\Psi}^{-1}\Psi_q^*$ for all $ g,p,q \in G.$
\item  $ A_g=A_e\pi_{g^{-1}}, \Psi_g=\Psi_eS_{A,\Psi}^{-1}\pi_{g^{-1}}S_{A,\Psi}  $ for all $ g \in G$  if and only if $A_{gp}A_{gq}^*=A_pA_q^* , A_{gp}S_{A,\Psi}^{-1}\Psi_{gq}^*=A_pS_{A,\Psi}^{-1}\Psi_q^*, \Psi_{gp}S_{A,\Psi}^{-2}\Psi_{gq}^*=\Psi_pS_{A,\Psi}^{-2}\Psi_q^*$ for all $ g,p,q \in G.$
\end{enumerate}	
\end{corollary}
\begin{proof}
We apply Theorem  \ref{gc1} to the Parseval continuous (ovf) 
\begin{enumerate}[\upshape(i)]
\item  $(\{A_gS_{A,\Psi}^{-1}\}_{g\in G}, \{\Psi_g\}_{g\in G})$ to get: there is a  unitary representation $ \pi$  of $ G$ on  $ \mathcal{H}$  for which $ A_gS_{A,\Psi}^{-1}=(A_eS_{A,\Psi}^{-1})\pi_{g^{-1}}, \Psi_g=\Psi_e\pi_{g^{-1}}  $ for all $ g \in G$  if and only if $(A_{gp}S_{A,\Psi}^{-1})(A_{gq}S_{A,\Psi}^{-1})^*=(A_pS_{A,\Psi}^{-1})(A_qS_{A,\Psi}^{-1})^*$, $(A_{gp}S_{A,\Psi}^{-1})\Psi_{gq}^*= (A_pS_{A,\Psi}^{-1})\Psi_q^*$, $ \Psi_{gp}\Psi_{gq}^*=\Psi_p\Psi_q^*$ for all $ g,p,q \in G.$
\item $( \{A_gS_{A,\Psi}^{-1/2}\}_{g\in G}, \{\Psi_gS_{A,\Psi}^{-1/2}\}_{g\in G})$ to get: there is a  unitary representation $ \pi$  of $ G$ on  $ \mathcal{H}$  for which $ A_gS_{A,\Psi}^{-1/2}=(A_eS_{A,\Psi}^{-1/2})\pi_{g^{-1}}, \Psi_gS_{A,\Psi}^{-1/2}=(\Psi_eS_{A,\Psi}^{-1/2})\pi_{g^{-1}}  $ for all $ g \in G$  if and only if $(A_{gp}S_{A,\Psi}^{-1/2})(A_{gq}S_{A,\Psi}^{-1/2})^*$ $=(A_pS_{A,\Psi}^{-1/2})(A_qS_{A,\Psi}^{-1/2})^*,(A_{gp}S_{A,\Psi}^{-1/2})(\Psi_{gq}S_{A,\Psi}^{-1/2})^*= (A_pS_{A,\Psi}^{-1/2})(\Psi_qS_{A,\Psi}^{-1/2})^*,$  $
(\Psi_{gp}S_{A,\Psi}^{-1/2})(\Psi_{gq}S_{A,\Psi}^{-1/2})^*$ $=(\Psi_pS_{A,\Psi}^{-1/2})(\Psi_qS_{A,\Psi}^{-1/2})^*$ for all $ g,p,q \in G.$
\item $( \{A_g\}_{g\in G}, \{\Psi_gS_{A,\Psi}^{-1}\}_{g\in G})$ to get: there is a  unitary representation $ \pi$  of $ G$ on  $ \mathcal{H}$  for which $ A_g=A_e\pi_{g^{-1}}, \Psi_gS_{A,\Psi}^{-1}=(\Psi_eS_{A,\Psi}^{-1})\pi_{g^{-1}}  $ for all $ g \in G$  if and only if $A_{gp}A_{gq}^*=A_pA_q^* , A_{gp}(\Psi_{gq}S_{A,\Psi}^{-1})^*=A_p(\Psi_qS_{A,\Psi}^{-1})^*, (\Psi_{gp}S_{A,\Psi}^{-1})(\Psi_{gq}S_{A,\Psi}^{-1})^*$ $=(\Psi_pS_{A,\Psi}^{-1})(\Psi_qS_{A,\Psi}^{-1})^*$ for all $ g,p,q \in G.$
\end{enumerate}		
\end{proof}

\begin{corollary}
Let $ G$ be a locally compact group with identity $ e$ and $ \{A_g\}_{g\in G}$ be a	
\begin{enumerate}[\upshape(i)]
\item  Parseval  continuous  (ovf) (w.r.t. itself) in $ \mathcal{B}(\mathcal{H},\mathcal{H}_0)$. Then there is a unitary representation $ \pi$  of $ G$ on  $ \mathcal{H}$  for which 
$$A_g=A_e\pi_{g^{-1}},\quad \forall g \in G$$  if and only if 
$$A_{gp}A_{gq}^*=A_pA_q^* ,  \quad\forall  g,p,q \in G.  $$
\item continuous (ovf)  (w.r.t. itself) in  $ \mathcal{B}(\mathcal{H},\mathcal{H}_0)$. Then there is a unitary representation $ \pi$  of $ G$ on  $ \mathcal{H}$  for which 
$$  A_g=A_eS_{A,\Psi}^{-1/2}\pi_{g^{-1}}S_{A,\Psi}^{1/2},\quad \forall g \in G$$  if and only if 
$$A_{gp}S_{A,A}^{-1}A_{gq}^*=A_pS_{A,A}^{-1}A_q^*,  \quad\forall  g,p,q \in G.  $$
\end{enumerate}
\end{corollary}

\section{Perturbations}\label{PERTURBATIONS}
\begin{theorem}\label{PERTURBATION RESULT 1}
Let $ (\{A_\alpha\}_{\alpha\in \Omega}, \{\Psi_\alpha\}_{\alpha\in \Omega}) $  be  a continuous (ovf) in $ \mathcal{B}(\mathcal{H}, \mathcal{H}_0)$. Suppose  $\{B_\alpha\}_{\alpha\in \Omega} $ in $ \mathcal{B}(\mathcal{H}, \mathcal{H}_0)$ is such that 
 \begin{enumerate}[\upshape(i)]
 \item $ \Psi_\alpha ^*B_\alpha\geq 0, \forall \alpha \in \Omega$,
\item for each  $h \in \mathcal{H}$, the map    $\Omega \ni \alpha \mapsto B_\alpha h\in \mathcal{H}_0$ is measurable,
 \item there exist $\alpha, \beta, \gamma \geq 0  $ with $ \max\{\alpha+\gamma\|\theta_\Psi S_{A,\Psi}^{-1}\|, \beta\}<1$ such that 
 \begin{equation}\label{p3}
 \left\|\int_{\Omega}(A_\alpha^*-B_\alpha^*)f(\alpha)\,d\mu(\alpha)\right\|\leq \alpha\left\|\int_{\Omega}A_\alpha^*f(\alpha)\,d\mu(\alpha)\right\|+\beta\left\|\int_{\Omega}B_\alpha^*f(\alpha)\,d\mu(\alpha)\right\|+\gamma \|f\|,\quad \forall f \in \mathcal{L}^2(\Omega, \mathcal{H}_0).
 \end{equation} 
\end{enumerate}
  Then  $ (\{B_\alpha\}_{\alpha\in \Omega},\{\Psi_\alpha\}_{\alpha\in \Omega}) $ is a continuous  (ovf) with bounds $ \frac{1-(\alpha+\gamma\|\theta_\Psi S_{A,\Psi}^{-1}\|)}{(1+\beta)\|S_{A,\Psi}^{-1}\|}$ and $\frac{\|\theta_\Psi\|((1+\alpha)\|\theta_A\|+\gamma)}{1-\beta} $.
 \end{theorem}
 \begin{proof}
 Define $T:\mathcal{L}^2(\Omega, \mathcal{H}_0)\ni f \mapsto \int_{\Omega}B_\alpha^*f(\alpha) \,d\mu(\alpha)\in \mathcal{H} $. Then for all $f \in \mathcal{L}^2(\Omega, \mathcal{H}_0)$,
 
 \begin{align*}
 \|Tf\|&=\left\|\int_{\Omega}B_\alpha^*f(\alpha)\,d\mu(\alpha)\right\|\leq \left\|\int_{\Omega}(B_\alpha^*-A_\alpha^*)f(\alpha)\,d\mu(\alpha)\right\|+\left\|\int_{\Omega}A_\alpha^*f(\alpha)\,d\mu(\alpha)\right\|\\
 &\leq(1+\alpha)\left\|\int_{\Omega}A_\alpha^*f(\alpha)\,d\mu(\alpha)\right\|+\beta\left\|\int_{\Omega}B_\alpha^*f(\alpha)\,d\mu(\alpha)\right\|+\gamma \|f\|\\
 &=(1+\alpha)\left\|\theta_A^*f\right\|+\beta\left\|Tf\right\|+\gamma \|f\|.
 \end{align*}
Hence 
 
\begin{align*}
\|Tf\|\leq \frac{1+\alpha}{1-\beta}\left\| \theta_A^*f\right\|+\frac{\gamma}{1-\beta}\|f\|,\quad f \in \mathcal{L}^2(\Omega, \mathcal{H}_0).
\end{align*}
Thus $T$ is bounded, therefore its adjoint exists, which is $ \theta_B$. Inequality (\ref{p3}) now gives
 $$ \|\theta_A^*f-\theta_B^*f\|\leq \alpha\|\theta_A^*f\|+\beta\|\theta_B^*f\|+\gamma\|f\|, \quad \forall  f  \in \mathcal{L}^2(\Omega, \mathcal{H}_0) .$$
 Other arguments are similar to the corresponding arguments used in the proof of Theorem 7.6 in \cite{MAHESHKRISHNASAMJOHNSON}. (we note that Theorem 1 in \cite{OLECAZASSA} (we also refer \cite{HILDING, CASAZZAKALTON}) was used in the proof of  Theorem 7.6 in \cite{MAHESHKRISHNASAMJOHNSON}).
\end{proof}

\begin{corollary}
 Let $ (\{A_\alpha\}_{\alpha\in \Omega}, \{\Psi_\alpha\}_{\alpha\in \Omega}) $  be  a continuous (ovf) in $ \mathcal{B}(\mathcal{H}, \mathcal{H}_0)$. Suppose  $\{B_\alpha\}_{\alpha\in \Omega} $ in $ \mathcal{B}(\mathcal{H}, \mathcal{H}_0)$ is such that 
 \begin{enumerate}[\upshape(i)]
  \item $ \Psi_\alpha ^*B_\alpha\geq 0, \forall \alpha \in \Omega$,
  \item for each  $h \in \mathcal{H}$, the map    $\Omega \ni \alpha \mapsto B_\alpha h\in \mathcal{H}_0$ is measurable,
   \item The map    $\Omega \ni \alpha \mapsto\|A_\alpha-B_\alpha\| \in \mathbb{R}$ is measurable,
 \item 
 $$ r \coloneqq \int_{\Omega}\|A_\alpha-B_\alpha\|^2\,d\mu(\alpha) <\frac{1}{\|\theta_\Psi S_{A,\Psi}^{-1}\|^2}.$$	
 \end{enumerate}
 Then $ (\{B_\alpha\}_{\alpha\in \Omega}, \{\Psi_\alpha\}_{\alpha\in \Omega}) $ is   a continuous  (ovf) with bounds $ \frac{1-\sqrt{r}\|\theta_\Psi S_{A,\Psi}^{-1}\|}{\|S_{A,\Psi}^{-1}\|}$ and ${\|\theta_\Psi\|(\|\theta_A\|+\sqrt{r})} $.
\end{corollary}
\begin{proof}
Set $ \alpha =0, \beta=0, \gamma=\sqrt{r}$. Then $ \max\{\alpha+\gamma\|\theta_\Psi S_{A,\Psi}^{-1}\|, \beta\}<1$ and 

$$ \left\|\int_{\Omega}(A_\alpha^*-B_\alpha^*)f(\alpha)\,d\mu(\alpha)\right\|\leq \left(\int_{\Omega}\|A_\alpha^*-B_\alpha^*\|^2\,d\mu(\alpha) \right)^\frac{1}{2}\left(\int_{\Omega}\|f(\alpha)\|^2\,d\mu(\alpha)\right)^\frac{1}{2}= \gamma\|f\|, \quad \forall f \in \mathcal{L}^2(\Omega, \mathcal{H}_0).$$
 Theorem \ref{PERTURBATION RESULT 1} applies now.
\end{proof}
 \begin{theorem}\label{PERTURBATION RESULT 2}
 Let $ (\{A_\alpha\}_{\alpha\in \Omega}, \{\Psi_\alpha\}_{\alpha\in \Omega} )$ be a continuous (ovf) in $ \mathcal{B}(\mathcal{H}, \mathcal{H}_0)$ with bounds $ a$ and $b$. Suppose  $\{B_\alpha\}_{\alpha\in \Omega} $ is continuous Bessel  (w.r.t. itself) in $ \mathcal{B}(\mathcal{H}, \mathcal{H}_0)$  such that $ \theta_\Psi^*\theta_B\geq0$  and 
 there exist $\alpha, \beta, \gamma \geq 0  $ with $ \max\{\alpha+\frac{\gamma}{\sqrt{a}}, \beta\}<1$ and for all $h \in \mathcal{H}$, 
\begin{equation} \label{pe1}
 \left|\int_{\Omega}\langle(A_\alpha-B_\alpha)h,\Psi_\alpha h \rangle\,d\mu(\alpha)\right|^\frac{1}{2}\leq \alpha \left(\int_{\Omega}\langle A_\alpha h,\Psi_\alpha h \rangle\,d\mu(\alpha)\right)^\frac{1}{2}+\beta\left(\int_{\Omega}\langle B_\alpha h,\Psi_\alpha h \rangle\,d\mu(\alpha)\right)^\frac{1}{2}+\gamma \|h\|. 
 \end{equation}
 Then  $ (\{B_\alpha\}_{\alpha\in \Omega}, \{\Psi_\alpha\}_{\alpha\in \Omega}) $ is a continuous (ovf) with bounds $a\left(1-\frac{\alpha+\beta+\frac{\gamma}{\sqrt{a}}}{1+\beta}\right)^2 $ and $b\left(1+\frac{\alpha+\beta+\frac{\gamma}{\sqrt{b}}}{1-\beta}\right)^2 .$
 \end{theorem}
 \begin{proof}
For all $ h $ in $\mathcal{H}$, 
	
 \begin{align*}
 \left(\int_{\Omega}\langle B_\alpha h,\Psi_\alpha h \rangle\,d\mu(\alpha)\right)^\frac{1}{2}
 &\leq  \left|\int_{\Omega}\langle (B_\alpha -A_\alpha )h,\Psi_\alpha h \rangle\,d\mu(\alpha)\right|^\frac{1}{2}+\left(\int_{\Omega}\langle A_\alpha h,\Psi_\alpha h \rangle\,d\mu(\alpha)\right)^\frac{1}{2} \\
 &\leq(1+ \alpha )\left(\int_{\Omega}\langle A_\alpha h,\Psi_\alpha h \rangle\,d\mu(\alpha)\right)^\frac{1}{2}+\beta\left(\int_{\Omega}\langle B_\alpha h,\Psi_\alpha h \rangle\,d\mu(\alpha)\right)^\frac{1}{2}+\gamma \|h\|
 \end{align*}
 
 which implies, for all   $h \in \mathcal{H}$,
 
 \begin{align*}
 (1-\beta)\left(\int_{\Omega}\langle B_\alpha h,\Psi_\alpha h \rangle\,d\mu(\alpha)\right)^\frac{1}{2}
 &\leq(1+ \alpha )\left(\int_{\Omega}\langle A_\alpha h,\Psi_\alpha h \rangle\,d\mu(\alpha)\right)^\frac{1}{2}+\gamma \|h\|
 \\
 &\leq (1+\alpha )\sqrt{b}\|h\|+\gamma\|h\|.
 \end{align*}
  From Inequality (\ref{pe1}),  for all   $h \in \mathcal{H}$,
  
 \begin{align*}
 \left(\int_{\Omega}\langle A_\alpha h,\Psi_\alpha h \rangle\,d\mu(\alpha)\right)^\frac{1}{2}
 &\leq  \left|\int_{\Omega}\langle (A_\alpha -B_\alpha )h,\Psi_\alpha h \rangle\,d\mu(\alpha)\right|^\frac{1}{2}+\left(\int_{\Omega}\langle B_\alpha h,\Psi_\alpha h \rangle\,d\mu(\alpha)\right)^\frac{1}{2} \\
 &\leq \alpha \left(\int_{\Omega}\langle A_\alpha h,\Psi_\alpha h \rangle\,d\mu(\alpha)\right)^\frac{1}{2}+(1+\beta)\left(\int_{\Omega}\langle B_\alpha h,\Psi_\alpha h \rangle\,d\mu(\alpha)\right)^\frac{1}{2}+\gamma \|h\|\\
 &\leq \left(\alpha+\frac{\gamma}{\sqrt{a}}\right) \left(\int_{\Omega}\langle A_\alpha h,\Psi_\alpha h \rangle\,d\mu(\alpha)\right)^\frac{1}{2}+(1+\beta)\left(\int_{\Omega}\langle B_\alpha h,\Psi_\alpha h \rangle\,d\mu(\alpha)\right)^\frac{1}{2}
 \end{align*}
 which produces
 
 \begin{align*}
 \left(1-\left(\alpha +\frac{\gamma}{\sqrt{a}}\right) \right)\left(\int_{\Omega}\langle A_\alpha h,\Psi_\alpha h \rangle\,d\mu(\alpha)\right)^\frac{1}{2}
 \leq  (1+\beta)\left(\int_{\Omega}\langle B_\alpha h,\Psi_\alpha h \rangle\,d\mu(\alpha)\right)^\frac{1}{2},\quad  \forall h \in \mathcal{H}.
  \end{align*}
  But $\sqrt{a}\|h\|\leq  \left(\int_{\Omega}\langle A_\alpha h,\Psi_\alpha h \rangle\,d\mu(\alpha)\right)^{1/2},  \forall h \in \mathcal{H}.$
 Thus $ (\{B_\alpha\}_{\alpha\in \Omega}, \{\Psi_\alpha\}_{\alpha\in \Omega} )$ is a continuous (ovf) with bounds $a\left(1-\frac{\alpha+\beta+\frac{\gamma}{\sqrt{a}}}{1+\beta}\right)^2 $ and 
 $b\left(1+\frac{\alpha+\beta+\frac{\gamma}{\sqrt{b}}}{1-\beta}\right)^2 .$
 \end{proof}
 \begin{theorem}\label{OVFQUADRATICPERTURBATION}
Let $ (\{A_\alpha\}_{\alpha\in \Omega}, \{\Psi_\alpha\}_{\alpha\in \Omega}) $  be a continuous  (ovf) in $ \mathcal{B}(\mathcal{H}, \mathcal{H}_0)$. Suppose  $\{B_\alpha\}_{\alpha\in \Omega} $ in $ \mathcal{B}(\mathcal{H}, \mathcal{H}_0)$ is such that 
 \begin{enumerate}[\upshape(i)]
\item $ \Psi_\alpha ^*B_\alpha\geq 0, \forall \alpha \in \Omega$,
 \item for each  $h \in \mathcal{H}$, the map    $\Omega \ni \alpha \mapsto B_\alpha h\in \mathcal{H}_0$ is measurable,
 \item The map    $\Omega \ni \alpha \mapsto\|A_\alpha-B_\alpha\| \in \mathbb{R}$ is measurable and $   \int_{\Omega}\|A_\alpha-B_\alpha\|\,d\mu(\alpha) \in \mathbb{R}$,
 \item  The map    $\Omega \ni \alpha \mapsto\|\Psi_\alpha S_{A,\Psi}^{-1}\|\in \mathbb{R}$ is measurable and $   \int_{\Omega}\|A_\alpha-B_\alpha\|\|\Psi_\alpha S_{A,\Psi}^{-1}\|\,d\mu(\alpha) \in \mathbb{R}$,
 \item $\int_{\Omega}\|A_\alpha-B_\alpha\|\|\Psi_\alpha S_{A,\Psi}^{-1}\|\,d\mu(\alpha)<1.$
 \end{enumerate}
Then  $ (\{B_\alpha\}_{\alpha\in \Omega}, \{\Psi_\alpha\}_{\alpha\in \Omega}) $ is a continuous  (ovf) with bounds   $\frac{1-\int_{\Omega}\|A_\alpha-B_\alpha\|\|\Psi_\alpha S_{A,\Psi}^{-1}\|\,d\mu(\alpha)}{\|S_{A,\Psi}^{-1}\|}$ and $\|\theta_\Psi\|((\int_{\Omega}\|A_\alpha-B_\alpha\|^2\,d\mu(\alpha))^{1/2}+\|\theta_A\|) $.
 \end{theorem}
 \begin{proof}
Let $ \alpha =(\int_{\Omega}\|A_\alpha-B_\alpha\|^2\,d\mu(\alpha))^{1/2} $ and $\beta =\int_{\Omega}\|A_\alpha-B_\alpha\|\|\Psi_\alpha S_{A,\Psi}^{-1}\|\,d\mu(\alpha)$.  Fix $f\in \mathcal{L}^2(\Omega, \mathcal{H}_0)$. Then for all $h \in \mathcal{H}$

\begin{align*}
\left| \int_{\Omega}\langle B_\alpha^*f(\alpha), h\rangle \,d\mu(\alpha)\right|&\leq\left| \int_{\Omega}\langle (B_\alpha^*-A_\alpha^*)f(\alpha), h\rangle \,d\mu(\alpha)\right|+\left| \int_{\Omega}\langle A_\alpha^*f(\alpha), h\rangle \,d\mu(\alpha)\right| \\
&=\left| \int_{\Omega}\langle f(\alpha), (B_\alpha-A_\alpha)h\rangle \,d\mu(\alpha)\right|+|\langle \theta_A^*f, h\rangle |\\
&\leq \int_{\Omega}\|f(\alpha)\|\|(B_\alpha-A_\alpha)h\|\,d\mu(\alpha)+ \|f\|\|\theta_Ah\|\\
&\leq\|h\| \int_{\Omega}\|f(\alpha)\|\|B_\alpha-A_\alpha\|\,d\mu(\alpha)+ \|f\|\|\theta_A\|\|h\|\\
&\leq\|h\| \left(\int_{\Omega}\|f(\alpha)\|^2\,d\mu(\alpha)\right)^\frac{1}{2}\left(\int_{\Omega}\|B_\alpha-A_\alpha\|^2\,d\mu(\alpha)\right)^\frac{1}{2}+ \|f\|\|\theta_A\|\|h\|\\
&=\|h\|\|f\|\alpha+ \|f\|\|\theta_A\|\|h\|=(\|f\|\alpha+ \|f\|\|\theta_A\|)\|h\|.
\end{align*}
 Hence $\theta_B$ exists and $\|\theta_B\|=\|\theta_B^*\|=\sup\{|\langle \theta_B^*f,h \rangle|:f \in \mathcal{L}^2(\Omega, \mathcal{H}_0) , h \in \mathcal{H}, \|f\|\leq1, \|h\|\leq1\}\leq \sup\{\|h\|\|f\|\alpha+ \|f\|\|\theta_A\|\|h\|:f \in \mathcal{L}^2(\Omega, \mathcal{H}_0) , h \in \mathcal{H}, \|f\|\leq1, \|h\|\leq1\}= \alpha+\|\theta_A\|$. Therefore  $S_{B,\Psi}=\theta_\Psi^*\theta_B$ exists and is positive. Now 

\begin{align*}
\|I_\mathcal{H}-S_{B,\Psi}S_{A,\Psi}^{-1}\|&=\sup_{h, g \in \mathcal{H}, \|h\|=1=\|g\|} |\langle (I_\mathcal{H}-S_{B,\Psi}S_{A,\Psi}^{-1})h, g\rangle |\\
&=\sup_{h, g \in \mathcal{H}, \|h\|=1=\|g\|}\left| \int_{\Omega}\langle A_\alpha^*\Psi_\alpha S_{A,\Psi}^{-1}h, g\rangle \,d \mu(\alpha) -\int_{\Omega}\langle B_\alpha^*\Psi_\alpha S_{A,\Psi}^{-1}h, g\rangle \,d \mu(\alpha)\right|\\
&=\sup_{h, g \in \mathcal{H}, \|h\|=1=\|g\|}\left| \int_{\Omega}\langle (A_\alpha^*-B_\alpha^*)\Psi_\alpha S_{A,\Psi}^{-1}h, g\rangle \,d \mu(\alpha) \right|\\
&\leq \sup_{h, g \in \mathcal{H}, \|h\|=1=\|g\|} \int_{\Omega}|\langle (A_\alpha^*-B_\alpha^*)\Psi_\alpha S_{A,\Psi}^{-1}h, g\rangle |\,d \mu(\alpha) \\
&\leq \sup_{h, g \in \mathcal{H}, \|h\|=1=\|g\|} \int_{\Omega}\| A_\alpha^*-B_\alpha^*\|\|\Psi_\alpha S_{A,\Psi}^{-1}\|\|h\|\|g\|\,d \mu(\alpha)\\
&=\int_{\Omega}\| A_\alpha-B_\alpha\|\|\Psi_\alpha S_{A,\Psi}^{-1}\|\,d \mu(\alpha) =\beta<1.
\end{align*}
Other arguments are  similar to the corresponding arguments used in the proof of Theorem \ref{PERTURBATION RESULT 1}.
\end{proof}

\section{Case $\mathcal{H}_0=\mathbb{K}$}\label{SEQUENTIAL} 
\begin{definition}\label{SEQUENTIAL2}
A set of vectors   $ \{x_\alpha\}_{\alpha\in \Omega}$  in a Hilbert space  $\mathcal{H}$ is said to be a continuous frame    w.r.t.  a set $ \{\tau_\alpha\}_{\alpha\in \Omega}$ in $\mathcal{H}$  if 
\begin{enumerate}[\upshape(i)]
\item for each  $h \in \mathcal{H}$, both    maps   $\Omega \ni \alpha \mapsto\langle  h, x_\alpha\rangle \in \mathbb{K}$ and $\Omega \ni\alpha \mapsto \langle  h, \tau_\alpha \rangle \in\mathbb{K}$ are measurable,
\item the map (we call as frame operator) $S_{x,\tau} : \mathcal{H} \ni h \mapsto \int_{\Omega}\langle  h, x_\alpha\rangle\tau_\alpha  \,d\mu(\alpha)\in \mathcal{H} $ (the integral is in the weak-sense) is a well-defined bounded positive invertible operator,
\item  both maps (we call as analysis operator and its adjoint as synthesis operator)  $ \theta_x:\mathcal{H} \ni h \mapsto \theta_xh \in \mathcal{L}^2(\Omega, \mathbb{K}) $, $\theta_xh: \Omega \ni \alpha \mapsto \langle  h, x_\alpha\rangle \in \mathbb{K}$, $\theta_\tau:\mathcal{H} \ni h \mapsto \theta_\tau h \in \mathcal{L}^2(\Omega, \mathcal{H}_0) $, $\theta_\tau h: \Omega \ni \alpha \mapsto \langle  h, \tau_\alpha\rangle \in \mathbb{K} $ are well-defined bounded linear operators. 
\end{enumerate}
We note that $\theta_x^*:\mathcal{L}^2(\Omega, \mathbb{K})\ni f \mapsto \int_{\Omega}f(\alpha)x_\alpha  \,d\mu(\alpha)\in \mathcal{H} $, $\theta_\tau^*:\mathcal{L}^2(\Omega, \mathbb{K})\ni f \mapsto \int_{\Omega}f(\alpha)\tau_\alpha  \,d\mu(\alpha)\in \mathcal{H}  $(both integrals are in the weak-sense). Notions of frame bounds, Parseval frame are similar to the same in Definition 8.1 in \cite{MAHESHKRISHNASAMJOHNSON}.  If the condition \text{\upshape(ii)}  is replaced by ``the map  $S_{x,\tau} : \mathcal{H} \ni h \mapsto \int_{\Omega}\langle  h, x_\alpha\rangle\tau_\alpha  \,d\mu(\alpha)\in \mathcal{H} $ is a well-defined bounded positive  operator", then we say $ \{x_\alpha\}_{\alpha\in \Omega}$   w.r.t. $ \{\tau_\alpha\}_{\alpha\in\Omega}$ is Bessel. If $ \{x_\alpha\}_{\alpha\in\Omega}$ is continuous frame (resp. Bessel) w.r.t. $\{\tau_\alpha\}_{\alpha\in\Omega}$, then we write $(\{x_\alpha\}_{\alpha\in\Omega},\{\tau_\alpha\}_{\alpha\in\Omega})$ is a continuous frame (resp. Bessel). 

For fixed $ \Omega, \mathcal{H},$  and $ \{\tau_\alpha\}_{\alpha\in \Omega}$,   the set of all continuous frames for $ \mathcal{H}$  w.r.t.  $ \{\tau_\alpha\}_{\alpha\in \Omega}$ is denoted by $ \mathscr{F}_\tau.$
\end{definition}
We note that (ii) in Definition \ref{SEQUENTIAL2} implies that  there are real $ a,b >0$ such that 

\begin{align*}
a\|h\|^2\leq \langle S_{x,\tau}h, h\rangle =\left\langle\int_\Omega\langle h, x_\alpha\rangle\tau_\alpha   \,d\mu(\alpha), h\right \rangle   =\int_\Omega\langle  h, x_\alpha \rangle\langle   \tau_\alpha,h \rangle \, d\mu(\alpha)\leq b\|h\|^2,\quad \forall h \in \mathcal{H},
\end{align*}
and  (iii) implies that there exist $c,d>0$ such that 

\begin{align*}
\|\theta_xh\|^2=\langle\theta_xh,\theta_xh \rangle=\int_\Omega\langle \theta_xh(\alpha),\theta_xh(\alpha) \rangle  \,d \mu(\alpha) =\int_\Omega|\langle  h, x_\alpha\rangle |^2  \,d\mu(\alpha)\leq c\|h\|^2,\quad \forall h \in \mathcal{H};\\
\|\theta_\tau  h\|^2=\langle\theta_\tau h,\theta_\tau h \rangle=\int_\Omega\langle \theta_\tau h(\alpha),\theta_\tau h(\alpha) \rangle  \,d \mu(\alpha) =\int_\Omega|\langle  h, \tau_\alpha\rangle |^2 \,d\mu(\alpha)\leq d\|h\|^2,\quad \forall h \in \mathcal{H}.
\end{align*}
 We note, whenever $(\{x_\alpha\}_{\alpha\in\Omega},\{\tau_\alpha\}_{\alpha\in\Omega})$ is a continuous frame for $ \mathcal{H}$, then $\overline{\operatorname{span}}\{x_\alpha\}_{\alpha\in \Omega}=\mathcal{H}=\overline{\operatorname{span}}\{\tau_\alpha\}_{\alpha\in\Omega}.$
\begin{theorem}\label{OVFTOSEQUENCEANDVICEVERSATHEOREM}
Let $\{x_\alpha\}_{\alpha\in\Omega}, \{\tau_\alpha\}_{\alpha\in\Omega}$ be in $\mathcal{H}$. Define $A_\alpha: \mathcal{H} \ni h \mapsto \langle h, x_\alpha \rangle \in \mathbb{K} $, $\Psi_\alpha: \mathcal{H} \ni h \mapsto \langle h, \tau_\alpha \rangle \in \mathbb{K}, \forall \alpha \in \Omega $. Then   $(\{x_\alpha\}_{\alpha\in\Omega}, \{\tau_\alpha\}_{\alpha\in\Omega})$ is a continuous frame for  $\mathcal{H}$ if and only if  $(\{A_\alpha\}_{\alpha\in\Omega}, \{\Psi_\alpha\}_{\alpha\in\Omega})$ is a  continuous operator-valued frame  in $\mathcal{B}(\mathcal{H},\mathbb{K})$.
\end{theorem} 
\begin{proof}
$\langle\int_{\Omega}\Psi^*_\alpha A_\alpha h \,d\mu(\alpha), g \rangle =\int_{\Omega}\langle\Psi^*_\alpha A_\alpha h,g \rangle\,d\mu(\alpha)=\int_{\Omega}\langle A_\alpha h, \Psi_\alpha g\rangle\,d\mu(\alpha)=\int_{\Omega}\langle h,x_\alpha \rangle\langle \tau_\alpha,h \rangle\,d\mu(\alpha)=\langle\int_{\Omega}\langle h,x_\alpha \rangle \tau_\alpha\,d\mu(\alpha), g\rangle,\forall h,g \in \mathcal{H}$.
\end{proof}
\begin{proposition}
Definition \ref{SEQUENTIAL2} holds if and only if there are  $ a, b, c,  d  >0$ such that
\begin{enumerate}[\upshape(i)]
\item for each  $h \in \mathcal{H}$, both    maps   $\Omega \ni \alpha \mapsto\langle  h, x_\alpha\rangle \in \mathbb{K}$, $\Omega \ni\alpha \mapsto \langle  h, \tau_\alpha \rangle \in\mathbb{K}$ are measurable, 
\item $ a\|h\|^2\leq \int_{\Omega}\langle h,x_\alpha\rangle\langle \tau_\alpha, h \rangle \,d\mu(\alpha)\leq b\|h\|^2 , \forall h \in \mathcal{H},$
\item $  \int_{\Omega}|\langle h,x_\alpha\rangle|^2\,d\mu(\alpha)  \leq c\|h\|^2 , \forall h \in \mathcal{H}; 
\int_{\Omega}|\langle h,\tau_\alpha\rangle|^2\,d\mu(\alpha) \leq d\|h\|^2 , \forall h \in \mathcal{H},$
\item $ \int_{\Omega}\langle h,x_\alpha\rangle \tau_\alpha\,d\mu(\alpha)=\int_{\Omega}\langle h,\tau_\alpha\rangle x_\alpha\,d\mu(\alpha),  \forall h \in \mathcal{H}$.
\end{enumerate}
If the space is over $ \mathbb{C},$ then  \text {\upshape{(iv)}} can be omitted.
\end{proposition}

\begin{proposition}
 Let $(\{x_\alpha\}_{\alpha\in \Omega},\{\tau_\alpha\}_{\alpha \in \Omega} )$ be a continuous frame for  $ \mathcal{H}$  with upper frame bound $b$.  If for some $ \alpha \in \Omega $ we have $\{\alpha\}$ is measurable and  $  \langle x_\alpha, x_\beta \rangle\langle \tau_\beta, x_\alpha \rangle \geq0, \forall \beta  \in \Omega$,  then $ \mu(\{\alpha\})\langle x_\alpha, \tau_\alpha \rangle\leq b$ for that $\alpha. $
 \end{proposition}
 \begin{proof}
 	
 \begin{align*}
\mu(\{\alpha\})\langle x_\alpha, x_\alpha \rangle\langle \tau_\alpha, x_\alpha \rangle &\leq \int_{\{\alpha\}}\langle x_\alpha, x_\beta \rangle\langle \tau_\beta, x_\alpha \rangle\,d\mu(\beta)+\int_{\Omega\setminus\{\alpha\}}\langle x_\alpha, x_\beta \rangle\langle \tau_\beta, x_\alpha \rangle\,d\mu(\beta)\\
&=\int_{\Omega}\langle x_\alpha, x_\beta \rangle\langle \tau_\beta, x_\alpha \rangle\,d\mu(\beta)\leq b\|x_\alpha\|^2.
 \end{align*}
 \end{proof}

\begin{proposition}
For every $ \{x_\alpha\}_{\alpha \in \Omega}  \in \mathscr{F}_\tau$,
\begin{enumerate}[\upshape (i)]
 \item  $\theta_x^* \theta_xh =  \int_{\Omega} \langle h,x_\alpha\rangle x_\alpha\,d\mu(\alpha), \theta_\tau^* \theta_\tau h =  \int_{\Omega} \langle h,\tau_\alpha\rangle \tau_\alpha\,d\mu(\alpha),  \forall h \in \mathcal{H}.$ 
\item $ S_{x, \tau} = \theta_\tau^*\theta_x=\theta_x^*\theta_\tau .$ In particular,
 $$ S_{x, \tau}h =\int_{\Omega}\langle h,x_\alpha\rangle\tau_\alpha\,d\mu(\alpha)=\int_{\Omega}\langle h,\tau_\alpha\rangle x_\alpha\,d\mu(\alpha),\quad \forall h \in \mathcal{H} ~\operatorname{and}~$$
 $$\langle S_{x, \tau}h, g\rangle =\int_{\Omega}\langle h,x_\alpha\rangle\langle\tau_\alpha, g\rangle\,d\mu(\alpha)=\int_{\Omega}\langle h,\tau_\alpha\rangle \langle x_\alpha,g\rangle\,d\mu(\alpha), \quad \forall h,g  \in \mathcal{H}.$$
 \item Every $ h \in \mathcal{H}$ can be written as 
 \begin{align*}
 h &=\int_{\Omega}\langle h, S^{-1}_{x, \tau}\tau_\alpha\rangle  x_\alpha\,d\mu(\alpha)=\int_{\Omega}\langle h,\tau_\alpha \rangle S^{-1}_{x, \tau}  x_\alpha\,d\mu(\alpha) \\
 &=\int_{\Omega}\langle h, S^{-1}_{x, \tau}x_\alpha\rangle  \tau_\alpha\,d\mu(\alpha)=\int_{\Omega}\langle h, x_\alpha\rangle S^{-1}_{x, \tau}  \tau_\alpha\,d\mu(\alpha).
 \end{align*}
 \item $(\{x_\alpha\}_{\alpha \in \Omega}, \{\tau_\alpha\}_ {\alpha \in \Omega})$ is Parseval  if and only if $  \theta_\tau^*\theta_x=I_\mathcal{H}.$ 
 \item $(\{x_\alpha\}_{\alpha \in \Omega}, \{\tau_\alpha \}_ {\alpha \in \Omega})$ is Parseval  if and only if $ \theta_x\theta_\tau^* $ is idempotent.
 \item $P_{x,\tau}\coloneqq\theta_xS_{x,\tau}^{-1}\theta_\tau^* $ is idempotent.
 \item $\theta_x$ and $ \theta_\tau$ are injective and  their  ranges are closed.
 \item  $\theta_x^* $ and $ \theta_\tau^*$ are surjective.
 \end{enumerate}
 \end{proposition} 

 \begin{definition}
 A continuous frame  $( \{x_\alpha\}_{\alpha \in \Omega}, \{\tau_\alpha\}_{\alpha \in \Omega})$ for $ \mathcal{H}$ is called a Riesz frame if $P_{x,\tau} = I_{\mathcal{L}^2(\Omega, \mathbb{K})} $. 
 \end{definition}

\begin{proposition}\label{RIESZFRAMECHARACTERIZATION}
A continuous frame  $( \{x_\alpha\}_{\alpha \in \Omega}, \{\tau_\alpha\}_{\alpha \in \Omega})$ for $ \mathcal{H}$ is  a Riesz continuous frame if  and only if $ \theta_x(\mathcal{H})=\mathcal{L}^2(\Omega, \mathbb{K})$ if  and only if $ \theta_\tau(\mathcal{H})=\mathcal{L}^2(\Omega, \mathbb{K}).$
\end{proposition}
\begin{proof}
Similar to the proof of Proposition 8.20 in \cite{MAHESHKRISHNASAMJOHNSON}.
\end{proof}

\begin{definition}
A continuous frame   $(\{y_\alpha\}_{\alpha\in \Omega}, \{\omega_\alpha\}_{\alpha\in \Omega})$  for  $\mathcal{H}$ is said to be a dual of a continuous frame  $ ( \{x_\alpha\}_{\alpha\in \Omega}, \{\tau_\alpha\}_{\alpha\in \Omega})$ for  $\mathcal{H}$  if $ \theta_\omega^*\theta_x= \theta_y^*\theta_\tau=I_{\mathcal{H}}$. The `continuous frame' $(   \{\widetilde{x}_\alpha\coloneqq S_{x,\tau}^{-1}x_\alpha\}_{\alpha\in \Omega},\{\widetilde{\tau}_\alpha\coloneqq S_{x,\tau}^{-1}\tau_\alpha\}_{\alpha \in \Omega} )$, which is a `dual' of $ (\{x_\alpha\}_{\alpha\in \Omega}, \{\tau_\alpha\}_{\alpha\in \Omega})$ is called the canonical dual of $ (\{x_\alpha\}_{\alpha\in \Omega}, \{\tau_\alpha\}_{\alpha\in \Omega})$.
\end{definition}
\begin{proposition}
 Let $( \{x_\alpha\}_{\alpha\in \Omega},\{\tau_\alpha\}_{\alpha\in \Omega} )$ be a continuous frame for  $\mathcal{H}.$ If $ h \in \mathcal{H}$ has representation  $ h=\int_{\Omega}f(\alpha)x_\alpha\,d\mu(\alpha)= \int_{\Omega}g(\alpha)\tau_\alpha\,d\mu(\alpha), $ for some  measurable  $ f,g : \Omega \rightarrow \mathbb{K}$,  then 
 $$ \int_{\Omega}f(\alpha)\overline{g(\alpha)} \,d\mu(\alpha)=\int_{\Omega}\langle h, \widetilde{\tau}_\alpha\rangle\langle \widetilde{x}_\alpha , h \rangle\,d\mu(\alpha)+\int_{\Omega}( f(\alpha)-\langle h, \widetilde{\tau}_\alpha\rangle)(\overline{g(\alpha)}-\langle \widetilde{x}_\alpha, h\rangle)\,d\mu(\alpha). $$
\end{proposition}
\begin{proof}
Right side $ =$
\begin{align*}
 &\int_{\Omega}\langle S_{x,\tau}^{-1}h, \tau_\alpha\rangle\langle x_\alpha , S_{x,\tau}^{-1}h \rangle\,d\mu(\alpha)+\int_{\Omega} f(\alpha)\overline{g(\alpha)}\,d\mu(\alpha)-\int_{\Omega}f(\alpha)\langle x_\alpha, S_{x,\tau}^{-1}h\rangle \,d\mu(\alpha)\\
 &~-\int_{\Omega}\overline{g(\alpha)}\langle S_{x,\tau}^{-1}h, \tau_\alpha\rangle \,d\mu(\alpha)+\int_{\Omega}\langle S_{x,\tau}^{-1} h, \tau_\alpha\rangle\langle x_\alpha, S_{x,\tau}^{-1}h\rangle  \,d\mu(\alpha)\\
&= 2\langle S_{x, \tau}S_{x, \tau}^{-1}h, S_{x, \tau}^{-1}h\rangle+\int_{\Omega} f(\alpha)\overline{g(\alpha)}\,d\mu(\alpha)-\left\langle\int_{\Omega}f(\alpha)x_\alpha\,d\mu(\alpha),S_{x, \tau}^{-1}h \right\rangle -\left\langle S_{x, \tau}^{-1}h,\int_{\Omega} g(\alpha)\tau_\alpha\,d\mu(\alpha) \right\rangle\\
&=2\langle h, S_{x, \tau}^{-1}h\rangle+\int_{\Omega} f(\alpha)\overline{g(\alpha)}\,d\mu(\alpha)-\langle h, S_{x, \tau}^{-1}h\rangle-\langle  S_{x, \tau}^{-1}h, h\rangle
 =\text{Left side.}
 \end{align*}
 \end{proof} 
 \begin{theorem}\label{CANONICALDUALFRAMEPROPERTYSEQUENTIALVERSION}
 Let $( \{x_\alpha\}_{\alpha\in \Omega},\{\tau_\alpha\}_{\alpha\in \Omega} )$ be a continuous frame for $ \mathcal{H}$ with frame bounds $ a$ and $ b.$ Then the following statements are true.
  \begin{enumerate}[\upshape(i)]
 \item The canonical dual continuous frame of the canonical dual continuous frame  of $ (\{x_\alpha\}_{\alpha\in \Omega} ,\{\tau_\alpha\}_{\alpha\in \Omega} )$ is itself.
 \item$ \frac{1}{b}, \frac{1}{a}$ are frame bounds for the  canonical dual of $ (\{x_\alpha\}_{\alpha\in \Omega},\{\tau_\alpha\}_{\alpha\in \Omega}).$
 \item If $ a, b $ are optimal frame bounds for $( \{x_\alpha\}_{\alpha\in \Omega} , \{\tau_\alpha\}_{\alpha\in \Omega}),$ then $ \frac{1}{b}, \frac{1}{a}$ are optimal  frame bounds for its canonical dual.
 \end{enumerate} 
 \end{theorem} 
 \begin{proof}
  For $ h, g \in \mathcal{H},$ 
  \begin{align*}
  \left\langle \int_{\Omega}\langle h, \widetilde{x}_\alpha \rangle\widetilde{\tau}_\alpha\,d\mu(\alpha), g\right \rangle &= \int_{\Omega}\left\langle h, S_{x,\tau}^{-1}x_\alpha \rangle \langle S_{x,\tau}^{-1}\tau_\alpha , g \right \rangle\,d\mu(\alpha)\\
   &=\left\langle\int_{\Omega}\langle  S_{x,\tau}^{-1}h, x_\alpha \rangle  \tau_\alpha \,d\mu(\alpha), S_{x,\tau}^{-1}g \right \rangle= \left\langle S_{x,\tau}S_{x,\tau}^{-1}h, S_{x,\tau}^{-1}g\right \rangle=\left\langle S_{x,\tau}^{-1}h, g\right \rangle.
  \end{align*}
Thus the frame operator for the canonical dual $(\{\widetilde{x}_\alpha\}_{\alpha \in \Omega},   \{\widetilde{\tau}_\alpha\}_{\alpha\in \Omega} )$ is  $ S_{x,\tau}^{-1}.$ Therefore, its canonical dual is $(\{ S_{x,\tau}S_{x,\tau}^{-1}x_\alpha\}_{\alpha \in \Omega},   \{ S_{x,\tau} S_{x,\tau}^{-1}\tau_\alpha\}_{\alpha\in \Omega} ).$ Others can be proved as in the  earlier  consideration `continuous operator-valued frame'.
  \end{proof}
 \begin{proposition}
 Let  $ (\{x_\alpha\}_{\alpha\in \Omega}, \{\tau_\alpha\}_{\alpha\in \Omega}) $ and $ (\{y_\alpha\}_{\alpha\in \Omega}, \{\omega_\alpha\}_{\alpha\in \Omega}) $ be  continuous frames for   $\mathcal{H}$. Then the following are equivalent.
 \begin{enumerate}[\upshape(i)]
 \item $ (\{y_\alpha\}_{\alpha\in \Omega}, \{\omega_\alpha\}_{\alpha\in \Omega}) $ is  dual of $ (\{x_\alpha\}_{\alpha\in \Omega}, \{\tau_\alpha\}_{\alpha\in \Omega}) $. 
 \item $ \int_{\Omega}\langle h, x_\alpha \rangle \omega_\alpha \,d \mu(\alpha)= \int_{\Omega}\langle h, \tau_\alpha\rangle y_\alpha \,d \mu(\alpha)=h, \forall h \in  \mathcal{H}.$ 
 \end{enumerate}
 \end{proposition}
 \begin{proof}
 $\langle \theta_\omega^*\theta_xh, g \rangle =\langle \theta_xh, \theta_\omega g \rangle=\int_{\Omega}\theta_xh(\alpha)\overline{\theta_\omega g(\alpha)}\,d \mu(\alpha) =\int_{\Omega}\langle h, x_\alpha\rangle \langle \omega_\alpha, g\rangle \,d \mu(\alpha)=\langle \int_{\Omega}\langle h, x_\alpha \rangle \omega_\alpha \,d \mu(\alpha), g\rangle, $ $\forall h, g \in \mathcal{H}$. Similarly $\langle \theta_y^*\theta_\tau h, g \rangle= \langle \int_{\Omega}\langle h, \tau_\alpha\rangle y_\alpha \,d \mu(\alpha), g\rangle $,  $\forall h, g \in \mathcal{H}$.
 \end{proof}
 \begin{theorem}
  If $ (\{x_\alpha\}_{\alpha\in \Omega}, \{\tau_\alpha\}_{\alpha\in \Omega})$ is a Riesz continuous  frame  for   $\mathcal{H}$, then it has unique dual. 
 \end{theorem}
 \begin{proof}
 Let $ (\{y_\alpha\}_{\alpha\in \Omega}, \{\omega_\alpha\}_{\alpha\in \Omega})$ and  $ (\{z_\alpha\}_{\alpha\in \Omega}, \{\rho_\alpha\}_{\alpha\in \Omega})$ be  dual continuous frames  of $ (\{x_\alpha\}_{\alpha\in \Omega}, \{\tau_\alpha\}_{\alpha\in \Omega})$. Then $\theta_x^*\theta_\omega=\theta_\tau^*\theta_y=\theta_x^*\theta_\rho=\theta_\tau^*\theta_z=I_\mathcal{H}$ $\Rightarrow$ $\theta_x^*(\theta_\omega-\theta_\rho)=0=\theta_\tau^*(\theta_y-\theta_z)$ $\Rightarrow$ $I_\mathcal{H}(\theta_\omega-\theta_\rho)=P_{\tau,x}(\theta_\omega-\theta_\rho)=\theta_\tau S_{x,\tau}^{-1}\theta_x^*(\theta_\omega-\theta_\rho)=0=\theta_xS_{x,\tau}^{-1}\theta_\tau^*(\theta_y-\theta_z)=P_{x,\tau}(\theta_y-\theta_z)=I_\mathcal{H}(\theta_y-\theta_z)$  $\Rightarrow$ $\theta_\omega=\theta_\rho$, $\theta_y=\theta_z$ $\Rightarrow$ $\langle h, \omega_\alpha\rangle=\theta_\omega h(\alpha)=\theta_\rho h(\alpha)=\langle h, \rho_\alpha\rangle$, $\langle h,y_\alpha \rangle=\theta_y h(\alpha)=\theta_z h(\alpha)=\langle h,z_\alpha\rangle$, $\forall h \in \mathcal{H}, \forall \alpha \in \Omega$ $\Rightarrow$ $\omega_\alpha=\rho_\alpha, y_\alpha=z_\alpha, \forall \alpha \in \Omega$. Hence the dual of $ (\{x_\alpha\}_{\alpha\in \Omega}, \{\tau_\alpha\}_{\alpha\in \Omega})$ is unique.
 \end{proof}
 
 \begin{proposition}
 Let $ (\{x_\alpha\}_{\alpha\in \Omega}, \{\tau_\alpha\}_{\alpha\in \Omega}) $  be a  continuous frame for   $\mathcal{H}$. If $ (\{y_\alpha\}_{\alpha\in \Omega}, \{\omega_\alpha\}_{\alpha\in \Omega}) $ is a dual of $ (\{x_\alpha\}_{\alpha\in \Omega}, \{\tau_\alpha\}_{\alpha\in \Omega}) $, then there exist continuous Bessel  $ \{z_\alpha\}_{\alpha\in \Omega}$ and $ \{\rho_\alpha\}_{\alpha\in \Omega} $ (w.r.t. themselves) for  $\mathcal{H}$ such that $ y_\alpha=S_{x,\tau}^{-1}x_\alpha+z_\alpha, \omega_\alpha=S_{x,\tau}^{-1}\tau_\alpha+\rho_\alpha,\forall \alpha \in \Omega$,  and $\theta_z(\mathcal{H})\perp \theta_\tau(\mathcal{H}),\theta_\rho(\mathcal{H})\perp \theta_x(\mathcal{H})$. Converse holds if  $ \theta_\rho^*\theta_z \geq 0$.
 \end{proposition}
 \begin{proof}
 Similar to the proof of Proposition 8.28 in \cite{MAHESHKRISHNASAMJOHNSON}.
 \end{proof}
 
\begin{proposition}\label{DUALCHARATERIZATIONLEMMA2}
 Let $ (\{x_\alpha\}_{\alpha\in \Omega}, \{\tau_\alpha\}_{\alpha\in \Omega}) $  be a  continuous frame for   $\mathcal{H}$. Then the bounded left-inverses of 
\begin{enumerate}[\upshape(i)]
\item $ \theta_x$ are precisely $S_{x,\tau}^{-1}\theta_\tau^*+U(I_{\mathcal{L}^2(\Omega, \mathbb{K})}-\theta_xS_{x,\tau}^{-1}\theta_\tau^*)$, where $U\in \mathcal{B}(\mathcal{L}^2(\Omega, \mathbb{K}), \mathcal{H})$.
\item $ \theta_\tau$ are precisely $S_{x,\tau}^{-1}\theta_x^*+V(I_{\mathcal{L}^2(\Omega, \mathbb{K})}-\theta_\tau S_{x,\tau}^{-1}\theta_x^*)$, where $V\in \mathcal{B}( \mathcal{L}^2(\Omega, \mathbb{K}),\mathcal{H})$. 
\end{enumerate}	
 \end{proposition}
 \begin{proof}
Similar to the proof of Lemma 8.30 in \cite{MAHESHKRISHNASAMJOHNSON}.
\end{proof}

\begin{definition}
A continuous frame   $(\{y_\alpha\}_{\alpha\in \Omega},  \{\omega_\alpha\}_{\alpha\in \Omega})$  for  $\mathcal{H}$ is said to be orthogonal to a continuous frame   $( \{x_\alpha\}_{\alpha\in \Omega}, \{\tau_\alpha\}_{\alpha\in \Omega})$ for $\mathcal{H}$ if $\theta_\omega^*\theta_x= \theta_y^*\theta_\tau=0.$
\end{definition}

\begin{proposition}
Let  $ (\{x_\alpha\}_{\alpha\in \Omega}, \{\tau_\alpha\}_{\alpha\in \Omega}) $, $ (\{y_\alpha\}_{\alpha\in \Omega}, \{\omega_\alpha\}_{\alpha\in \Omega})$ be continuous  frames for  $\mathcal{H}$. Then the following are equivalent.
\begin{enumerate}[\upshape(i)]
\item $(\{y_\alpha\}_{\alpha\in \Omega},\{\omega_\alpha\}_{\alpha\in \Omega}) $ is orthogonal to  $( \{x_\alpha\}_{\alpha\in \Omega},  \{\tau_\alpha\}_{\alpha\in \Omega}) $.
\item $  \int_{ \Omega}\langle h, x_\alpha\rangle \omega_\alpha\,d\mu(\alpha)=0=\int_{ \Omega}\langle g, \tau_\alpha\rangle y_\alpha\,d\mu(\alpha), \forall h \in  \mathcal{H}$. 
\end{enumerate}
\end{proposition}
\begin{proposition}
Two orthogonal continuous frames  have a common dual continuous frame.	
\end{proposition}
\begin{proof}
Similar to the proof of Proposition 8.34 in \cite{MAHESHKRISHNASAMJOHNSON}.
\end{proof}

\begin{proposition}
Let $ (\{x_\alpha\}_{\alpha\in  \Omega}, \{\tau_\alpha\}_{\alpha\in  \Omega}) $ and $ (\{y_\alpha\}_{\alpha\in  \Omega}, \{\omega_\alpha\}_{\alpha\in  \Omega}) $ be  two Parseval continuous  frames for  $\mathcal{H}$ which are  orthogonal. If $A,B,C,D \in \mathcal{B}(\mathcal{H})$ are such that $ AC^*+BD^*=I_\mathcal{H}$, then  $ (\{Ax_\alpha+By_\alpha\}_{\alpha\in  \Omega}, \{C\tau_\alpha+D\omega_\alpha\}_{\alpha\in  \Omega}) $ is a Parseval continuous    frame for  $\mathcal{H}$. In particular,  if scalars $ a,b,c,d$ satisfy $a\bar{c}+b\bar{d} =1$, then $ (\{ax_\alpha+by_\alpha\}_{\alpha\in  \Omega}, \{c\tau_j+d\omega_\alpha\}_{\alpha\in  \Omega}) $ is a  Parseval continuous  frame for  $\mathcal{H}$.
\end{proposition} 
\begin{proof}
For all $h \in \mathcal{H}$ and $\alpha \in \Omega$ we see $ \theta_{Ax+By} h (\alpha)= \langle h,Ax_\alpha+By_\alpha \rangle=\langle A^*h,x_\alpha \rangle+\langle B^*h,y_\alpha \rangle=\theta_x(A^*h)(\alpha)+\theta_y(B^*h)(\alpha)=(\theta_xA^*+\theta_yB^*)h(\alpha).$ Similarly $ \theta_{C\tau+D\omega}=\theta_\tau C^*+\theta_\omega D^* $. Other arguments are similar to that in the proof of Proposition 8.35 in \cite{MAHESHKRISHNASAMJOHNSON}.
\end{proof} 
 
\begin{definition}
Two continuous frames  $(\{x_\alpha\}_{\alpha\in \Omega},  \{\tau_\alpha\}_{\alpha\in  \Omega}) $ and $(\{y_\alpha\}_{\alpha\in  \Omega},\{\omega_\alpha\}_{\alpha\in  \Omega})$  for $ \mathcal{H}$ are called disjoint if $(\{x_\alpha\oplus y_\alpha\}_{\alpha\in  \Omega},\{\tau_\alpha\oplus\omega_\alpha\}_{\alpha\in  \Omega})$ is a continuous  frame for $\mathcal{H}\oplus\mathcal{H} $.
\end{definition} 
\begin{proposition}\label{SEQUENTIALDISJOINTFRAMEPROPOSITION}
If $(\{x_\alpha\}_{\alpha\in  \Omega},\{\tau_\alpha\}_{\alpha\in  \Omega} )$  and $ (\{y_\alpha\}_{\alpha\in  \Omega}, \{\omega_\alpha\}_{\alpha\in  \Omega} )$  are  orthogonal continuous  frames  for $\mathcal{H}$, then  they  are disjoint. Further, if both $(\{x_\alpha\}_{\alpha\in  \Omega},\{\tau_\alpha\}_{\alpha\in  \Omega} )$  and $ (\{y_\alpha\}_{\alpha\in  \Omega}, \{\omega_\alpha\}_{\alpha\in  \Omega} )$ are  Parseval, then $(\{x_\alpha\oplus y_\alpha\}_{\alpha \in  \Omega},\{\tau_\alpha\oplus \omega_\alpha\}_{\alpha \in  \Omega})$ is Parseval.
\end{proposition} 
\begin{proof}
For all $h\oplus g \in \mathcal{H}\oplus\mathcal{H} $, $\theta_{x\oplus y}(h\oplus g)(\alpha)=\langle h\oplus g,x_\alpha\oplus y_\alpha \rangle =\langle h,x_\alpha \rangle+\langle  g, y_\alpha \rangle=\theta_xh(\alpha)+\theta_yg(\alpha)=(\theta_xh+\theta_yg)(\alpha)$ and for all $f \in \mathcal{L}^2(\Omega, \mathbb{K})$, $\langle \theta_{\tau\oplus \omega}^*f, h\oplus g\rangle =\langle f,\theta_{\tau\oplus \omega}(h\oplus g)\rangle=\langle \theta_\tau^*f, h \rangle+\langle\theta_\omega^*f,  g \rangle=\langle \theta_\tau^*f\oplus\theta_\omega^*f, h\oplus g\rangle .$ Thus  $S_{x\oplus y,\tau\oplus\omega}(h\oplus g)=\theta^*_{\tau\oplus\omega}\theta_{x\oplus y}(h\oplus g)=\theta^*_{\tau\oplus\omega}(\theta_xh+\theta_yg)=\theta_\tau^*(\theta_xh+\theta_yg)\oplus\theta_\omega^*(\theta_xh+\theta_yg)=(S_{x ,\tau}+0)\oplus (0+S_{ y,\omega})=S_{x ,\tau}\oplus S_{ y,\omega}$, which is bounded positive invertible with $S_{x\oplus y,\tau\oplus\omega}^{-1}=S_{x ,\tau}^{-1}\oplus S_{ y,\omega}^{-1}$. 
\end{proof}
  
\textbf{Characterization} 
\begin{theorem}\label{SEQUENTIALCHARACTERIZATIONHILBERT2}
Let $\{x_\alpha\}_{\alpha\in\Omega},\{\tau_\alpha\}_{\alpha\in\Omega}$ be in $ \mathcal{H}$ such that for each  $h \in \mathcal{H}$, both    maps   $\Omega \ni \alpha \mapsto\langle  h, x_\alpha\rangle \in \mathbb{K}$, $\Omega \ni\alpha \mapsto \langle  h, \tau_\alpha \rangle \in\mathbb{K}$ are measurable. Then $(\{x_\alpha\}_{\alpha\in\Omega},\{\tau_\alpha\}_{\alpha\in\Omega})$  is a continuous frame with bounds $ a $ and $ b$ (resp. Bessel with bound $ b$)
 \begin{enumerate}[\upshape(i)]
\item  if and only if 
$$U: \mathcal{L}^2(\Omega, \mathbb{K})\ni f\mapsto\int_{\Omega}f(\alpha)x_\alpha \,d \mu(\alpha) \in \mathcal{H}, ~\text{and} ~ V:\mathcal{L}^2(\Omega, \mathbb{K})\ni g\mapsto\int_{\Omega}g(\alpha)\tau_\alpha \,d \mu(\alpha) \in \mathcal{H} $$ 
are well-defined, $ U,V \in  \mathcal{B}(\mathcal{L}^2(\Omega, \mathbb{K}), \mathcal{H})$ such that $ aI_\mathcal{H}\leq VU^*\leq bI_\mathcal{H}$   (resp. $ 0\leq VU^*\leq bI_\mathcal{H}).$
\item  if and only if 
$$U: \mathcal{L}^2(\Omega, \mathbb{K})\ni f\mapsto\int_{\Omega}f(\alpha)x_\alpha \,d \mu(\alpha) \in \mathcal{H}, ~\text{and} ~ S: \mathcal{H} \ni x\mapsto Sx \in \mathcal{L}^2(\Omega, \mathbb{K}),~ Sx: \Omega \ni \alpha \mapsto \langle x,\tau_\alpha\rangle   \in \mathbb{K} $$ 
are well-defined, $ U \in  \mathcal{B}(\mathcal{L}^2(\Omega, \mathbb{K}), \mathcal{H})$, $ S \in  \mathcal{B}(\mathcal{H}, \mathcal{L}^2(\Omega, \mathbb{K}))$  such that  $ aI_\mathcal{H}\leq S^*U^*\leq bI_\mathcal{H}$ (resp.  $ 0\leq S^*U^*\leq bI_\mathcal{H}).$
\item   if and only if 
$$ R: \mathcal{H} \ni h\mapsto  Rh \in \mathcal{L}^2(\Omega, \mathbb{K}),Rh: \Omega \ni \alpha \mapsto \langle h, x_\alpha \rangle \in \mathbb{K}, ~\text{and} ~  V:\mathcal{L}^2(\Omega, \mathbb{K})\ni g\mapsto\int_{\Omega}g(\alpha)\tau_\alpha \,d \mu(\alpha) \in \mathcal{H} $$ 
are well-defined, $ R \in  \mathcal{B}(\mathcal{H}, \mathcal{L}^2(\Omega, \mathbb{K}))$, $ V \in  \mathcal{B}(\mathcal{L}^2(\Omega, \mathbb{K}), \mathcal{H})$ such that  $ aI_\mathcal{H}\leq VR\leq bI_\mathcal{H}$  (resp.  $ 0\leq VR\leq bI_\mathcal{H} ).$
\item  if and only if 
$$ R: \mathcal{H} \ni h\mapsto  Rh \in \mathcal{L}^2(\Omega, \mathbb{K}),Rh: \Omega \ni \alpha \mapsto \langle h, x_\alpha \rangle \in \mathbb{K} ,~\text{and} ~  S: \mathcal{H} \ni x\mapsto Sx\in \mathcal{L}^2(\Omega, \mathbb{K}),~ Sx: \Omega \ni \alpha \mapsto \langle x, \tau_\alpha \rangle \in \mathbb{K} $$
are well-defined, $ R,S \in  \mathcal{B}(\mathcal{H}, \mathcal{L}^2(\Omega, \mathbb{K}))$ such that  $ aI_\mathcal{H}\leq S^*R\leq bI_\mathcal{H}$ (resp.   $ 0\leq S^*R\leq bI_\mathcal{H}).$   
\end{enumerate}
\end{theorem}
\begin{proof}
We prove the first one for continuous Bessel, others are similar. 
 
$(\Rightarrow)$ $U=\theta_x^*$, $V=\theta_\tau^*$ and $VU^*=\theta_\tau^*\theta_x=S_{x,\tau}$. $ (\Leftarrow)$ $\theta_x=U^*$, $\theta_\tau=V^*$ and $S_{x,\tau}=\theta_\tau^*\theta_x=VU^*$.
\end{proof}
 
\textbf{Similarity}
\begin{definition}
 A continuous frame  $ (\{y_\alpha\}_{\alpha\in  \Omega},\{\omega_\alpha\}_{\alpha\in  \Omega})$ for $ \mathcal{H}$ is said to be  similar to a continuous frame $ (\{x_\alpha\}_{\alpha\in  \Omega},\{\tau_\alpha\}_{\alpha\in  \Omega})$ for $ \mathcal{H}$ if there are invertible operators $ T_{x,y}, T_{\tau,\omega} \in \mathcal{B}(\mathcal{H})$ such that $ y_\alpha=T_{x,y}x_\alpha, \omega_\alpha=T_{\tau,\omega}\tau_\alpha,   \forall \alpha \in  \Omega.$
 \end{definition}     
   
 \begin{proposition}
 Let $ \{x_\alpha\}_{\alpha\in  \Omega}\in \mathscr{F}_\tau$  with frame bounds $a, b,$  let $T_{x,y} , T_{\tau,\omega}\in \mathcal{B}(\mathcal{H})$ be positive, invertible, commute with each other, commute with $ S_{x, \tau}$, and let $y_\alpha=T_{x,y}x_\alpha , \omega_\alpha=T_{\tau,\omega}\tau_\alpha,  \forall \alpha \in  \Omega.$ Then 
 \begin{enumerate}[\upshape(i)]
 \item $ \{y_\alpha\}_{\alpha\in  \Omega}\in \mathscr{F}_\tau$ and $ \frac{a}{\|T_{x,y}^{-1}\|\|T_{\tau,\omega}^{-1}\|}\leq S_{y, \omega} \leq b\|T_{x,y}T_{\tau,\omega}\|.$ Assuming that $ (\{x_\alpha\}_{\alpha\in  \Omega},\{\tau_\alpha\}_{\alpha\in  \Omega})$ is Parseval, then $(\{y_\alpha\}_{\alpha\in  \Omega},  \{\omega_\alpha\}_{\alpha\in  \Omega})$ is Parseval  if and only if   $ T_{\tau, \omega}T_{x,y}=I_\mathcal{H}.$  
 \item $ \theta_y=\theta_x T_{x,y}, \theta_\omega=\theta_\tau T_{\tau,\omega}, S_{y,\omega}=T_{\tau,\omega}S_{x, \tau}T_{x,y},  P_{y,\omega} =P_{x, \tau}.$
 \end{enumerate}
 \end{proposition}     
 \begin{proof}
 For all $ h, g \in \mathcal{H},$ 
 
 \begin{align*}
 &\langle T_{\tau,\omega}S_{x,\tau}T_{x,y}h, g\rangle=\langle S_{x,\tau}T_{x,y}h, T_{\tau,\omega}^*g\rangle=\int_{\Omega}\langle  T_{x,y}h, x_\alpha \rangle \langle \tau_\alpha, T_{\tau,\omega}^*g\rangle \,d\mu(\alpha)\\
 &=\int_{\Omega}\langle  h, T_{x,y} x_\alpha \rangle \langle T_{\tau,\omega}\tau_\alpha, g\rangle \,d\mu(\alpha)=\int_{\Omega}\langle h,y_\alpha \rangle \langle \omega_\alpha , g \rangle \,d\mu(\alpha)=\langle S_{y,\omega}h, g\rangle.
 \end{align*}
 \end{proof}     
 \begin{lemma}\label{SEQUENTIALSIMILARITYLEMMA}
 Let $ \{x_\alpha\}_{\alpha\in  \Omega}\in \mathscr{F}_\tau,$ $ \{y_\alpha\}_{\alpha\in  \Omega}\in \mathscr{F}_\omega$ and   $y_\alpha=T_{x, y}x_\alpha , \omega_\alpha=T_{\tau,\omega}\tau_\alpha,  \forall \alpha \in  \Omega$, for some invertible $T_{x,y}, T_{\tau,\omega}\in \mathcal{B}(\mathcal{H}).$ Then 
 $ \theta_y=\theta_x T^*_{x,y}, \theta_\omega=\theta_\tau T^*_{\tau,\omega}, S_{y,\omega}=T_{\tau,\omega}S_{x, \tau}T_{x,y}^*,  P_{y,\omega}=P_{x, \tau}.$ Assuming that $ (\{x_\alpha\}_{\alpha\in  \Omega},\{\tau_\alpha\}_{\alpha\in  \Omega})$ is Parseval, then $ (\{y_\alpha\}_{\alpha\in  \Omega},\{\omega_\alpha\}_{\alpha\in  \Omega})$ is Parseval if and only if $T_{\tau,\omega}T_{x,y}^*=I_\mathcal{H}.$
 \end{lemma}     
 \begin{proof}
 $\theta_yh(\alpha)=\langle  h, y_\alpha \rangle =\langle  h, T_{x, y}x_\alpha \rangle=\langle   T_{x, y}^*h,x_\alpha \rangle=\theta_x( T_{x, y}^*h)(\alpha)$ $\Rightarrow$ $ \theta_y=\theta_x T^*_{x,y}$. Similarly $\theta_\omega=\theta_\tau T^*_{\tau,\omega}$.	
 \end{proof}  
 \begin{theorem}\label{SEQUENTIALSIMILARITYCHARACTERIZATION}
 Let $ \{x_\alpha\}_{\alpha\in  \Omega}\in \mathscr{F}_\tau,$ $ \{y_\alpha\}_{\alpha\in  \Omega}\in \mathscr{F}_\omega.$ The following are equivalent.
 \begin{enumerate}[\upshape(i)]
 \item $y_\alpha=T_{x,y}x_\alpha , \omega_\alpha=T_{\tau, \omega}\tau_\alpha ,  \forall \alpha \in  \Omega,$ for some invertible  $ T_{x,y} ,T_{\tau, \omega} \in \mathcal{B}(\mathcal{H}). $
 \item $\theta_y=\theta_x{T'}_{x,y}^* , \theta_\omega=\theta_\tau {T'}_{\tau, \omega}^* $ for some invertible  $ {T'}_{x,y} ,{T'}_{\tau, \omega} \in \mathcal{B}(\mathcal{H}). $
 \item $P_{y,\omega}=P_{x,\tau}.$
 \end{enumerate}
 If one of the above conditions is satisfied, then  invertible operators in  $\operatorname{(i)}$ and  $\operatorname{(ii)}$ are unique and are given by  $T_{x,y}=\theta_y^*\theta_\tau S_{x,\tau}^{-1}, T_{\tau, \omega}=\theta_\omega^*\theta_x S_{x,\tau}^{-1}.$
 In the case that $(\{x_\alpha\}_{\alpha\in  \Omega},  \{\tau_\alpha\}_{\alpha\in  \Omega})$ is Parseval, then $(\{y_\alpha\}_{\alpha\in \Omega},  \{\omega_\alpha\}_{\alpha\in   \Omega})$ is  Parseval if and only if $T_{\tau, \omega}T_{x,y}^* =I_\mathcal{H}$   if and only if $ T_{x,y}^*T_{\tau, \omega} =I_\mathcal{H}$. 
 \end{theorem}
 \begin{proof}
 (ii) $\Rightarrow$ (i) For all $h \in \mathcal{H}$ and $\alpha \in \Omega$, $\langle h, y_\alpha\rangle =\theta_yh(\alpha)=\theta_x({T'}_{x,y}^*h)(\alpha)=\langle {T'}_{x,y}^*h, y_\alpha\rangle =\langle h, {T'}_{x,y} y_\alpha\rangle$ which implies $y_\alpha={T'}_{x,y}x_\alpha , \omega_\alpha={T'}_{\tau, \omega}\tau_\alpha ,  \forall \alpha \in  \Omega$. Other arguments are similar to that in the proof of Theorem 8.45 in \cite{MAHESHKRISHNASAMJOHNSON}.
 \end{proof}
 \begin{corollary}
 For any given continuous frame $ (\{x_\alpha\}_{\alpha \in  \Omega} , \{\tau_\alpha\}_{\alpha \in  \Omega})$, the canonical dual of $ (\{x_\alpha\}_{\alpha \in  \Omega} $, $ \{\tau_\alpha\}_{\alpha \in  \Omega})$ is the only dual continuous frame that is similar to $ (\{x_\alpha\}_{\alpha \in  \Omega} , \{\tau_\alpha\}_{\alpha \in  \Omega} )$.
 \end{corollary}
 \begin{proof}
 Similar to the proof of Corollary 8.46 in	\cite{MAHESHKRISHNASAMJOHNSON}.
 \end{proof}
 \begin{corollary}
 Two similar  continuous frames cannot be orthogonal.
 \end{corollary}
 \begin{proof}
Similar to the proof of Corollary 8.47 in \cite{MAHESHKRISHNASAMJOHNSON}.
 \end{proof}
 \begin{remark}
 For every continuous frame  $(\{x_\alpha\}_{\alpha \in  \Omega}, \{\tau_\alpha\}_{\alpha \in  \Omega}),$ each  of `continuous frames'  $( \{S_{x, \tau}^{-1}x_\alpha\}_{\alpha \in  \Omega}$, $ \{\tau_\alpha\}_{\alpha \in  \Omega})$,    $(\{S_{x, \tau}^{-1/2}x_\alpha\}_{\alpha \in  \Omega}, \{S_{x,\tau}^{-1/2}\tau_\alpha\}_{\alpha \in  \Omega}),$ and  $ (\{x_\alpha \}_{\alpha \in  \Omega}, \{S_{x,\tau}^{-1}\tau_\alpha\}_{\alpha \in  \Omega})$ is a  Parseval continuous  frame  which is similar to  $ (\{x_\alpha\}_{\alpha \in  \Omega} , \{\tau_\alpha\}_{\alpha \in  \Omega}).$  Hence each continuous frame is similar to  Parseval continuous  frames.
\end{remark}

\textbf{Continuous frames and   representations of locally compact  groups}

Let $G$, $d\mu_G$, $\lambda$, $\rho$ be as in Section \ref{FRAMESANDDISCRETEGROUPREPRESENTATIONS} and $\mathcal{H}_0=\mathbb{K}$. We denote the von Neumann algebra generated by unitaries $\{\lambda_g\}_{g\in G} $ (resp. $\{\rho_g\}_{g\in G} $ )  in $ \mathcal{B}(\mathcal{L}^2(G,\mathbb{K}))$ by $\mathscr{L}(G) $ (resp. $\mathscr{R}(G) $). Then $\mathscr{L}(G)'=\mathscr{R}(G)$ and $\mathscr{R}(G)'=\mathscr{L}(G)$ \cite{TAKESAKI2}.
\begin{definition}
Let $ \pi$ be a unitary representation of a locally compact group $ G$ on a Hilbert space $ \mathcal{H}.$ An element  $ x$ in $ \mathcal{H}$ is called a continuous frame generator (resp. a  Parseval frame generator) w.r.t. $ \tau$ in $ \mathcal{H}$  if $(\{x_g\coloneqq  \pi_{g}x\}_{g\in G},\{\tau_g\coloneqq  \pi_{g}\tau\}_{g\in G}) $ is a  continuous frame  (resp. Parseval frame) for  $ \mathcal{H}$.
In this case we write $ (x,\tau)$ is a frame generator for $\pi$. 
\end{definition}
\begin{proposition}\label{REPRESENATIONLEMMAGROUP}
Let $ (x,\tau)$ and $ (y,\omega)$ be  frame generators    in $\mathcal{H}$ for a unitary representation $ \pi$ of  $G$ on $ \mathcal{H}.$ Then
\begin{enumerate}[\upshape(i)]
 \item $ \theta_x\pi_g=\lambda_g\theta_x,  \theta_\tau \pi_g=\lambda_g\theta_\tau,   \forall g \in G.$
 \item $ \theta_x^*\theta_y,   \theta_\tau^*\theta_\omega,\theta_x^*\theta_\omega$ are in the commutant $ \pi(G)'$ of $ \pi(G)''.$ Further, $ S_{x,\tau} \in \pi(G)'$ and $(S_{x, \tau}^{-1/2}x, S_{x, \tau}^{-1/2}\tau)$ is a Parseval frame generator. 
 \item $ \theta_xT\theta_\tau^*,\theta_xT\theta_y^*, \theta_\tau T\theta_\omega^* \in \mathscr{R}(G), \forall T \in \pi (G)'.$ In particular, $ P_{x, \tau} \in \mathscr{R}(G). $
 \end{enumerate}
 \end{proposition}   
  \begin{proof} 
  \begin{enumerate}[\upshape(i)]
  \item For all $h\in \mathcal{H}$ and $f\in \mathcal{L}^2(G, \mathbb{K})$,
  
  \begin{align*}
  \langle \lambda_g\theta_xh, f\rangle &=\int_{G}\lambda_g(\theta_xh)(\alpha)\overline{f(\alpha)}\,d\mu_G(\alpha)=\int_{G}(\theta_xh)(g^{-1}\alpha)\overline{f(\alpha)}\,d\mu_G(\alpha)\\
  &=\int_{G}\langle h, x_{g^{-1}\alpha}\rangle\overline{f(\alpha)}\,d\mu_G(\alpha)=\int_{G}\langle h, \pi_{g^{-1}\alpha}x\rangle\overline{f(\alpha)}\,d\mu_G(\alpha)
  \\
  &=\int_{G}\langle \pi_gh, \pi_{\alpha}x\rangle\overline{f(\alpha)}\,d\mu_G(\alpha)=\int_{G}\langle \pi_gh, x_\alpha\rangle\overline{f(\alpha)}\,d\mu_G(\alpha)\\
  &=\int_{G}\theta_x(\pi_gh)(\alpha)\overline{f(\alpha)}\,d\mu_G(\alpha)=\langle \theta_x\pi_gh, f\rangle.
  \end{align*}
  \item First part is similar to the proof of (ii) in Proposition 8.51 in \cite{MAHESHKRISHNASAMJOHNSON}. For second, let  $h, g \in \mathcal{H}$. Then 
  
  \begin{align*}
  \left \langle S_{S_{x, \tau}^{-\frac{1}{2}}x,S_{x, \tau}^{-\frac{1}{2}}\tau}h, g\right \rangle &=\int_{G}\langle h, \pi_\alpha S_{x, \tau}^{-\frac{1}{2}}x \rangle\langle \pi_\alpha S_{x, \tau}^{-\frac{1}{2}}\tau, g\rangle\,d\mu_G(\alpha)\\
  &=\int_{G}\langle h,  S_{x, \tau}^{-\frac{1}{2}} \pi_\alpha x \rangle\langle  S_{x, \tau}^{-\frac{1}{2}} \pi_\alpha\tau, g\rangle\,d\mu_G(\alpha)
 \\
 & =\int_{G}\langle   S_{x, \tau}^{-\frac{1}{2}}h,  \pi_\alpha x \rangle\langle \pi_\alpha\tau,  S_{x, \tau}^{-\frac{1}{2}}  g\rangle\,d\mu_G(\alpha)\\
  &=\langle S_{x, \tau}  (S_{x, \tau}^{-\frac{1}{2}}  h), S_{x, \tau}^{-\frac{1}{2}}  g\rangle=\langle h, g\rangle.
  \end{align*}
 \item Similar to the proof of (iii) in Proposition 8.51 in \cite{MAHESHKRISHNASAMJOHNSON}.
\end{enumerate}
\end{proof} 
\begin{theorem}\label{GROUPC1}
Let $ G$ be a locally compact group with identity $ e$ and $( \{x_g\}_{g\in G},  \{\tau_g\}_{g\in G})$ be a   Parseval continuous frame  for  $\mathcal{H}.$ Then there is a unitary representation $ \pi$  of $ G$ on  $ \mathcal{H}$  for which 
$$ x_g=\pi_gx_e, ~\tau_g=\pi_g\tau_e ,\quad \forall g \in G$$ 
 if and only if 
$$\langle x_{gp} , x_{gq}\rangle=\langle x_p, x_q \rangle, ~ \langle x_{gp},  \tau_{gq}\rangle=\langle x_p, \tau_q \rangle, ~\langle \tau_{gp} , \tau_{gq}\rangle=\langle \tau_p, \tau_q \rangle,  \quad \forall  g,p,q \in G.  $$
\end{theorem} 
\begin{proof} 
Proof 1. $(\Rightarrow)$   Similar to the proof of `only if' part of Theorem 8.52 in \cite{MAHESHKRISHNASAMJOHNSON}.
 
$(\Leftarrow)$ We state the   following three, among them we prove third, others are similar.
$$   \lambda_g\theta_x\theta_x^*=\theta_x\theta_x^*\lambda_g,~  \lambda_g\theta_x\theta_\tau^*=\theta_x\theta_\tau^*\lambda_g,~ \lambda_g\theta_\tau\theta_\tau^*=\theta_\tau\theta_\tau^*\lambda_g,\quad \forall g \in G.$$
For all $u,v \in \mathcal{L}^2(G,\mathbb{K})$,

\begin{align*}
\langle\lambda_g\theta_\tau\theta_\tau^*\lambda_g^*u ,v \rangle&=\langle\theta_\tau^*\lambda_g^*u ,\theta_\tau^*\lambda_g^*v \rangle=\langle\theta_\tau^*\lambda_{g^{-1}}u ,\theta_\tau^*\lambda_{g^{-1}}v \rangle\\
& =\left\langle \int_{G}\lambda_{g^{-1}}u(\alpha)\tau_\alpha\,d\mu_G(\alpha),\int_{G}\lambda_{g^{-1}}v(\beta)\tau_\beta\,d\mu_G(\beta)\right\rangle\\
&=\left\langle \int_{G}u(g\alpha)\tau_\alpha\,d\mu_G(\alpha),\int_{G}v(g\beta)\tau_\beta\,d\mu_G(\beta)\right\rangle\\
&=\left\langle \int_{G}u(p)\tau_{g^{-1}p}\,d\mu_G(g^{-1}p),\int_{G}v(q)\tau_{g^{-1}q}\,d\mu_G(g^{-1}q)\right\rangle\\
&=\left\langle \int_{G}u(p)\tau_{g^{-1}p}\,d\mu_G(p),\int_{G}v(q)\tau_{g^{-1}q}\,d\mu_G(q)\right\rangle\\
&=\int_{G}\int_{G}u(p)\overline{v(q)}\langle\tau_{g^{-1}p}, \tau_{g^{-1}q}\rangle \,d\mu_G(q)\,d\mu_G(p)\\
&=\int_{G}\int_{G}u(p)\overline{v(q)}\langle\tau_{p}, \tau_{q}\rangle \,d\mu_G(q)\,d\mu_G(p)\\
&=\left\langle \int_{G}u(p)\tau_p\,d\mu_G(p), \int_{G}v(q)\tau_q\,d\mu_G(q)\right\rangle 
=\langle\theta_\tau^*u ,\theta_\tau^* v \rangle=\langle\theta_\tau\theta_\tau^*u ,v \rangle.
\end{align*}

Define $ \pi : G \ni g  \mapsto \pi_g\coloneqq \theta_\tau^*\lambda_g\theta_x  \in \mathcal{B}(\mathcal{H}).$ Using the Parsevalness of given frame, we get   $ \pi_g\pi_h=\theta_\tau^*\lambda_g \theta_x \theta_\tau^*\lambda_h \theta_x =\theta_\tau^*\lambda_g \lambda_h\theta_x \theta_\tau^* \theta_x = \theta_\tau^*\lambda_g \lambda_h\theta_x = \theta_\tau^*\lambda_{gh} \theta_x =\pi_{gh}$ for all $g, h \in G,$ and $\pi_g\pi_g^*=\theta_\tau^*\lambda_g\theta_x\theta_x^*\lambda_{g^{-1}}\theta_\tau$ $=\theta_\tau^*\theta_x\theta_x^*\lambda_g\lambda_{g^{-1}}\theta_\tau=I_\mathcal{H},  \pi_g^*\pi_g= \theta_x^*\lambda_{g^{-1}}\theta_\tau\theta_\tau^*\lambda_g\theta_x= \theta_x^*\lambda_{g^{-1}}\lambda_g\theta_\tau\theta_\tau^*\theta_x=I_\mathcal{H} $ for all $ g \in G$. To prove $ \pi$ is a unitary representation we use the same idea used in the proof of Theorem \ref{gc1}. Let $h \in \mathcal{H}$ be fixed. Then $\theta_xh $ is fixed. Since $\lambda$ is a unitary  representation, the map $G \ni g \mapsto \lambda_g(\theta_xh) \in \mathcal{L}^2(G,\mathbb{K}) $ is continuous. Continuity of $\theta_\tau^*$ now gives that the map $G \ni g \mapsto \theta_\tau^*(\lambda_g(\theta_xh)) \in \mathcal{H} $ is continuous.   We now establish   $ x_g=\pi_gx_e, \tau_g=\pi_g\tau_e  $ for all $ g \in G$. For all $h \in \mathcal{H}$,

\begin{align*}
\langle \pi_gx_e,h \rangle &=\langle \theta_\tau^*\lambda_g\theta_xx_e,h \rangle\\
&=\left\langle\int_{G}\lambda_g\theta_xx_e (\alpha)\tau_\alpha\,d\mu_G(\alpha), h\right\rangle =\left\langle\int_{G}\theta_xx_e (g^{-1}\alpha)\tau_\alpha\,d\mu_G(\alpha), h\right\rangle\\
&=\int_{G}\langle \theta_xx_e (g^{-1}\alpha)\tau_\alpha,h \rangle \,d\mu_G(\alpha)=\int_{G}\langle \langle x_e,x_{g^{-1}\alpha} \rangle  \tau_\alpha,h \rangle \,d\mu_G(\alpha)\\
&=\int_{G} \langle x_{g^{-1}g},x_{g^{-1}\alpha} \rangle  \langle \tau_\alpha,h \rangle \,d\mu_G(\alpha)
=\int_{G} \langle x_{g},x_{\alpha} \rangle  \langle \tau_\alpha,h \rangle \,d\mu_G(\alpha)\\
&=\left\langle\int_{G}\langle x_{g},x_{\alpha} \rangle \tau_\alpha\,d\mu_G(\alpha) ,h \right\rangle=\langle x_g, h\rangle ,
\end{align*}
and

\begin{align*}
\langle \pi_g\tau_e,h \rangle &=\langle \theta_\tau^*\lambda_g\theta_x\tau_e,h \rangle\\
&=\left\langle\int_{G}\lambda_g\theta_x\tau_e (\alpha)\tau_\alpha\,d\mu_G(\alpha), h\right\rangle =\left\langle\int_{G}\theta_x\tau_e (g^{-1}\alpha)\tau_\alpha\,d\mu_G(\alpha), h\right\rangle\\
&=\int_{G}\langle \theta_x\tau_e (g^{-1}\alpha)\tau_\alpha,h \rangle \,d\mu_G(\alpha)=\int_{G}\langle \langle \tau_e,x_{g^{-1}\alpha} \rangle  \tau_\alpha,h \rangle \,d\mu_G(\alpha)\\
&=\int_{G} \langle \tau_{g^{-1}g},x_{g^{-1}\alpha} \rangle  \langle \tau_\alpha,h \rangle \,d\mu_G(\alpha)
=\int_{G} \langle \tau_{g},x_{\alpha} \rangle  \langle \tau_\alpha,h \rangle \,d\mu_G(\alpha)\\
&=\left\langle\int_{G}\langle \tau_{g},x_{\alpha} \rangle \tau_\alpha\,d\mu_G(\alpha) ,h \right\rangle=\langle \tau_g, h\rangle .
\end{align*}
Proof 2. Define $A_g: \mathcal{H} \ni h \mapsto \langle h, x_g \rangle \in \mathbb{K} $, $\Psi_g: \mathcal{H} \ni h \mapsto \langle h, \tau_g \rangle \in \mathbb{K}, \forall g \in G $. Then, from Theorem \ref{OVFTOSEQUENCEANDVICEVERSATHEOREM},  $(\{x_g\}_{g\in G}, \{\tau_g\}_{g\in G})$ is a continuous frame for  $\mathcal{H}$ if and only if  $(\{A_g\}_{g\in G}, \{\Psi_g\}_{g\in G})$ is a continuous (ovf)  in $\mathcal{B}(\mathcal{H},\mathbb{K})$. Further, from the proof of Theorem \ref{OVFTOSEQUENCEANDVICEVERSATHEOREM}, we also see that $(\{x_g\}_{g\in G}, \{\tau_g\}_{g\in G})$ is a Parseval continuous frame  if and only if  $(\{A_g\}_{g\in G}, \{\Psi_g\}_{g\in G})$ is a Parseval continuous (ovf). Now applying Theorem \ref{gc1} to the  Parseval continuous (ovf) $(\{A_g\}_{g\in G}, \{\Psi_g\}_{g\in G})$ yields - there is a unitary representation $ \pi$  of $ G$ on  $ \mathcal{H}$  for which 

\begin{align}\label{PROOF2FIRSTEQUATION}
 A_g=A_e\pi_{g^{-1}}, ~\Psi_g=\Psi_e\pi_{g^{-1}}, \quad\forall  g \in G
\end{align}
if and only if 

\begin{align}\label{PROOF2SECONDEQUATION}
A_{gp}A_{gq}^*=A_pA_q^* ,~ A_{gp}\Psi_{gq}^*=A_p\Psi_q^*,~ \Psi_{gp}\Psi_{gq}^*=\Psi_p\Psi_q^*, \quad \forall g,p,q \in G.
\end{align}
But Equation (\ref{PROOF2FIRSTEQUATION}) holds if and only if 

\begin{align*}
&\langle h, x_g \rangle=A_gh=A_e\pi_{g^{-1}}h=\langle \pi_{g^{-1}}h, x_e \rangle=\langle h, \pi_gx_e \rangle, \\
&\langle h, \tau_g \rangle=\Psi_gh=\Psi_e\pi_{g^{-1}}h=\langle \pi_{g^{-1}}h, \tau_g \rangle=\langle h, \pi_g\tau_e \rangle, \quad\forall  g \in G, \forall h \in \mathcal{H}\\
&\iff x_g=\pi_gx_e, ~\tau_g=\pi_g\tau_e ,\quad \forall g \in G.
\end{align*}
Also, Equation (\ref{PROOF2SECONDEQUATION}) holds if and only if 

\begin{align*}
&\langle \alpha x_{gq}, x_{gp}\rangle =A_{gp}A_{gq}^*\alpha=A_pA_q^*\alpha=\langle \alpha x_{q}, x_{p}\rangle,\\
&\langle \alpha \tau_{gq}, x_{gp}\rangle=A_{gp}\Psi_{gq}^*\alpha=A_p\Psi_q^*\alpha=\langle \alpha \tau_{q}, x_{p}\rangle,\\
& \langle \alpha \tau_{gq}, \tau_{gp}\rangle=\Psi_{gp}\Psi_{gq}^*\alpha=\Psi_p\Psi_q^*\alpha=\langle \alpha \tau_{q}, \tau_{p}\rangle  , \quad\forall \alpha \in \mathbb{K}\\
&\iff \langle x_{gp} , x_{gq}\rangle=\langle x_p, x_q \rangle, ~ \langle x_{gp},  \tau_{gq}\rangle=\langle x_p, \tau_q \rangle, ~\langle \tau_{gp} , \tau_{gq}\rangle=\langle \tau_p, \tau_q \rangle,  \quad\forall  g,p,q \in G.
\end{align*}
\end{proof}
\begin{corollary}
Let $ G$ be a locally compact  group with identity $ e$ and $( \{x_g\}_{g\in G},  \{\tau_g\}_{g\in G})$ be a  continuous frame  for  $\mathcal{H}.$ Then there is a unitary representation $ \pi$  of $ G$ on  $ \mathcal{H}$  for which
\begin{enumerate}[\upshape(i)]
\item  $ x_g=S_{x,\tau}\pi_gS_{x,\tau}^{-1}x_e, \tau_g=\pi_g\tau_e $ for all $ g \in G$  if and only if $\langle S_{x,\tau}^{-2} x_{gp} , x_{gq}\rangle=\langle S_{x,\tau}^{-2} x_p, x_q \rangle,  \langle S_{x,\tau}^{-1} x_{gp},  \tau_{gq}\rangle=\langle S_{x,\tau}^{-1} x_p, \tau_q \rangle, \langle \tau_{gp} , \tau_{gq}\rangle=\langle \tau_p, \tau_q \rangle  $ for all $ g,p,q \in G.$
\item  $ x_g=S_{x,\tau}^{1/2}\pi_gS_{x,\tau}^{-1/2}x_e, \tau_g=S_{x,\tau}^{1/2}\pi_gS_{x,\tau}^{-1/2}\tau_e $ for all $ g \in G$  if and only if $\langle S_{x,\tau}^{-1} x_{gp} , x_{gq}\rangle=\langle S_{x,\tau}^{-1}x_p, x_q \rangle$, $  \langle S_{x,\tau}^{-1}x_{gp},  \tau_{gq}\rangle=\langle S_{x,\tau}^{-1}x_p, \tau_q \rangle, \langle S_{x,\tau}^{-1}\tau_{gp} , \tau_{gq}\rangle=\langle S_{x,\tau}^{-1}\tau_p, \tau_q \rangle  $ for all $ g,p,q \in G.$
\item  $ x_g=\pi_gx_e, \tau_g=S_{x,\tau}\pi_gS_{x,\tau}^{-1}\tau_e $ for all $ g \in G$  if and only if $\langle x_{gp} , x_{gq}\rangle=\langle x_p, x_q \rangle,  \langle x_{gp}, S_{x,\tau}^{-1} \tau_{gq}\rangle=\langle x_p, S_{x,\tau}^{-1}\tau_q \rangle, \langle \tau_{gp} , S_{x,\tau}^{-2}\tau_{gq}\rangle=\langle \tau_p, S_{x,\tau}^{-2}\tau_q \rangle  $ for all $ g,p,q \in G.$
\end{enumerate}
\end{corollary}   
\begin{proof} Apply Theorem \ref{GROUPC1} to 
\begin{enumerate}[\upshape(i)]
\item $(\{S^{-1}_{x,\tau}x_g\}_{g\in G},  \{\tau_g\}_{g\in G}) $ to get: there is a unitary representation $ \pi$  of $ G$ on  $ \mathcal{H}$  for which $ S_{x,\tau}^{-1}x_g=\pi_gS_{x,\tau}^{-1}x_e, \tau_g=\pi_g\tau_e $ for all $ g \in G$  if and only if $\langle S_{x,\tau}^{-1}x_{gp} , S_{x,\tau}^{-1}x_{gq}\rangle=\langle S_{x,\tau}^{-1}x_p,S_{x,\tau}^{-1} x_q \rangle,  \langle S_{x,\tau}^{-1} x_{gp},  \tau_{gq}\rangle=\langle S_{x,\tau}^{-1} x_p, \tau_q \rangle, \langle \tau_{gp} , \tau_{gq}\rangle=\langle \tau_p, \tau_q \rangle  $ for all $ g,p,q \in G.$
\item $(\{S^{-1/2}_{x,\tau}x_g\}_{g\in G},  \{S^{-1/2}_{x,\tau}\tau_g\}_{g\in G}) $ to get: there is a unitary representation $ \pi$  of $ G$ on  $ \mathcal{H}$  for which $ S_{x,\tau}^{-1/2}x_g=\pi_g(S_{x,\tau}^{-1/2}x_e), S_{x,\tau}^{-1/2}\tau_g=\pi_g(S_{x,\tau}^{-1/2}\tau_e) $ for all $ g \in G$  if and only if $\langle S_{x,\tau}^{-1/2}x_{gp} , S_{x,\tau}^{-1/2}x_{gq}\rangle=\langle S_{x,\tau}^{-1/2}x_p, S_{x,\tau}^{-1/2}x_q \rangle,  \langle S_{x,\tau}^{-1/2}x_{gp},  S_{x,\tau}^{-1/2}\tau_{gq}\rangle=\langle S_{x,\tau}^{-1/2}x_p, S_{x,\tau}^{-1/2}\tau_q \rangle, \langle S_{x,\tau}^{-1/2}\tau_{gp},
S_{x,\tau}^{-1/2}\tau_{gq}\rangle=\langle S_{x,\tau}^{-1/2}\tau_p, $ $ S_{x,\tau}^{-1/2}\tau_q \rangle  $ for all $ g,p,q \in G.$
\item $(\{x_g\}_{g\in G},  \{S^{-1}_{x,\tau}\tau_g\}_{g\in G}) $ to get: there is a unitary representation $ \pi$  of $ G$ on  $ \mathcal{H}$  for which $ x_g=\pi_gx_e, S_{x,\tau}^{-1}\tau_g=\pi_g(S_{x,\tau}^{-1}\tau_e) $ for all $ g \in G$  if and only if $\langle x_{gp} , x_{gq}\rangle=\langle x_p, x_q \rangle,  \langle x_{gp},  S_{x,\tau}^{-1}\tau_{gq}\rangle=\langle x_p, S_{x,\tau}^{-1}\tau_q \rangle,  \langle S_{x,\tau}^{-1}\tau_{gp} , S_{x,\tau}^{-1}\tau_{gq}\rangle=\langle S_{x,\tau}^{-1}\tau_p, S_{x,\tau}^{-1}\tau_q \rangle  $ for all $ g,p,q \in G.$
\end{enumerate}
\end{proof}
 \begin{corollary}
 Let $ G$ be a locally compact group with identity $ e$ and $ \{x_g\}_{g\in G}$ be a	
 \begin{enumerate}[\upshape(i)]
 \item  Parseval  continuous  frame (w.r.t. itself) for  $ \mathcal{H}$. Then there is a unitary representation $ \pi$  of $ G$ on  $ \mathcal{H}$  for which 
 $$x_g=\pi_{g}x_e,\quad \forall g \in G$$  
 if and only if 
 $$\langle x_{gp} , x_{gq}\rangle=\langle x_p, x_q \rangle,  \quad\forall  g,p,q \in G.  $$
 \item continuous frame (w.r.t. itself) for  $ \mathcal{H}$. Then there is a unitary representation $ \pi$  of $ G$ on  $ \mathcal{H}$  for which 
 $$ x_g=S_{x,x}^{1/2}\pi_{g}S_{x,x}^{-1/2}x_e,\quad \forall g \in G$$ 
  if and only if 
 $$\langle S_{x,x}^{-1}x_{gp} , x_{gq}\rangle=\langle S_{x,x}^{-1}x_p, x_q \rangle,  \quad\forall  g,p,q \in G.  $$
 \end{enumerate}
 \end{corollary}

\textbf{Perturbations}  
\begin{theorem}\label{PERTURBATION1SV}
Let $ (\{x_\alpha\}_{\alpha \in \Omega},\{\tau_\alpha\}_{\alpha \in \Omega} )$ be a continuous frame for  $\mathcal{H}. $ Suppose $ \{y_\alpha\}_{\alpha \in \Omega}$  in $ \mathcal{H}$ is  such that
\begin{enumerate}[\upshape(i)]
\item $ \langle h,y_\alpha \rangle \tau_\alpha=\langle h,\tau_\alpha \rangle y_\alpha, \langle h, y_\alpha\rangle  \langle \tau_\alpha,h \rangle \geq 0, \forall h \in \mathcal{H}, \forall \alpha \in \Omega$,
\item for each  $h \in \mathcal{H}$, the map    $\Omega \ni \alpha \mapsto \langle h,y_\alpha \rangle\in \mathbb{K}$ is measurable,
\item there exist $ \alpha, \beta, \gamma \geq0$ with  $ \max\{\alpha+\gamma\|\theta_\tau S_{x,\tau}^{-1}\|, \beta\}<1$ such that 

\begin{align*}\label{PERTURBATIONINEQUALITYAEQUENTIALFIRST}
\left\|\int_{\Omega}f(\alpha)(x_\alpha-y_\alpha) \,d\mu(\alpha)\right\|\leq \alpha\left\|\int_{\Omega}f(\alpha)x_\alpha\,d\mu(\alpha)\right \|+\beta\left\|\int_{ \Omega}f(\alpha)y_\alpha\,d\mu(\alpha)\right \|+\gamma \|f\|,\quad \forall f \in \mathcal{L}^2(\Omega, \mathbb{K}).
\end{align*}	 
\end{enumerate}
Then  $ (\{y_\alpha\}_{\alpha \in \Omega},\{\tau_\alpha\}_{\alpha \in \Omega} )$ is  a continuous frame  with bounds $ \frac{1-(\alpha+\gamma\|\theta_\tau S_{x,\tau}^{-1}\|)}{(1+\beta)\|S_{x,\tau}^{-1}\|}$ and $\frac{\|\theta_\tau\|((1+\alpha)\|\theta_x\|+\gamma)}{1-\beta} $.
\end{theorem}
\begin{proof}
Define $T:\mathcal{L}^2(\Omega, \mathbb{K})\ni f \mapsto \int_{\Omega}f(\alpha)y_\alpha \,d\mu(\alpha)\in \mathcal{H} $. Then for all $f \in \mathcal{L}^2(\Omega, \mathbb{K})$,
 
 \begin{align*}
 \|Tf\|&=\left\|\int_{\Omega}f(\alpha)y_\alpha\,d\mu(\alpha)\right\|\leq \left\|\int_{\Omega}f(\alpha)(y_\alpha-x_\alpha)\,d\mu(\alpha)\right\|+\left\|\int_{\Omega}f(\alpha)x_\alpha\,d\mu(\alpha)\right\|\\
 &=(1+\alpha)\left\|\int_{\Omega}f(\alpha)x_\alpha\,d\mu(\alpha)\right\|+\beta\left\|\int_{\Omega}f(\alpha)y_\alpha\,d\mu(\alpha)\right\|+\gamma \|f\|\\
 &=(1+\alpha)\left\|\theta_x^*f\right\|+\beta\left\|Tf\right\|+\gamma \|f\|,
 \end{align*}
which implies

 \begin{align*}
 \|Tf\|\leq \frac{1+\alpha}{1-\beta}\left\| \theta_x^*f\right\|+\frac{\gamma}{1-\beta}\|f\|,\quad \forall f \in \mathcal{L}^2(\Omega, \mathbb{K}) \text{ and }
 \end{align*}
$$ \|\theta_x^*f-\theta_y^*f\|\leq \alpha\|\theta_x^*f\|+\beta\|\theta_y^*f\|+\gamma\|f\|, \quad \forall  f  \in \mathcal{L}^2(\Omega, \mathbb{K}) .$$
 Other arguments are similar to the corresponding arguments used in the proof of Theorem \ref{PERTURBATION RESULT 1}.
\end{proof} 
\begin{corollary}
Let $ (\{x_\alpha\}_{\alpha \in \Omega},\{\tau_\alpha\}_{\alpha \in \Omega} )$ be a continuous frame for  $\mathcal{H}. $ Suppose $ \{y_\alpha\}_{\alpha \in \Omega}$  in $ \mathcal{H}$ is  such that
\begin{enumerate}[\upshape(i)]
\item $ \langle h,y_\alpha \rangle \tau_\alpha=\langle h,\tau_\alpha \rangle y_\alpha, \langle h, y_\alpha\rangle  \langle \tau_\alpha,h \rangle \geq 0, \forall h \in \mathcal{H}, \forall \alpha \in \Omega$,
\item for each  $h \in \mathcal{H}$, the map    $\Omega \ni \alpha \mapsto \langle h,y_\alpha \rangle\in \mathbb{K}$ is measurable,
\item The map    $\Omega \ni \alpha \mapsto\|x_\alpha-y_\alpha\| \in \mathbb{R}$ is measurable,
\item  
$$ r\coloneqq \int_{\Omega}\|x_\alpha-y_\alpha\|^2\,d\mu(\alpha)<\frac{1}{\|\theta_\tau S_{x,\tau}^{-1}\|^2}.$$	 
\end{enumerate}
Then  $ (\{y_\alpha\}_{\alpha \in \Omega},\{\tau_\alpha\}_{\alpha \in \Omega} )$ is  a continuous frame  with bounds $ \frac{1-\sqrt{r}\|\theta_\tau S_{x,\tau}^{-1}\|}{\|S_{x,\tau}^{-1}\|}$ and $\|\theta_\tau\|(\|\theta_x\|+\sqrt{r}) $.
\end{corollary}
\begin{proof}
Take $ \alpha =0, \beta=0, \gamma=\sqrt{r}$. Then $ \max\{\alpha+\gamma\|\theta_\tau S_{x,\tau}^{-1}\|, \beta\}<1$ and 
$$\left\|\int_{\Omega}f(\alpha)(x_\alpha-y_\alpha)\,d\mu(\alpha) \right\|\leq  \left(\int_{\Omega}|f(\alpha)|^2\,d\mu(\alpha)\right)^\frac{1}{2} \left(\int_{\Omega}\|x_\alpha-y_\alpha\|^2\,d\mu(\alpha)\right)^\frac{1}{2}\leq \gamma  \|f\|,\quad\forall f \in \mathcal{L}^2(\Omega, \mathbb{K}).$$	
 Now we can apply Theorem \ref{PERTURBATION1SV}.
\end{proof}
\begin{theorem}\label{PERTURBATION2SV}
Let $ (\{x_\alpha\}_{\alpha \in \Omega},\{\tau_\alpha\}_{\alpha \in \Omega} )$ be a continuous frame for  $\mathcal{H}$ with bounds $ a$ and $b.$ Suppose $ \{y_\alpha\}_{\alpha \in \Omega}$  in $ \mathcal{H}$ is  such that   $ \int_{\Omega}\langle h,y_\alpha \rangle \langle \tau_\alpha, h \rangle\,d\mu(\alpha)$ exists for all $ h \in \mathcal{H}$ and  is nonnegative for all $ h \in \mathcal{H}$ and there exist $ \alpha, \beta, \gamma \geq0$ with  $ \max\{\alpha+\frac{\gamma}{\sqrt{a}}, \beta\}<1$ and for all $ h \in \mathcal{H}$,

\begin{align*}
\left| \int_{\Omega}\langle h,x_\alpha-y_\alpha\rangle \langle \tau_\alpha, h\rangle \,d\mu(\alpha)\right|^\frac{1}{2} \leq \alpha\left(\int_{\Omega}\langle h,x_\alpha\rangle \langle \tau_\alpha,h \rangle \,d\mu(\alpha)\right)^\frac{1}{2} + \beta\left(\int_{\Omega}\langle h,y_\alpha\rangle \langle \tau_\alpha,h \rangle\,d\mu(\alpha) \right)^\frac{1}{2} +\gamma \|h\|.
\end{align*}
Then $ (\{y_\alpha\}_{\alpha \in \Omega},\{\tau_\alpha\}_{\alpha\in \Omega} )$ is  a continuous frame with bounds $a\left(1-\frac{\alpha+\beta+\frac{\gamma}{\sqrt{a}}}{1+\beta}\right)^2 $ and $b\left(1+\frac{\alpha+\beta+\frac{\gamma}{\sqrt{b}}}{1-\beta}\right)^2.$
\end{theorem}
\begin{proof}
Similar to the proof of Theorem \ref{PERTURBATION RESULT 2}.
\end{proof} 
\begin{theorem}\label{PERTURBATION3SV}
Let $ (\{x_\alpha\}_{\alpha\in \Omega}, \{\tau_\alpha\}_{\alpha\in \Omega}) $  be  a continuous frame for  $\mathcal{H}$. Suppose  $\{y_\alpha\}_{\alpha\in \Omega} $ in $\mathcal{H}$ is such that 
\begin{enumerate}[\upshape(i)]
\item $ \langle h,y_\alpha \rangle \tau_\alpha=\langle h,\tau_\alpha\rangle y_\alpha, \langle h, y_\alpha\rangle  \langle \tau_\alpha,h \rangle \geq 0, \forall h \in \mathcal{H}, \forall \alpha \in \Omega$,
\item for each  $h \in \mathcal{H}$, the map    $\Omega \ni \alpha \mapsto \langle h, y_\alpha \rangle \in \mathbb{K}$ is measurable,
\item The map    $\Omega \ni \alpha \mapsto\|x_\alpha-y_\alpha\| \in \mathbb{R}$ is measurable and $   \int_{\Omega}\|x_\alpha-y_\alpha\|\,d\mu(\alpha) \in \mathbb{R}$,
\item  The map    $\Omega \ni \alpha \mapsto\| S_{x,\tau}^{-1}\tau_\alpha\|\in \mathbb{R}$ is measurable and $   \int_{\Omega}\|x_\alpha-y_\alpha\|\| S_{x,\tau}^{-1}\tau_\alpha\|\,d\mu(\alpha) \in \mathbb{R}$,
\item $\int_{\Omega}\|x_\alpha-y_\alpha\|\| S_{x,\tau}^{-1}\tau_\alpha\|\,d\mu(\alpha)<1.$
\end{enumerate}
Then  $ (\{y_\alpha\}_{\alpha\in \Omega}, \{\tau_\alpha\}_{\alpha\in \Omega}) $ is a continuous frame with bounds  $\frac{1-\int_{\Omega}\|x_\alpha-y_\alpha\|\|S_{x,\tau}^{-1}\tau_\alpha\|\,d\mu(\alpha)}{\|S_{x,\tau}^{-1}\|}$ and $\|\theta_\tau\|((\int_{\Omega}\|x_\alpha-y_\alpha\|^2\,d\mu(\alpha))^{1/2}+\|\theta_x\|) $.
\end{theorem}
\begin{proof}
Let $\alpha=(\int_{\Omega}\|x_\alpha-y_\alpha\|^2\,d\mu(\alpha))^{1/2} $,  $\beta=\int_{\Omega}\|x_\alpha-y_\alpha\|\|S_{x,\tau}^{-1}\tau_\alpha\|\,d\mu(\alpha) $.  Fix $f\in \mathcal{L}^2(\Omega, \mathbb{K})$. Then for all $h \in \mathcal{H}$,

\begin{align*}
\left| \int_{\Omega}\langle f(\alpha)y_\alpha, h\rangle \,d\mu(\alpha)\right|&\leq\left| \int_{\Omega}\langle f(\alpha)(y_\alpha-x_\alpha), h\rangle \,d\mu(\alpha)\right|+\left| \int_{\Omega}\langle f(\alpha)x_\alpha, h\rangle \,d\mu(\alpha)\right| \\
&=\left| \int_{\Omega}\langle f(\alpha)(y_\alpha-x_\alpha), h\rangle \,d\mu(\alpha)\right|+|\langle \theta_x^*f, h\rangle |\\
&\leq \int_{\Omega}|f(\alpha)|\|y_\alpha-x_\alpha\|\|h\|\,d\mu(\alpha)+ \|f\|\|\theta_xh\|\\
&\leq\|h\| \int_{\Omega}|f(\alpha)|\|y_\alpha-x_\alpha\|\,d\mu(\alpha)+ \|f\|\|\theta_x\|\|h\|\\
&\leq\|h\| \left(\int_{\Omega}|f(\alpha)|^2\,d\mu(\alpha)\right)^\frac{1}{2}\left(\int_{\Omega}\|y_\alpha-x_\alpha\|^2\,d\mu(\alpha)\right)^\frac{1}{2}+ \|f\|\|\theta_x\|\|h\|\\
&=\|h\|\|f\|\alpha+ \|f\|\|\theta_x\|\|h\|=(\|f\|\alpha+ \|f\|\|\theta_x\|)\|h\|
\end{align*}
and hence

\begin{align*}
\|I_\mathcal{H}-S_{y,\tau}S_{x,\tau}^{-1}\|&=\sup_{h, g \in \mathcal{H}, \|h\|=1=\|g\|} |\langle (I_\mathcal{H}-S_{y,\tau}S_{x,\tau}^{-1})h, g\rangle |\\
&=\sup_{h, g \in \mathcal{H}, \|h\|=1=\|g\|}\left| \int_{\Omega}\langle S_{x,\tau}^{-1}h, \tau_\alpha \rangle \langle x_\alpha ,  g\rangle \,d \mu(\alpha) -\int_{\Omega}\langle S_{x,\tau}^{-1}h, \tau_\alpha \rangle \langle y_\alpha ,  g\rangle \,d \mu(\alpha)\right|\\
&=\sup_{h, g \in \mathcal{H}, \|h\|=1=\|g\|}\left| \int_{\Omega}\langle S_{x,\tau}^{-1}h, \tau_\alpha \rangle \langle x_\alpha-y_\alpha ,  g\rangle \,d \mu(\alpha) \right|\\
&=\sup_{h, g \in \mathcal{H}, \|h\|=1=\|g\|}\left| \int_{\Omega}\langle h, S_{x,\tau}^{-1} \tau_\alpha \rangle \langle x_\alpha-y_\alpha ,  g\rangle \,d \mu(\alpha) \right|\\
&\leq \sup_{h, g \in \mathcal{H}, \|h\|=1=\|g\|} \int_{\Omega}\| x_\alpha-y_\alpha\|\| S_{x,\tau}^{-1}\tau_\alpha\|\|h\|\|g\|\,d \mu(\alpha)\\
&=\int_{\Omega}\| x_\alpha-y_\alpha\|\| S_{x,\tau}^{-1}\tau_\alpha\|\,d \mu(\alpha) =\beta<1.
\end{align*}
Other arguments are  similar to the corresponding arguments in the proof of  Theorem \ref{OVFQUADRATICPERTURBATION}.

\end{proof}

\section{The finite dimensional case}\label{THEFINITEDIMENSIONALCASE}
\begin{theorem}\label{FINITEDIMENSIONALCHARATERIZATIONHILBERT}
Let $ \mathcal{H}$ be a finite dimensional Hilbert space, $G$ be a locally compact group, $ \{x_\alpha\}_{\alpha \in G} , \{\tau_\alpha\}_{\alpha \in G} $ be  a  set of vectors in $ \mathcal{H}$  such that 
\begin{enumerate}[\upshape(i)]
\item $\langle h,  x_\alpha \rangle \tau_\alpha =\langle h, \tau_\alpha \rangle x_\alpha, \forall h \in  \mathcal{H}, \forall \alpha \in G$.
\item $ \langle h, x_\alpha\rangle \langle\tau_\alpha, h\rangle\geq0, \forall h \in  \mathcal{H}, \forall \alpha \in G. $
\item for each  $h \in \mathcal{H}$, both    maps   $G \ni \alpha \mapsto\langle  h, x_\alpha\rangle \in \mathbb{K}$, $G \ni\alpha \mapsto \langle  h, \tau_\alpha \rangle \in\mathbb{K}$ are measurable.
\item the map $G \ni \alpha \mapsto \|x_\alpha\| \|\tau_\alpha\|\in \mathbb{K}$ is in $\mathcal{L}^2(G, \mathbb{K})$.
\item for each  $h \in \mathcal{H}$, the map   $G \ni \alpha \mapsto\langle  h, x_\alpha\rangle\langle  \tau_\alpha, h \rangle \in \mathbb{K}$ is continuous.
\end{enumerate}
Then  $ (\{x_\alpha\}_{\alpha \in G} , \{\tau_\alpha\}_{\alpha \in G}) $ is a continuous frame for $ \mathcal{H}$ if and only if for every pair $G_x, G_\tau$ of subsets of $G$ satisfying $G_x\cap G_\tau=\emptyset$ and $G_x\cup G_\tau=G$ one has $ \operatorname{span}\{x_{\alpha},\tau_{\beta}:\alpha \in G_x ,  \beta \in  G_\tau\}= \mathcal{H}.$
\end{theorem} 
\begin{proof} We can assume $\mathcal{H}\neq \{0\}$.  
	
$(\Leftarrow)$ There exists $\alpha \in G $  such that $x_\alpha \neq0\neq\tau_\alpha$ (else $ S_{x, \tau}$ is zero). Hence $\int_{G}\| x_\alpha\|\|\tau_\alpha\|\,d\mu_G(\alpha)>0$. Clearly $ S_{x, \tau}$ is self-adjoint and positive. Now 

\begin{align*}
\int_{G}\langle h, x_\alpha\rangle \langle\tau_\alpha, h\rangle\,d\mu_G(\alpha)\leq \int_{G}|\langle h, x_\alpha\rangle \langle\tau_\alpha, h\rangle|\,d\mu_G(\alpha)\leq \left(\int_{G}\| x_\alpha\|\|\tau_\alpha\|\,d\mu_G(\alpha)\right)\|h\|^2.
\end{align*}  
Hence the upper frame bound condition is satisfied. Define $ \phi :  \mathcal{H} \ni h \mapsto  \int_{G}\langle h, x_\alpha\rangle \langle\tau_\alpha, h\rangle\,d\mu(\alpha) \in \mathbb{R}.$ We argue that $ \phi$ is continuous. Let $h_n \to h $ in $\mathcal{H}$ as $n \to \infty.$ Then 

\begin{align*}
|\phi(h_n)-\phi(h)|&=\left|\int_{G}\langle h_n, x_\alpha\rangle \langle\tau_\alpha, h_n\rangle\,d\mu_G(\alpha)-\int_{G}\langle h, x_\alpha\rangle \langle\tau_\alpha, h\rangle\,d\mu_G(\alpha)\right|\\
&=\bigg|\int_{G}\langle h_n, x_\alpha\rangle \langle\tau_\alpha, h_n\rangle\,d\mu_G(\alpha)-\int_{G}\langle h, x_\alpha\rangle \langle\tau_\alpha, h_n\rangle\,d\mu_G(\alpha)\\
&\quad +\int_{G}\langle h, x_\alpha\rangle \langle\tau_\alpha, h_n\rangle\,d\mu_G(\alpha)-\int_{G}\langle h, x_\alpha\rangle \langle\tau_\alpha, h\rangle\,d\mu_G(\alpha)\bigg|\\
&=\left|\int_{G}\langle h_n-h, x_\alpha\rangle \langle\tau_\alpha, h_n\rangle\,d\mu_G(\alpha)+\int_{G}\langle h, x_\alpha\rangle \langle\tau_\alpha, h_n-h \rangle\,d\mu_G(\alpha)\right|\\
&\leq \int_{G}\|h_n-h\|\| x_\alpha\|\| \|\tau_\alpha\| \|h_n\|\,d\mu_G(\alpha)+\int_{G}\| h\|\|x_\alpha\| \|\tau_\alpha\| \|h_n-h \|\,d\mu_G(\alpha)\\
&\leq \left(\left(\sup_{n \in \mathbb{N}}\|h_n\|+\|h\|\right)\int_{G}\| x_\alpha\|\|\tau_\alpha\|\,d\mu_G(\alpha)\right)\|h_n-h\| \to 0 \text{ as }  n \to \infty.
\end{align*}
Compactness of  the unit sphere of  $\mathcal{H}$ gives the existence of   $ g \in  \mathcal{H}$ with $ \|g\|=1$ such that $ a\coloneqq\int_{G}\langle g, x_\alpha\rangle \langle\tau_\alpha, g\rangle\,d\mu_G(\alpha)=\inf\{\int_{G}\langle h, x_\alpha\rangle \langle\tau_\alpha, h\rangle\,d\mu_G(\alpha): h \in\mathcal{H}, \|h\|=1 \}.$ We claim that  $ a>0.$ If this  fails:  since $\langle g, x_\alpha\rangle \langle\tau_\alpha, g\rangle\geq0, \forall \alpha \in G  $, $G$ is a locally compact group and the map $G \ni \alpha \mapsto\langle  g, x_\alpha\rangle\langle  \tau_\alpha, g \rangle \in \mathbb{K}$ is continuous, from \cite{HEWITTROSS} we must have $ \langle g, x_\alpha\rangle \langle\tau_\alpha, g\rangle=0, \forall \alpha \in G.$ Define $G_x\coloneqq \{\alpha \in G:\langle  g, x_\alpha\rangle=0\}$  and $G_\tau\coloneqq \{\alpha \in G:\langle  \tau_\alpha, g \rangle=0\}\setminus G_x$. Now using  $\langle g, x_\alpha\rangle \langle\tau_\alpha, g\rangle=0, \forall \alpha \in G$ we get $G_x \cup G_\tau=G $. Clearly $G_x \cap G_\tau=\emptyset $. Then $ g \perp \operatorname{span}\{x_{\alpha},\tau_{\beta}:\alpha \in G_x ,  \beta \in  G_\tau\}= \mathcal{H}$  which implies $ g=0$ which is forbidden. We claim that  $ a$ is a lower frame bound. For all nonzero $ h \in \mathcal{H}$,

\begin{align*}
a\|h\|^2\leq\left(\int_{G}\left\langle \frac{h}{\|h\|}, x_\alpha\right\rangle \left\langle\tau_\alpha, \frac{h}{\|h\|}\right\rangle\,d\mu_G(\alpha)\right)\|h\|^2=\int_{G}\langle h, x_\alpha\rangle \langle\tau_\alpha, h\rangle\,d\mu_G(\alpha).
\end{align*} 
$ (\Rightarrow)$  We prove by contrapositive. Suppose  there are  subsets $G_x, G_\tau$  of $G$ satisfying $G_x\cap G_\tau=\emptyset$ and $G_x\cup G_\tau=G$ such that $ \operatorname{span}\{x_{\alpha},\tau_{\beta}:\alpha \in G_x ,  \beta \in G_\tau\} \subsetneq \mathcal{H}.$ Let $ h \in \mathcal{H}$ be nonzero such that $ h\perp \operatorname{span}\{x_{\alpha},\tau_{\beta}:\alpha \in G_x ,  \beta \in  G_\tau\}.$ Now because of  $G_x\cap G_\tau=\emptyset$ and $G_x\cup G_\tau=G$ we get $ \int_{G}\langle h, x_\alpha\rangle \langle\tau_\alpha, h\rangle\,d\mu_G(\alpha) =0 $ which says that the lower frame bound condition fails.
\end{proof} 
\begin{proposition}
Let $ (\{x_\alpha\}_{\alpha\in \Omega} , \{\tau_\alpha\}_{\alpha\in \Omega}) $ be  a continuous frame for $\mathcal{H}$ with a lower frame bound  $a $ and $ \langle h, x_\alpha\rangle \langle\tau_\alpha, h\rangle\geq0, \forall h \in  \mathcal{H}, \forall \alpha \in \Omega. $ If $ \Delta$ is any measurable subset of $\Omega$ such that $\Delta \ni \alpha \mapsto \|x_\alpha\| \|\tau_\alpha\|\in \mathbb{K}$ is measurable and  $\int_{\Delta}\|x_\alpha\| \|\tau_\alpha\|\,d\mu(\alpha)<a ,$  then $ (\{x_\alpha\}_{\alpha\in \Omega\setminus\Delta} , \{\tau_\alpha\}_{\alpha\in \Omega\setminus\Delta}) $ is a continuous  frame for $\mathcal{H} $ with lower frame bound $a-\int_{\Delta}\|x_\alpha\| \|\tau_\alpha\|\,d\mu(\alpha)$.
\end{proposition}
\begin{proof}
\begin{align*}
a\|h\|^2&\leq \int_{\Omega}\langle h,x_\alpha \rangle\langle \tau_\alpha, h\rangle\,d\mu(\alpha)=\int_{\Delta}\langle h,x_\alpha \rangle\langle \tau_\alpha, h\rangle\,d\mu(\alpha)+\int_{ \Omega\setminus\Delta}\langle h,x_\alpha \rangle\langle \tau_\alpha, h\rangle \,d\mu(\alpha)\\
&\leq\|h\|^2\int_{\Delta}\|x_\alpha\|\|\tau_\alpha\| \,d\mu(\alpha)+\int_{ \Omega\setminus\Delta}\langle h,x_\alpha \rangle\langle \tau_\alpha,h\rangle\,d\mu(\alpha), \quad\forall h \in \mathcal{H}.
\end{align*}
\end{proof}

\begin{theorem}\label{FINITEDIMENSIONALPASTINGTHEOREM}
Let   $ (\{x_\alpha\}_{\alpha\in \Omega} , \{\tau_\alpha\}_{\alpha\in \Omega}) $ be  a continuous frame for a finite dimensional complex Hilbert space $ \mathcal{H}$ of dimension $ m$. Then we have the following. 
\begin{enumerate}[\upshape(i)]
\item The optimal lower frame bound (resp. optimal upper frame bound) is the smallest (resp. largest) eigenvalue for $ S_{x, \tau}.$ 
\item If $ \{\lambda_j\}_{j=1}^m$ denotes the eigenvalues for $S_{x, \tau},$  each appears as many times as its algebraic multiplicity, then  
$$ \sum_{j=1}^m\lambda_j=\int_{\Omega}\langle x_\alpha,\tau_\alpha \rangle\,d\mu(\alpha) =\int_{\Omega}\langle \tau_\alpha, x_\alpha \rangle\,d\mu(\alpha).$$
\item Condition number for  $S_{x, \tau}$ is  equal to the ratio between the optimal upper frame bound and the optimal lower frame bound.
\item If the optimal upper frame bound is $b,$  then 
$$ b\leq\int_{\Omega}\langle x_\alpha,\tau_\alpha \rangle\,d\mu(\alpha) =\int_{\Omega}\langle \tau_\alpha, x_\alpha \rangle\,d\mu(\alpha)\leq mb. $$
\item \begin{align*}
&\operatorname{Trace}(S_{x,\tau})=\int_{\Omega}\langle x_\alpha,\tau_\alpha \rangle\,d\mu(\alpha)=\int_{\Omega}\langle \tau_\alpha,x_\alpha \rangle\,d\mu(\alpha);\\
 &\operatorname{Trace}(S^2_{x,\tau}) =\int_{\Omega}\int_{\Omega}\langle \tau_\alpha,x_\beta\rangle \langle \tau_\beta,x_\alpha\rangle\,d\mu(\beta)\,d\mu(\alpha)=\int_{\Omega}\int_{\Omega}\langle \tau_\alpha,\tau_\beta\rangle \langle x_\beta,x_\alpha\rangle \,d\mu(\beta)\,d\mu(\alpha).
\end{align*}
\item If the frame is tight, then the optimal frame bound $b=\frac{1}{m}\int_{\Omega}\langle x_\alpha,\tau_\alpha \rangle\,d\mu(\alpha)=\frac{1}{m}\int_{\Omega}\langle \tau_\alpha,x_\alpha \rangle\,d\mu(\alpha).$ In particular, if $\langle x_\alpha,\tau_\alpha\rangle=1, \forall \alpha\in\Omega,$ then $b={\mu(\Omega)}/m.$ Further, 

\begin{align*}
h&=\frac{1}{b}\int_{\Omega}\langle h,x_\alpha \rangle \tau_\alpha\,d\mu(\alpha)=\frac{1}{b}\int_{\Omega}\langle h,\tau_\alpha \rangle x_\alpha\,d\mu(\alpha), \quad\forall h \in \mathcal{H} ;\\
 \|h\|^2&=\frac{1}{b}\int_{\Omega}\langle h,x_\alpha \rangle \langle \tau_\alpha, h \rangle\,d\mu(\alpha)=\frac{1}{b}\int_{\Omega}\langle h,\tau_\alpha \rangle \langle x_\alpha, h \rangle\,d\mu(\alpha),\quad\forall h \in \mathcal{H}.
\end{align*}
\item If the frame is tight, then 

\begin{align*}
\text{(Extended variation formula)}\quad  &\int_{\Omega}\int_{\Omega}\langle \tau_\alpha,x_\beta\rangle \langle \tau_\beta,x_\alpha\rangle \,d\mu(\beta)\,d\mu(\alpha)=\frac{1}{\dim\mathcal{H}}\left(\int_{\Omega}\langle x_\alpha,\tau_\alpha \rangle\,d\mu(\alpha)\right)^2\\
&=\frac{1}{\dim\mathcal{H}}\left(\int_{\Omega}\langle \tau_\alpha,x_\alpha \rangle\,d\mu(\alpha)\right)^2 =\int_{\Omega}\int_{\Omega}\langle \tau_\alpha,\tau_\beta\rangle \langle x_\beta,x_\alpha\rangle\,d\mu(\beta)\,d\mu(\alpha) . 
\end{align*}
\item If the frame is Parseval, then 
$$\text{(Extended dimension formula)}\quad \quad\dim\mathcal{H}=\int_{\Omega}\langle x_\alpha,\tau_\alpha \rangle \,d\mu(\alpha)=\int_{\Omega}\langle \tau_\alpha, x_\alpha \rangle\,d\mu(\alpha). $$
\item If the frame is Parseval, then for every $T \in \mathcal{B}(\mathcal{H}),$
$$ \text{(Extended trace formula)}\quad \quad\operatorname{Trace}(T)=\int_{\Omega}\langle Tx_\alpha,\tau_\alpha \rangle \,d\mu(\alpha)=\int_{\Omega}\langle T\tau_\alpha, x_\alpha \rangle\,d\mu(\alpha).  $$
\end{enumerate}
\end{theorem}
\begin{proof}
\begin{enumerate}[\upshape(i)]
\item Using spectral theorem,  $\mathcal{H}$ has an orthonormal basis $ \{e_j\}_{j=1}^m$ consisting of eigenvectors for $S_{x, \tau}.$ Let $\{\lambda_j\}_{j=1}^m $ denote the corresponding eigenvalues. Then $S_{x, \tau}h=\sum_{j=1}^m\langle h, e_j \rangle S_{x, \tau}e_j=\sum_{j=1}^m\lambda_j\langle h, e_j \rangle e_j, \forall h \in \mathcal{H}.$ Since  $S_{x, \tau}$ is positive invertible, $ \lambda_j>0, \forall j =1, ..., m.$ Therefore 

\begin{align*}
 \min\{\lambda_j\}_{j=1}^m \|h\|^2 &\leq \sum_{j=1}^m\lambda_j|\langle h, e_j \rangle|^2 =\langle S_{x, \tau}h,h \rangle \\
 &= \int_{\Omega}\langle h, x_\alpha \rangle \langle \tau_\alpha, h \rangle\,d\mu(\alpha)\leq \max\{\lambda_j\}_{j=1}^m\|h\|^2 ,\quad\forall h \in \mathcal{H}.
\end{align*}
To get optimal frame bounds we take eigenvectors corresponding to $\min\{\lambda_j\}_{j=1}^m $ and $\max\{\lambda_j\}_{j=1}^m.$
\item \begin{align*}
\sum_{j=1}^m\lambda_j&= \sum_{j=1}^m\lambda_j\|e_j\|^2 =\sum_{j=1}^m\langle S_{x, \tau}e_j,e_j \rangle =\sum_{j=1}^m \int_{\Omega}\langle e_j, x_\alpha \rangle \langle \tau_\alpha, e_j \rangle \,d\mu(\alpha)
\\
&=\int_{\Omega}(\sum_{j=1}^m \langle \tau_\alpha, e_j \rangle\langle e_j, x_\alpha \rangle)\,d\mu(\alpha) =\int_{\Omega}\langle \tau_\alpha, x_\alpha \rangle\,d\mu(\alpha) .
\end{align*}
Since $S_{x, \tau}=S_{\tau, x} $, we get  $\sum_{j=1}^m\lambda_j=\int_{\Omega}\langle x_\alpha, \tau_\alpha\rangle\,d\mu(\alpha).$
\item This follows from (i).
\item Let  $ \{e_j\}_{j=1}^m$ and $\{\lambda_j\}_{j=1}^m $ be as in (i). We may assume $ \lambda_1\geq \cdots \geq \lambda_m$. Then (i) gives $b=\lambda_1.$ Now  use  (ii): $ b=\lambda_1 \leq \sum_{j=1}^m\lambda_j=\int_{\Omega}\langle x_\alpha,\tau_\alpha \rangle\,d\mu(\alpha) =\int_{\Omega}\langle \tau_\alpha, x_\alpha \rangle \,d\mu(\alpha)\leq \lambda_1m=bm.$
\item Let $ \{f_j\}_{j=1}^m$ be an orthonormal basis for $\mathcal{H}$. Then  

\begin{align*}
\operatorname{Trace}(S_{x,\tau})&=\sum_{j=1}^m\langle S_{x,\tau}f_j,f_j \rangle=\sum_{j=1}^m\left\langle\int_{\Omega}\langle f_j, x_\alpha\rangle \tau_\alpha\,d\mu(\alpha) ,f_j \right\rangle\\
&=\int_{\Omega}\left(\sum_{j=1}^m\langle f_j, x_\alpha\rangle \langle \tau_\alpha ,f_j \rangle\right)\,d\mu(\alpha) =\int_{\Omega}\langle \tau_\alpha, x_\alpha\rangle\,d\mu(\alpha),  ~\text{and}
\end{align*}
\begin{align*}
\operatorname{Trace}(S^2_{x,\tau})&=\sum_{j=1}^m\langle S_{x,\tau}f_j,S_{x,\tau}f_j \rangle=\sum_{j=1}^m\left\langle \int_{\Omega}\langle f_j, x_\alpha\rangle \tau_\alpha\,d\mu(\alpha), \int_{\Omega}\langle f_j, \tau_\beta\rangle  x_\beta\,d\mu(\beta)\right \rangle\\ &=\int_{\Omega}\int_{\Omega}\langle \tau_\alpha,x_\beta \rangle \sum_{j=1}^m\langle f_j, x_\alpha\rangle\langle \tau_\beta,f_j\rangle\,d\mu(\beta)\,d\mu(\alpha)=\int_{\Omega}\int_{\Omega}\langle \tau_\alpha,x_\beta \rangle \langle \tau_\beta,x_\alpha \rangle\,d\mu(\beta)\,d\mu(\alpha),
\end{align*} 
\begin{align*}
\operatorname{Trace}(S^2_{x,\tau})&=\sum_{j=1}^m\langle S_{x,\tau}f_j,S_{x,\tau}f_j \rangle=\sum_{j=1}^m\left\langle \int_{\Omega}\langle f_j, x_\alpha\rangle \tau_\alpha\,d\mu(\alpha), \int_{\Omega}\langle f_j, x_\beta\rangle  \tau_\beta\,d\mu(\beta)\right \rangle\\ 
&=\int_{\Omega}\int_{\Omega}\langle \tau_\alpha,\tau_\beta \rangle \sum_{j=1}^m\langle f_j, x_\alpha\rangle\langle x_\beta,f_j\rangle\,d\mu(\beta)\,d\mu(\alpha)=\int_{\Omega}\int_{\Omega}\langle \tau_\alpha,\tau_\beta \rangle \langle x_\beta,x_\alpha \rangle\,d\mu(\beta)\,d\mu(\alpha).
\end{align*} 
\item Now $S_{x, \tau}=\lambda I_{\mathcal{H}}, $ for some  positive $ \lambda.$ This gives $ \lambda_1=\cdots=\lambda_m=\lambda=b.$ From (ii) we get the conclusions.
\item Let the  optimal frame bound be $b$. From (v) and (vi), 

\begin{align*}
\int_{\Omega}\int_{\Omega}\langle \tau_\alpha,\tau_\beta\rangle \langle x_\beta,x_\alpha\rangle\,d\mu(\beta)\,d\mu(\alpha)&=\int_{\Omega}\int_{\Omega}\langle \tau_\alpha,x_\beta\rangle \langle \tau_\beta,x_\alpha\rangle\,d\mu(\beta)\,d\mu(\alpha)\\
&= \operatorname{Trace}(S^2_{x,\tau})= \operatorname{Trace}(b^2I_\mathcal{H})=b^2m\\
&=\left(\frac{1}{m}\int_{\Omega}\langle x_\alpha, \tau_\alpha \rangle\,d\mu(\alpha)\right)^2m=\frac{1}{m}\left(\int_{\Omega}\langle x_\alpha, \tau_\alpha \rangle\,d\mu(\alpha)\right)^2 .
\end{align*}
\item Let $ \{f_j\}_{j=1}^m$ be as in (v). Then

\begin{align*}
\dim\mathcal{H}&=\sum_{j=1}^m\|f_j\|^2 =\sum_{j=1}^m\int_{\Omega}\langle f_j,  x_\alpha \rangle\langle  \tau_\alpha, f_j \rangle \,d\mu(\alpha)\\
&=\sum_{j=1}^m\int_{\Omega}\langle f_j,  \tau_\alpha \rangle\langle  x_\alpha, f_j \rangle \,d\mu(\alpha)=
\int_{\Omega}\sum_{j=1}^m\langle f_j,  x_\alpha \rangle\langle  \tau_\alpha, f_j \rangle\,d\mu(\alpha)\\
&=\int_{\Omega}\sum_{j=1}^m\langle f_j,  \tau_\alpha \rangle\langle  x_\alpha, f_j \rangle\,d\mu(\alpha)=\int_{\Omega}\langle x_\alpha,\tau_\alpha \rangle \,d\mu(\alpha)=\int_{\Omega}\langle \tau_\alpha, x_\alpha \rangle\,d\mu(\alpha).
\end{align*}
\item Let $ \{f_j\}_{j=1}^m$ be as in (v). Then

\begin{align*}
\operatorname{Trace}(T)&=\sum_{j=1}^m\langle Tf_j,f_j \rangle = \sum_{j=1}^m\left\langle \int_{\Omega}\langle Tf_j,x_\alpha \rangle \tau_\alpha \,d\mu(\alpha),f_j \right\rangle\\
&=\int_{\Omega}\sum_{j=1}^m\langle \tau_\alpha ,f_j \rangle\langle Tf_j,x_\alpha  \rangle\,d\mu(\alpha)
=\int_{\Omega}\sum_{j=1}^m\langle \tau_\alpha,f_j \rangle\langle f_j,T^*x_\alpha \rangle\,d\mu(\alpha)\\
&=\int_{\Omega}\langle \tau_\alpha,T^*x_\alpha \rangle\,d\mu(\alpha)=\int_{\Omega}\langle T\tau_\alpha,x_\alpha \rangle\,d\mu(\alpha).
\end{align*}
  Similarly by using $Tf_j=\int_{\Omega}\langle Tf_j,\tau_\alpha \rangle x_\alpha \,d\mu(\alpha)$, we get $\operatorname{Trace}(T) =\int_{\Omega}\langle Tx_\alpha, \tau_\alpha \rangle\,d\mu(\alpha). $
\end{enumerate}	
\end{proof}
\begin{theorem}\label{REALTOCOMPLEX}
If a continuous frame $(\{x_\alpha\}_{\alpha\in \Omega} , \{\tau_\alpha\}_{\alpha\in \Omega})$ for $\mathbb{R}^m $ is such that 
$$ \int_{\Omega}\langle h, x_\alpha \rangle \langle \tau_\alpha, g \rangle \,d\mu(\alpha)=\int_{\Omega}\langle g, x_\alpha \rangle \langle \tau_\alpha, h \rangle\,d\mu(\alpha), \quad\forall h, g \in \mathbb{R}^n,$$
then it is  also a continuous frame for $\mathbb{C}^m. $ Further, if $(\{x_\alpha\}_{\alpha\in \Omega} , \{\tau_\alpha\}_{\alpha\in \Omega}) $ is a tight (resp. Parseval) continuous frame for $\mathbb{R}^m $, then it is also  a tight (resp. Parseval) continuous frame for $\mathbb{C}^m. $
\end{theorem}
\begin{proof}
Let $a, b$  be lower and upper frame bounds, in order. For $ z\in \mathbb{C}^m$ we write $ z=\operatorname{Re}z+i\operatorname{Im}z,\operatorname{Re}z,\operatorname{Im}z $ $ \in \mathbb{R}^m.$ Then

\begin{align*}
a\|z\|^2&=a\|\operatorname{Re}z\|^2+a\|\operatorname{Im}z\|^2\leq \int_{\Omega}\langle \operatorname{Re}z, x_\alpha\rangle\langle \tau_\alpha, \operatorname{Re}z\rangle\,d\mu(\alpha)+\int_{\Omega}\langle \operatorname{Im}z, x_\alpha\rangle\langle \tau_\alpha, \operatorname{Im}z\rangle \,d\mu(\alpha)\\
&=\left(\int_{\Omega}\langle \operatorname{Re}z, x_\alpha\rangle\langle \tau_\alpha, \operatorname{Re}z\rangle\,d\mu(\alpha)+i\int_{\Omega}\langle \operatorname{Im}z, x_\alpha\rangle\langle \tau_\alpha, \operatorname{Re}z\rangle\,d\mu(\alpha)\right) \\
&\quad-i\left( \int_{\Omega}\langle \operatorname{Re}z, x_\alpha\rangle\langle \tau_\alpha, \operatorname{Im}z\rangle \,d\mu(\alpha)+i\int_{\Omega}\langle \operatorname{Im}z, x_\alpha\rangle\langle \tau_\alpha, \operatorname{Im}z\rangle\,d\mu(\alpha)\right)\\
&=\int_{\Omega}\langle \operatorname{Re}z+i\operatorname{Im}z, x_\alpha\rangle\langle \tau_\alpha, \operatorname{Re}z\rangle\,d\mu(\alpha)-i \int_{\Omega}\langle \operatorname{Re}z+i\operatorname{Im}z, x_\alpha\rangle\langle \tau_\alpha, \operatorname{Im}z\rangle\,d\mu(\alpha)\\
&=\int_{\Omega}\langle z,x_\alpha\rangle \langle \tau_\alpha, \operatorname{Re}z\rangle \,d\mu(\alpha)+\int_{\Omega}\langle z,x_\alpha\rangle \langle \tau_\alpha, i\operatorname{Im}z\rangle \,d\mu(\alpha)=\int_{\Omega}\langle z,x_\alpha\rangle \langle \tau_\alpha, z\rangle \,d\mu(\alpha)\\
&\leq b\|\operatorname{Re}z\|^2+b\|\operatorname{Im}z\|^2=b\|z\|^2.
\end{align*}

\end{proof}
\begin{theorem}\label{COMPLEXTOREAL}
If  $(\{x_\alpha\}_{\alpha \in \Omega} , \{\tau_\alpha\}_{\alpha\in \Omega})$ is a continuous frame  for  $\mathbb{C}^m $  such that 
$$ \int_{\Omega}\langle h, \operatorname{Re}x_\alpha \rangle \langle \operatorname{Im}\tau_\alpha, h \rangle \,d\mu(\alpha)=\int_{\Omega}\langle h, \operatorname{Im}x_\alpha \rangle \langle \operatorname{Re}\tau_\alpha, h \rangle\,d\mu(\alpha), \quad\forall h \in \mathbb{C}^m,$$  
then  $(\{\operatorname{Re}x_\alpha\}_{\alpha \in \Omega}\cup\{\operatorname{Im}x_\alpha\}_{\alpha \in \Omega} , \{\operatorname{Re}\tau_\alpha\}_{\alpha \in \Omega}\cup\{\operatorname{Im}\tau_\alpha\}_{\alpha \in \Omega})$ is   a continuous frame for $\mathbb{R}^m. $ Further, if $(\{x_\alpha\}_{\alpha \in \Omega} , \{\tau_\alpha\}_{\alpha \in \Omega}) $ is a tight (resp. Parseval) continuous frame  for $\mathbb{C}^m $, then  $(\{\operatorname{Re}x_\alpha\}_{\alpha \in \Omega}\cup\{\operatorname{Im}x_\alpha\}_{\alpha \in \Omega} $, $ \{\operatorname{Re}\tau_\alpha\}_{\alpha \in \Omega}\cup\{\operatorname{Im}\tau_\alpha\}_{\alpha \in \Omega})$ is  a tight (resp. Parseval) continuous frame for $\mathbb{R}^m. $
\end{theorem}
\begin{proof}
Consider 

\begin{align*}
\int_{\Omega}\langle h,x_\alpha\rangle \langle \tau_\alpha, h\rangle\,d\mu(\alpha)&=\int_{\Omega}\langle h,\operatorname{Re}x_\alpha+i\operatorname{Im}x_\alpha\rangle \langle \operatorname{Re}\tau_\alpha+i\operatorname{Im}\tau_\alpha, h\rangle\,d\mu(\alpha)\\
&=\int_{\Omega}\langle h,\operatorname{Re}x_\alpha\rangle \langle \operatorname{Re}\tau_\alpha, h\rangle\,d\mu(\alpha)+i\int_{\Omega}\langle h,\operatorname{Re}x_\alpha\rangle \langle \operatorname{Im}\tau_\alpha, h\rangle\,d\mu(\alpha)
\\&\quad-i\int_{\Omega}\langle h,\operatorname{Im}x_\alpha\rangle \langle \operatorname{Re}\tau_\alpha, h\rangle\,d\mu(\alpha)+\int_{\Omega}\langle h,\operatorname{Im}x_\alpha\rangle \langle \operatorname{Im}\tau_\alpha, h\rangle\,d\mu(\alpha)\\
&=\int_{\Omega}\langle h,\operatorname{Re}x_\alpha\rangle \langle \operatorname{Re}\tau_\alpha, h\rangle\,d\mu(\alpha)+\int_{\Omega}\langle h,\operatorname{Im}x_\alpha\rangle \langle \operatorname{Im}\tau_\alpha, h\rangle\,d\mu(\alpha), \quad\forall h \in \mathbb{R}^m.
\end{align*}
Therefore, if $a, b$  are  lower and upper frame bounds, respectively, then $ a\|h\|^2\leq \int_{\Omega}\langle h,\operatorname{Re}x_\alpha\rangle \langle \operatorname{Re}\tau_\alpha, h\rangle\,d\mu(\alpha)$ $+\int_{\Omega}\langle h,\operatorname{Im}x_\alpha\rangle \langle \operatorname{Im}\tau_\alpha, h\rangle\,d\mu(\alpha)\leq b\|h\|^2, \forall h \in \mathbb{R}^m.$
\end{proof}

\begin{proposition}
Let $(\{x_\alpha\}_{\alpha\in \Omega} , \{\tau_\alpha\}_{\alpha\in \Omega})$ be a  continuous frame for $\mathcal{H}.$  Then $\mathcal{H} $ is finite dimensional if and only if 
$$ \int_{\Omega}\langle x_\alpha,\tau_\alpha \rangle\,d\mu(\alpha) =\int_{\Omega}\langle \tau_\alpha, x_\alpha \rangle\,d\mu(\alpha)<\infty.$$
In particular, if dimension of $\mathcal{H} $ is finite, $\langle x_\alpha,\tau_\alpha \rangle>0,\forall \alpha\in \Omega $ and $ \inf_{\alpha\in\Omega}\langle x_\alpha,\tau_\alpha \rangle$ is  positive, then $ \mu(\Omega)<\infty.$
\end{proposition}
\begin{proof}
Let $\{e_k\}_{k\in \mathbb{L}}$ be an orthonormal basis for $ \mathcal{H}$, and  $a, b$ be frame bounds for $(\{x_\alpha\}_{\alpha\in \Omega} , \{\tau_\alpha\}_{\alpha\in \Omega})$.

$(\Rightarrow)$ Now $ \mathbb{L}$  is finite. Then 

\begin{align*}
\int_{\Omega}\langle x_\alpha,\tau_\alpha \rangle \,d\mu(\alpha)&=\int_{\Omega}\sum_{k\in\mathbb{L}}\langle x_\alpha,e_k \rangle\langle e_k,\tau_\alpha \rangle\,d\mu(\alpha)=\sum_{k\in\mathbb{L}}\int_{\Omega}\langle x_\alpha,e_k \rangle\langle e_k,\tau_\alpha \rangle\,d\mu(\alpha) \\
&\leq b\sum_{k\in\mathbb{L}}\|e_k\|^2=b\operatorname{Card}(\mathbb{L})<\infty.
\end{align*}
Since $(\{x_\alpha\}_{\alpha\in \Omega} , \{\tau_\alpha\}_{\alpha\in \Omega})$ is a continuous frame, we must have 
$$\int_{\Omega}\langle x_\alpha,e_k \rangle\langle e_k,\tau_\alpha \rangle\,d\mu(\alpha)=\overline{\int_{\Omega}\langle x_\alpha,e_k \rangle\langle e_k,\tau_\alpha \rangle} \,d\mu(\alpha).$$
 Therefore $ \int_{\Omega}\langle x_\alpha,\tau_\alpha \rangle \,d\mu(\alpha)=\int_{\Omega}\langle \tau_\alpha, x_\alpha \rangle\,d\mu(\alpha)<\infty.$

 $(\Leftarrow)$	$ \operatorname{dim}\mathcal{H}=\sum_{k\in\mathbb{L}}\|e_k\|^2\leq (1/a)\sum_{k\in\mathbb{L}}\int_{\Omega}\langle e_k,x_\alpha \rangle\langle \tau_\alpha, e_k \rangle\,d\mu(\alpha)=(1/a)\int_{\Omega}\sum_{k\in\mathbb{L}}\langle e_k, x_\alpha\rangle\langle \tau_\alpha, e_k \rangle\,d\mu(\alpha)=(1/a)\int_{\Omega}\langle x_\alpha,\tau_\alpha\rangle\,d\mu(\alpha)<\infty.$
 
 For the second, $ 0<\inf_{\alpha\in \Omega}\langle x_\alpha,\tau_\alpha \rangle\mu(\Omega)=\int _{\Omega}\inf_{\alpha\in \Omega}\langle x_\alpha,\tau_\alpha \rangle\,d\mu(\alpha)\leq \int _{\Omega}\langle x_\alpha,\tau_\alpha \rangle\,d\mu(\alpha)<\infty. $
\end{proof}
\begin{proposition}\label{INTERESTINGEXAMPLEPROPOSITION}
Let $I\subseteq \mathbb{R}$ be a bounded interval. Let  $ a_\alpha, b_\alpha \geq 0,\theta_\alpha, \phi_\alpha \in \mathbb{R}, \forall \alpha \in I$ and $I \ni \alpha \mapsto \theta_\alpha \in \mathbb{R}$, $I \ni \alpha \mapsto \phi_\alpha \in \mathbb{R}$ be continuous. Then $\left\{ x_\alpha\coloneqq\begin{bmatrix}
a_\alpha\cos\theta_\alpha \\
a_\alpha\sin\theta_\alpha
\end{bmatrix}\right\}_{\alpha \in I}$ is a tight continuous frame for $ \mathbb{R}^2$ w.r.t.  $\left\{ \tau_\alpha\coloneqq \begin{bmatrix}
b_\alpha\cos\phi_\alpha \\
b_\alpha\sin\phi_\alpha
\end{bmatrix}\right\}_{\alpha \in I}$ if and only if 
$ \int_{I}\begin{bmatrix}
a_\alpha b_\alpha\cos(\theta_\alpha+\phi_\alpha) \\
a_\alpha b_\alpha\sin(\theta_\alpha+\phi_\alpha)\\
a_\alpha b_\alpha\sin(\theta_\alpha-\phi_\alpha)
\end{bmatrix}\,d\alpha=\begin{bmatrix}
0 \\
0\\
0
\end{bmatrix}.$
\end{proposition}
\begin{proof}
Matrix of  $ S_{x,\tau}$  is 
\begin{align*}
\begin{bmatrix}
\int_{I} a_\alpha b_\alpha\cos\theta_\alpha\cos\phi_\alpha\,d\alpha & \int_{I} a_\alpha b_\alpha\sin\theta_\alpha\cos\phi_\alpha \,d\alpha\\
\int_{I} a_\alpha b_\alpha\cos\theta_\alpha\sin\phi_\alpha\,d\alpha &\int_{I} a_\alpha b_\alpha\sin\theta_\alpha\sin\phi_\alpha\,d\alpha	\\
\end{bmatrix}.
\end{align*}
We next observe that  $ S_{x,\tau}=aI_{\mathbb{R}^2}$ for some $ a>0$ if and only if 

\begin{align*}
 \int_{I}\begin{bmatrix}
 a_\alpha b_\alpha \cos(\theta_\alpha +\phi_\alpha ) \\
 a_\alpha b_\alpha \sin(\theta_\alpha +\phi_\alpha )\\
 a_\alpha b_\alpha \sin(\theta_\alpha -\phi_\alpha )
 \end{bmatrix}\,d\alpha=\begin{bmatrix}
 0 \\
 0\\
 0
 \end{bmatrix}.
\end{align*}
\end{proof}

\section{Further extension}\label{FURTHEREXTENSION}
\begin{definition}
A collection $ \{A_\alpha\}_{\alpha \in \Omega} $  in $ \mathcal{B}(\mathcal{H}, \mathcal{H}_0)$ is said to be a weak  \textit{ continuous operator-valued frame} (we write weak continuous (ovf)) in $ \mathcal{B}(\mathcal{H}, \mathcal{H}_0) $  with respect to a collection  $ \{\Psi_\alpha\}_{\alpha \in \Omega}  $ in $ \mathcal{B}(\mathcal{H}, \mathcal{H}_0) $ if 
\begin{enumerate}[\upshape(i)]
\item for each  $h \in \mathcal{H}$, both  maps   $\Omega \ni \alpha \mapsto A_\alpha h\in \mathcal{H}_0$ and $\Omega \ni\alpha \mapsto \Psi_\alpha h\in \mathcal{H}_0$ are measurable,
\item the map (we call as frame operator) $S_{A,\Psi} : \mathcal{H} \ni h \mapsto \int_{\Omega}\Psi_\alpha^*A_\alpha h  \,d\mu(\alpha)\in \mathcal{H}$  is a well-defined bounded positive invertible operator.
\end{enumerate}	
 Notions of frame bounds, optimal bounds, tight frame, Parseval frame, Bessel are in same fashion  as in Definition \ref{1}.
 
 For fixed $ \Omega$, $\mathcal{H}, \mathcal{H}_0, $ and $ \{\Psi_\alpha \}_{\alpha \in \Omega}$, the set of all weak  continuous operator-valued  frames in $ \mathcal{B}(\mathcal{H}, \mathcal{H}_0) $ with respect to collection  $ \{\Psi_\alpha \}_{\alpha \in \Omega}$ is denoted by $ \mathscr{F}^\text{w}_\Psi.$
\end{definition}

 \begin{proposition}
 A collection $ \{A_\alpha\}_{\alpha \in \Omega} $   in $ \mathcal{B}(\mathcal{H}, \mathcal{H}_0)$ is  a weak continuous  (ovf)  w.r.t. $ \{\Psi_\alpha\}_{\alpha \in \Omega}  $ in $ \mathcal{B}(\mathcal{H}, \mathcal{H}_0) $ if and only if there exist $ a, b, r >0$ such that 
 \begin{enumerate}[\upshape(i)]
 \item for each  $h \in \mathcal{H}$, both  maps   $\Omega \ni \alpha \mapsto A_\alpha h\in \mathcal{H}_0$, $\Omega \ni\alpha \mapsto \Psi_\alpha h\in \mathcal{H}_0$ are measurable,
 \item $\|\int_{\Omega}\Psi_\alpha^*A_\alpha h\,d\mu(\alpha) \|\leq r\|h\|, \forall h \in \mathcal{H},$
\item $a\|h\|^2\leq\int_{\Omega}\langle A_\alpha h, \Psi_\alpha h\rangle \,d\mu(\alpha) \leq b\|h\|^2, \forall h \in \mathcal{H},$
\item $\int_{\Omega}\Psi_\alpha^*A_\alpha h\,d\mu(\alpha)=\int_{\Omega}A_\alpha^*\Psi_\alpha h\,d\mu(\alpha), \forall h \in \mathcal{H}$.
\end{enumerate}
If the Hilbert space is complex, then  condition \text{\upshape(iv)} can be dropped.
 \end{proposition}
\begin{proposition}
Let $(\{A_\alpha\}_{\alpha\in \Omega}, \{\Psi_\alpha\}_{\alpha\in \Omega}) $ be  a weak continuous (ovf)   in $ \mathcal{B}(\mathcal{H}, \mathcal{H}_0)$  with an upper frame  bound $b$. If $\{\alpha\}$ is measurable and $\Psi_\alpha^*A_\alpha\geq 0, \forall \alpha \in \Omega ,$ then $ \mu(\{\alpha\})\|\Psi_\alpha^*A_\alpha\|\leq b, \forall \alpha \in \Omega.$
\end{proposition}

\begin{definition}
A weak continuous (ovf)  $(\{B_\alpha\}_{\alpha\in \Omega}, \{\Phi_\alpha\}_{\alpha\in \Omega})$  in $\mathcal{B}(\mathcal{H}, \mathcal{H}_0)$ is said to be dual of a weak continuous (ovf) $ ( \{A_\alpha\}_{\alpha\in \Omega}, \{\Psi_\alpha\}_{\alpha\in \Omega})$ in $\mathcal{B}(\mathcal{H}, \mathcal{H}_0)$  if  $ \int_{\Omega }B_\alpha^*\Psi_\alpha h\,d\mu(\alpha)= \int_{\Omega}\Phi^*_\alpha A_\alpha h\,d\mu(\alpha)=h,\forall h \in \mathcal{H}$. The `weak continuous (ovf)' $( \{\widetilde{A}_\alpha\coloneqq A_\alpha S_{A,\Psi}^{-1}\}_{\alpha\in \Omega},\{\widetilde{\Psi}_\alpha\coloneqq\Psi_\alpha S_{A,\Psi}^{-1}\}_{\alpha \in \Omega})$, which is a `dual' of $ (\{A_\alpha\}_{\alpha\in \Omega}, \{\Psi_\alpha\}_{\alpha\in \Omega})$ is called the canonical dual of $ (\{A_\alpha\}_{\alpha\in \Omega}, \{\Psi_\alpha\}_{\alpha\in \Omega})$. 
\end{definition}  
\begin{proposition}
Let $( \{A_\alpha\}_{\alpha\in \Omega}, \{\Psi_\alpha\}_{\alpha\in \Omega} )$ be a weak continuous (ovf) in $ \mathcal{B}(\mathcal{H}, \mathcal{H}_0).$  If $ h \in \mathcal{H}$ has representation  $ h=\int_{\Omega}A_\alpha^*f(\alpha) \,d \mu (\alpha)= \int_{\Omega}\Psi_\alpha^*g(\alpha) \,d \mu (\alpha), $ for some measurable  $ f,g : \Omega \rightarrow \mathcal{H}_0$, then 
$$ \int_{\Omega}\langle f(\alpha),g(\alpha)\rangle\, d \mu (\alpha)=\int_{\Omega}\langle \widetilde{\Psi}_\alpha h,\widetilde{A}_\alpha h\rangle \,d \mu (\alpha)+\int_{\Omega}\langle f(\alpha)-\widetilde{\Psi}_\alpha h,g(\alpha)-\widetilde{A}_\alpha h\rangle \,d \mu (\alpha). $$
\end{proposition}  
  
\begin{theorem}\label{CANONICALDUALFRAMEPROPERTYOPERATORVERSIONWEAK}
Let $ (\{A_\alpha\}_{\alpha\in \Omega},\{\Psi_\alpha\}_{\alpha\in \Omega}) $ be a weak continuous (ovf) with frame bounds $ a$ and $ b.$ Then the following statements are true.
\begin{enumerate}[\upshape(i)]
\item The canonical dual weak continuous (ovf) of the canonical dual weak continuous (ovf)  of $ (\{A_\alpha\}_{\alpha\in \Omega} $, $\{\Psi_\alpha\}_{\alpha\in \Omega} )$ is itself.
\item$ \frac{1}{b}, \frac{1}{a}$ are frame bounds for the canonical dual of $ (\{A_\alpha\}_{\alpha\in \Omega},\{\Psi_\alpha\}_{\alpha\in \Omega}).$
\item If $ a, b $ are optimal frame bounds for $( \{A_\alpha\}_{\alpha\in \Omega} , \{\Psi_\alpha\}_{\alpha\in \Omega}),$ then $ \frac{1}{b}, \frac{1}{a}$ are optimal  frame bounds for its canonical dual.
\end{enumerate} 
\end{theorem} 
 \begin{definition}
A weak continuous (ovf)  $(\{B_\alpha\}_{\alpha\in \Omega},  \{\Phi_\alpha\}_{\alpha\in \Omega})$  in $\mathcal{B}(\mathcal{H}, \mathcal{H}_0)$ is said to be orthogonal to a weak continuous (ovf)  $( \{A_\alpha\}_{\alpha\in \Omega}, \{\Psi_\alpha\}_{\alpha\in \Omega})$ in $\mathcal{B}(\mathcal{H}, \mathcal{H}_0)$ if $ \int_{\Omega}B_\alpha^*\Psi_\alpha h\,d\mu(\alpha)= \int_{\Omega}\Phi^*_\alpha A_\alpha h\,d\mu(\alpha)=0,\forall h \in \mathcal{H}$.
\end{definition}  
\begin{proposition}
Two orthogonal weak continuous operator-valued frames  have common dual weak continuous (ovf).	
\end{proposition} 
\begin{proof}
Let  $ (\{A_\alpha\}_{\alpha\in \Omega}, \{\Psi_\alpha\}_{\alpha\in \Omega}) $ and $ (\{B_\alpha\}_{\alpha\in \Omega}, \{\Phi_\alpha\}_{\alpha\in \Omega}) $ be  two orthogonal weak  continuous operator-valued frames in $\mathcal{B}(\mathcal{H}, \mathcal{H}_0)$. Define $ C_\alpha\coloneqq A_\alpha S_{A,\Psi}^{-1}+B_\alpha S_{B,\Phi}^{-1},\Xi_\alpha\coloneqq \Psi _\alpha S_{A,\Psi}^{-1}+\Phi_\alpha S_{B,\Phi}^{-1}, \forall \alpha \in \Omega$. Then  for all $h, g \in \mathcal{H}$,

\begin{align*}
\langle S_{C,\Xi}h, g\rangle &=\int_{\Omega}\langle C_\alpha h, \Xi_\alpha g\rangle \,d\mu(\alpha)\\
&=\int_{\Omega}\langle (A_\alpha S_{A,\Psi}^{-1}+B_\alpha S_{B,\Phi}^{-1})h, (\Psi _\alpha S_{A,\Psi}^{-1}+\Phi_\alpha S_{B,\Phi}^{-1}) g\rangle \,d\mu(\alpha)\\
&=\int_{\Omega}\langle A_\alpha S_{A,\Psi}^{-1}h, \Psi _\alpha S_{A,\Psi}^{-1} g\rangle \,d\mu(\alpha)+\int_{\Omega}\langle A_\alpha S_{A,\Psi}^{-1}h, \Phi_\alpha S_{B,\Phi}^{-1} g\rangle \,d\mu(\alpha)\\
&\quad +\int_{\Omega}\langle B_\alpha S_{B,\Phi}^{-1}h, \Psi _\alpha S_{A,\Psi}^{-1} g\rangle \,d\mu(\alpha)+\int_{\Omega}\langle B_\alpha S_{B,\Phi}^{-1}h, \Phi_\alpha S_{B,\Phi}^{-1} g\rangle \,d\mu(\alpha)\\
&=\langle  S_{A,\Psi}(S_{A,\Psi}^{-1}h),  S_{A,\Psi}^{-1} g\rangle+\langle 0,  S_{B,\Phi}^{-1} g\rangle+\langle 0,  S_{A,\Psi}^{-1} g\rangle+\langle S_{B,\Phi}(S_{B,\Phi}^{-1}h),  S_{B,\Phi}^{-1} g\rangle
\\
&=\langle (S_{A,\Psi}^{-1}+S_{B,\Phi}^{-1})h, g\rangle.
\end{align*}

Hence $S_{C,\Xi}=S_{A,\Psi}^{-1}+S_{B,\Phi}^{-1}$ which is positive  invertible. Therefore $(\{C_\alpha\}_{\alpha\in \Omega}, \{\Xi_\alpha\}_{\alpha\in \Omega})$ is a weak continuous (ovf) in $\mathcal{B}(\mathcal{H}, \mathcal{H}_0)$. Further, for all $h, g \in \mathcal{H}$,

\begin{align*}
\int_{\Omega}\langle \Psi_\alpha^*C_\alpha h, g\rangle \,d\mu(\alpha)&=\int_{\Omega}\langle (A_\alpha S_{A,\Psi}^{-1}+B_\alpha S_{B,\Phi}^{-1}) h, \Psi_\alpha g\rangle \,d\mu(\alpha)\\
&=\int_{\Omega}\langle A_\alpha S_{A,\Psi}^{-1}h, \Psi_\alpha g\rangle \,d\mu(\alpha)+\int_{\Omega}\langle B_\alpha S_{B,\Phi}^{-1} h, \Psi_\alpha g\rangle \,d\mu(\alpha)=\langle h, g \rangle +\langle 0, g \rangle ,
\end{align*}

\begin{align*}
\int_{\Omega}\langle A_\alpha^*\Xi_\alpha h, g\rangle \,d\mu(\alpha)&=\int_{\Omega}\langle (\Psi _\alpha S_{A,\Psi}^{-1}+\Phi_\alpha S_{B,\Phi}^{-1}) h,  A_\alpha g\rangle \,d\mu(\alpha)\\
&=\int_{\Omega}\langle \Psi_\alpha S_{A,\Psi}^{-1}h, A_\alpha g\rangle \,d\mu(\alpha)+\int_{\Omega}\langle \Phi_\alpha S_{B,\Phi}^{-1} h, A_\alpha g\rangle \,d\mu(\alpha)=\langle h, g \rangle +\langle 0, g \rangle ,
\end{align*}
and 

\begin{align*}
\int_{\Omega}\langle \Phi_\alpha^*C_\alpha h, g\rangle \,d\mu(\alpha)&=\int_{\Omega}\langle (A_\alpha S_{A,\Psi}^{-1}+B_\alpha S_{B,\Phi}^{-1}) h, \Phi_\alpha g\rangle \,d\mu(\alpha)\\
&=\int_{\Omega}\langle A_\alpha S_{A,\Psi}^{-1}h, \Phi_\alpha g\rangle \,d\mu(\alpha)+\int_{\Omega}\langle B_\alpha S_{B,\Phi}^{-1} h, \Phi_\alpha g\rangle \,d\mu(\alpha)=\langle 0, g \rangle +\langle h, g \rangle ,
\end{align*}

\begin{align*}
\int_{\Omega}\langle B_\alpha^*\Xi_\alpha h, g\rangle \,d\mu(\alpha)&=\int_{\Omega}\langle (\Psi _\alpha S_{A,\Psi}^{-1}+\Phi_\alpha S_{B,\Phi}^{-1}) h,  B_\alpha g\rangle \,d\mu(\alpha)\\
&=\int_{\Omega}\langle \Psi_\alpha S_{A,\Psi}^{-1}h, B_\alpha g\rangle \,d\mu(\alpha)+\int_{\Omega}\langle \Phi_\alpha S_{B,\Phi}^{-1} h, B_\alpha g\rangle \,d\mu(\alpha)=\langle 0, g \rangle +\langle h, g \rangle .
\end{align*}
Thus $(\{C_\alpha\}_{\alpha\in \Omega}, \{\Xi_\alpha\}_{\alpha\in \Omega})$ is a common dual of  $ (\{A_\alpha\}_{\alpha\in \Omega}, \{\Psi_\alpha\}_{\alpha\in \Omega}) $ and $ (\{B_\alpha\}_{\alpha\in \Omega}, \{\Phi_\alpha\}_{\alpha\in \Omega}).$	
\end{proof}
\begin{proposition}
Let $ (\{A_\alpha\}_{\alpha\in \Omega}, \{\Psi_\alpha\}_{\alpha\in \Omega}) $ and $ (\{B_\alpha\}_{\alpha\in \Omega}, \{\Phi_\alpha\}_{\alpha\in \Omega}) $ be  two Parseval weak continuous operator-valued frames in   $\mathcal{B}(\mathcal{H}, \mathcal{H}_0)$ which are  orthogonal. If $C,D,E,F \in \mathcal{B}(\mathcal{H})$ are such that $ C^*E+D^*F=I_\mathcal{H}$, then  $ (\{A_\alpha C+B_\alpha D\}_{\alpha\in\Omega}, \{\Psi_\alpha E+\Phi_\alpha F\}_{\alpha\in \Omega}) $ is a  Parseval weak continuous (ovf) in  $\mathcal{B}(\mathcal{H}, \mathcal{H}_0)$. In particular,  if scalars $ c,d,e,f$ satisfy $\bar{c}e+\bar{d}f =1$, then $ (\{cA_\alpha+dB_\alpha\}_{\alpha\in \Omega}, \{e\Psi_\alpha+f\Phi_\alpha\}_{\alpha\in \Omega}) $ is   a Parseval weak continuous (ovf).
\end{proposition} 
\begin{proof}
For all $h, g \in \mathcal{H}$,	

\begin{align*}
\langle  S_{AC+BD,\Psi E+\Phi F}h, g\rangle &=\int_{\Omega}\langle (A_\alpha C+B_\alpha D)h, (\Psi_\alpha E+\Phi_\alpha F)g\rangle \,d\mu(\alpha)\\
&=\int_{\Omega}\langle A_\alpha( Ch), \Psi_\alpha (Eg)\rangle \,d\mu(\alpha)+\int_{\Omega}\langle A_\alpha (Ch), \Phi_\alpha (Fg)\rangle \,d\mu(\alpha)\\
&\quad +\int_{\Omega}\langle B_\alpha (Dh), \Psi_\alpha (Eg)\rangle \,d\mu(\alpha)+\int_{\Omega}\langle B_\alpha (Dh), \Phi_\alpha (Fg)\rangle \,d\mu(\alpha)\\
&=\langle Ch, Eg\rangle +\langle 0, Fg\rangle+\langle 0, Eg \rangle +\langle Dh, Fg\rangle\\
&=\langle E^*Ch, g\rangle  +\langle F^*Dh, g\rangle =\langle (E^*C+ F^*D)h, g\rangle=\langle h, g\rangle .
\end{align*}
\end{proof}
\begin{definition}
Two weak continuous operator-valued frames $(\{A_\alpha\}_{\alpha\in \Omega},\{\Psi_\alpha\}_{\alpha\in \Omega} )$  and $ (\{B_\alpha\}_{\alpha\in \Omega}$, $ \{\Phi_\alpha\}_{\alpha\in \Omega} )$   in $ \mathcal{B}(\mathcal{H}, \mathcal{H}_0)$  are called 
disjoint if $(\{A_\alpha\oplus B_\alpha\}_{\alpha \in \Omega}, \{\Psi_\alpha\oplus \Phi_\alpha\}_{\alpha \in \Omega})$ is a weak continuous (ovf) in $ \mathcal{B}(\mathcal{H}\oplus \mathcal{H}, \mathcal{H}_0).$   
\end{definition}
\begin{proposition}
If $(\{A_\alpha\}_{\alpha\in \Omega},\{\Psi_\alpha\}_{\alpha\in \Omega} )$  and $ (\{B_\alpha\}_{\alpha\in \Omega}, \{\Phi_\alpha\}_{\alpha\in \Omega} )$  are  orthogonal weak continuous operator-valued frames  in $ \mathcal{B}(\mathcal{H}, \mathcal{H}_0)$, then  they  are disjoint. Further, if both $(\{A_\alpha\}_{\alpha\in \Omega},\{\Psi_\alpha\}_{\alpha\in \Omega} )$  and $ (\{B_\alpha\}_{\alpha\in \Omega}, \{\Phi_\alpha\}_{\alpha\in \Omega} )$ are  Parseval weak, then $(\{A_\alpha\oplus B_\alpha\}_{\alpha \in \Omega},\{\Psi_\alpha\oplus \Phi_\alpha\}_{\alpha \in \Omega})$ is Parseval weak.
\end{proposition}
\begin{proof}
Let $ h \oplus g , u \oplus v\in \mathcal{H}\oplus \mathcal{H}$. Then 

\begin{align*}
&\langle S_{A\oplus B, \Psi\oplus \Phi}(h\oplus g),  u \oplus v\rangle =\int_{\Omega}\langle (A_\alpha\oplus B_\alpha)(h\oplus g), (\Psi_\alpha\oplus \Phi_\alpha)(u \oplus v)\rangle \,d\mu(\alpha)\\
&=\int_{\Omega}\langle A_\alpha h+ B_\alpha g, \Psi_\alpha u+ \Phi_\alpha v\rangle \,d\mu(\alpha)=\int_{\Omega}\langle A_\alpha h, \Psi_\alpha u\rangle \,d\mu(\alpha)+\int_{\Omega}\langle A_\alpha h,  \Phi_\alpha v\rangle \,d\mu(\alpha)\\
&\quad+\int_{\Omega}\langle B_\alpha g, \Psi_\alpha u\rangle \,d\mu(\alpha)+\int_{\Omega}\langle  B_\alpha g, \Phi_\alpha v\rangle \,d\mu(\alpha)\\
&=\langle S_{A,\Psi}h, u\rangle +\langle 0, u\rangle+\langle 0, v\rangle+\langle S_{B,\Phi}g, v\rangle=\langle S_{A,\Psi}h\oplus S_{B,\Phi}g,u\oplus v\rangle \\
&=\langle (S_{A,\Psi}\oplus S_{B,\Phi})(h \oplus g ),u\oplus v\rangle \quad \implies S_{A\oplus B, \Psi\oplus \Phi}=S_{A,\Psi}\oplus S_{B,\Phi}.
\end{align*}
\end{proof}

\textbf{Characterization}
\begin{theorem}\label{WEAKSEQUENTIALCHARACTERIZATION}
Let $ \{A_\alpha\}_{\alpha\in\Omega}, \{\Psi_\alpha\}_{\alpha\in\Omega}$ be in $ \mathcal{B}(\mathcal{H},\mathcal{H}_0).$ Suppose $ \{e_{\alpha, \beta}\}_{\beta\in\Omega_\alpha }$ is an   orthonormal basis for $ \mathcal{H}_0,$ for each $\alpha \in \Omega.$ Let  $ u_{\alpha,\beta}=A_\alpha^*e_{\alpha,\beta}, v_{\alpha,\beta}=\Psi_\alpha^*e_{\alpha,\beta}, \forall \beta \in  \Omega_\alpha, \forall \alpha\in \Omega.$ Then $(\{A_\alpha\}_{\alpha\in\Omega}, \{\Psi_\alpha\}_{\alpha\in\Omega})$ is   a weak continuous	
\begin{enumerate}[\upshape(i)]
\item (ovf)  in $ \mathcal{B}(\mathcal{H},\mathcal{H}_0)$  with bounds $a $ and $ b$  if and only if  for each  $h \in \mathcal{H}$, both  maps   $\Omega \ni \alpha \mapsto  \sum_{\beta \in\Omega_\alpha}\langle  h,u_{\alpha ,\beta}\rangle e_{\alpha ,\beta}\in \mathcal{H}_0$, $\Omega \ni\alpha \mapsto \sum_{\beta \in\Omega_\alpha}\langle  h,v_{\alpha ,\beta}\rangle e_{\alpha ,\beta}\in \mathcal{H}_0$ are measurable and  the map 
$$ T: \mathcal{H} \ni h \mapsto\int_{\Omega}\sum_{\beta \in \Omega_\alpha}\langle h, u_{\alpha,\beta}\rangle v_{\alpha,\beta} \,d\mu(\alpha)\in  \mathcal{H} $$
is a well-defined bounded positive invertible operator such that $ a\|h\|^2 \leq \langle Th,h \rangle \leq b\|h\|^2, \forall h \in \mathcal{H}. $
\item    Bessel   in $ \mathcal{B}(\mathcal{H},\mathcal{H}_0)$  with bound  $ b$  if and only if  for each  $h \in \mathcal{H}$, both  maps   $\Omega \ni \alpha \mapsto  \sum_{\beta \in\Omega_\alpha}\langle  h,u_{\alpha ,\beta}\rangle e_{\alpha ,\beta}$ $\in \mathcal{H}_0$, $\Omega \ni\alpha \mapsto \sum_{\beta \in\Omega_\alpha}\langle  h,v_{\alpha ,\beta}\rangle e_{\alpha ,\beta}\in \mathcal{H}_0$ are measurable and the map 
$$ T: \mathcal{H} \ni h \mapsto\int_{\Omega}\sum_{\beta \in \Omega_\alpha}\langle h, u_{\alpha,\beta}\rangle v_{\alpha,\beta} \,d\mu(\alpha)\in  \mathcal{H} $$
is a well-defined bounded positive  operator such that $ 0 \leq \langle Th,h \rangle \leq b\|h\|^2, \forall h \in \mathcal{H}. $
\item   (ovf)  in $ \mathcal{B}(\mathcal{H},\mathcal{H}_0)$  with bounds $a $ and $ b$  if and only if for each  $h \in \mathcal{H}$, both  maps   $\Omega \ni \alpha \mapsto  \sum_{\beta \in\Omega_\alpha}\langle  h,u_{\alpha ,\beta}\rangle e_{\alpha ,\beta}\in \mathcal{H}_0$, $\Omega \ni\alpha \mapsto \sum_{\beta \in\Omega_\alpha}\langle  h,v_{\alpha ,\beta}\rangle e_{\alpha ,\beta}\in \mathcal{H}_0$ are measurable  and there exists $ r >0$ such that 
$$\left \|\int_{\Omega}\sum_{\beta \in\Omega_\alpha}\langle h, u_{\alpha,\beta}\rangle v_{\alpha,\beta}\,d\mu(\alpha)\right\|\leq r\|h\|,~\forall h \in \mathcal{H}   ;$$ 
$$\int_{\Omega}\sum_{\beta \in\Omega_\alpha}\langle h, u_{\alpha,\beta}\rangle v_{\alpha,\beta} \,d\mu(\alpha)=\int_{\Omega}\sum_{\beta \in\Omega_\alpha}\langle h, v_{\alpha,\beta}\rangle u_{\alpha,\beta}\,d\mu(\alpha) ,~\forall h \in \mathcal{H} ;$$
$$a\|h\|^2\leq \int_{\Omega}\sum_{\beta \in\Omega_\alpha}\langle h, u_{\alpha,\beta}\rangle \langle  v_{\alpha,\beta} , h\rangle \,d\mu(\alpha) \leq b\|h\|^2 ,~ \forall h \in \mathcal{H}.$$
\item  Bessel 	  in $ \mathcal{B}(\mathcal{H},\mathcal{H}_0)$  with bound  $ b$  if and only if for each  $h \in \mathcal{H}$, both  maps   $\Omega \ni \alpha \mapsto  \sum_{\beta \in\Omega_\alpha}\langle  h,u_{\alpha ,\beta}\rangle e_{\alpha ,\beta}$ $\in \mathcal{H}_0$, $\Omega \ni\alpha \mapsto \sum_{\beta \in\Omega_\alpha}\langle  h,v_{\alpha ,\beta}\rangle e_{\alpha ,\beta}\in \mathcal{H}_0$ are measurable  and there exists $ r >0$ such that 
 $$\left \|\int_{\Omega}\sum_{\beta \in\Omega_\alpha}\langle h, u_{\alpha,\beta}\rangle v_{\alpha,\beta}\,d\mu(\alpha)\right\|\leq r\|h\|,~\forall h \in \mathcal{H}   ;$$ 
 $$\int_{\Omega}\sum_{\beta \in\Omega_\alpha}\langle h, u_{\alpha,\beta}\rangle v_{\alpha,\beta} \,d\mu(\alpha)=\int_{\Omega}\sum_{\beta \in\Omega_\alpha}\langle h, v_{\alpha,\beta}\rangle u_{\alpha,\beta}\,d\mu(\alpha) ,~\forall h \in \mathcal{H} ;$$
 $$0\leq \int_{\Omega}\sum_{\beta \in\Omega_\alpha}\langle h, u_{\alpha,\beta}\rangle \langle  v_{\alpha,\beta} , h\rangle \,d\mu(\alpha) \leq b\|h\|^2 ,~ \forall h \in \mathcal{H}.$$
 \end{enumerate}	
 \end{theorem}

\textbf{Similarity  of weak continuous operator-valued  frames}
 \begin{definition}
A weak continuous (ovf)  $(\{B_\alpha\}_{\alpha\in \Omega} ,\{\Phi_\alpha\}_{\alpha\in \Omega})$ in $ \mathcal{B}(\mathcal{H}, \mathcal{H}_0)$  is said to be right-similar  to a weak continuous (ovf) $(\{A_\alpha\}_{\alpha\in \Omega},\{\Psi_\alpha\}_{\alpha\in \Omega})$ in $ \mathcal{B}(\mathcal{H}, \mathcal{H}_0)$ if there exist invertible  $ R_{A,B}, R_{\Psi, \Phi} \in \mathcal{B}(\mathcal{H})$  such that $B_\alpha=A_\alpha R_{A,B} , \Phi_\alpha=\Psi_\alpha R_{\Psi, \Phi}, \forall \alpha \in \Omega. $
 \end{definition}
 \begin{proposition}
 Let $ \{A_\alpha\}_{\alpha\in \Omega}\in \mathscr{F}^\text{w}_\Psi$  with frame bounds $a, b,$  let $R_{A,B}, R_{\Psi, \Phi} \in \mathcal{B}(\mathcal{H})$ be positive, invertible, commute with each other, commute with $ S_{A, \Psi}$, and let $B_\alpha=A_\alpha R_{A,B} , \Phi_\alpha=\Psi_\alpha R_{\Psi, \Phi},  \forall \alpha \in \Omega.$ Then 
 $ \{B_\alpha\}_{\alpha\in \Omega}\in \mathscr{F}^\text{w}_\Phi,$  $ S_{B,\Phi}=R_{\Psi,\Phi}S_{A, \Psi}R_{A,B} $, and    $ \frac{a}{\|R_{A,B}^{-1}\|\|R_{\Psi,\Phi}^{-1}\|}\leq S_{B, \Phi} \leq b\|R_{A,B}R_{\Psi,\Phi}\|.$ Assuming that $( \{A_\alpha\}_{\alpha\in \Omega}, \{\Psi_\alpha\}_{\alpha\in \Omega})$ is a Parseval weak continuous (ovf), then $ (\{B_\alpha\}_{\alpha\in \Omega},\{\Phi_\alpha\}_{\alpha\in \Omega} ) $ is a Parseval  weak continuous (ovf) if and only if   $ R_{\Psi, \Phi}R_{A,B}=I_\mathcal{H}.$  
 \end{proposition}

 \begin{proposition}
 Let $ \{A_\alpha\}_{\alpha\in \Omega}\in \mathscr{F}^\text{w}_\Psi,$ $ \{B_\alpha\}_{\alpha\in \Omega}\in \mathscr{F}^\text{w}_\Phi$ and   $B_\alpha=A_\alpha R_{A,B} , \Phi_\alpha=\Psi_\alpha R_{\Psi, \Phi},  \forall \alpha \in \Omega$, for some invertible $ R_{A,B} ,R_{\Psi, \Phi} \in \mathcal{B}(\mathcal{H}).$ Then $  S_{B,\Phi}=R_{\Psi,\Phi}^*S_{A, \Psi}R_{A,B}.$ Assuming that $ (\{A_\alpha\}_{\alpha\in \Omega},\{\Psi_\alpha\}_{\alpha\in \Omega})$ is a  Parseval weak continuous (ovf), then $(\{B_\alpha\}_{\alpha\in \Omega},  \{\Phi_\alpha\}_{\alpha\in \Omega})$ is a Parseval  weak continuous (ovf) if and only if   $ R_{\Psi, \Phi}^*R_{A,B}=I_\mathcal{H}.$
 \end{proposition}
\begin{proof}
For all $h, g \in \mathcal{H}$, 

\begin{align*}
\langle S_{B,\Phi}h, g\rangle &=\int_{\Omega}\langle B_\alpha h, \Phi_\alpha g\rangle \,d\mu(\alpha)=\int_{\Omega}\langle A_\alpha (R_{A,B}h), \Psi_\alpha (R_{\Psi, \Phi}g)\rangle \,d\mu(\alpha)\\
&=\langle S_{A,\Psi}(R_{A,B}), R_{\Psi, \Phi}g\rangle=\langle R_{\Psi,\Phi}^*S_{A, \Psi}R_{A,B}h, g\rangle .
\end{align*}
\end{proof}
\begin{remark}
For every weak continuous (ovf) $(\{A_\alpha\}_{\alpha \in \Omega},\{\Psi_\alpha\}_{\alpha \in \Omega})$, each  of `weak continuous operator-valued frames'  $( \{A_\alpha S_{A, \Psi}^{-1}\}_{\alpha \in \Omega}, \{\Psi_\alpha\}_{\alpha \in \Omega}),$   $( \{A_\alpha S_{A, \Psi}^{-1/2}\}_{\alpha \in \Omega}, \{\Psi_\alpha S_{A,\Psi}^{-1/2}\}_{\alpha \in \Omega}),$ and  $ (\{A_\alpha \}_{\alpha \in \Omega}, \{\Psi_\alpha S_{A,\Psi}^{-1}\}_{\alpha \in \Omega})$ is  a Parseval weak  continuous (ovf)  which is right-similar to  $ (\{A_\alpha\}_{\alpha \in \Omega} , \{\Psi_\alpha\}_{\alpha \in \Omega}).$  
\end{remark}

 \textbf{The case $\mathcal{H}_0=\mathbb{K}$ of  weak continuous operator-valued frames}
 \begin{definition}\label{WEAKFRAMEDEFINITION}
 	
 A set of vectors   $ \{x_\alpha\}_{\alpha\in \Omega}$  in a Hilbert space  $\mathcal{H}$ is said to be a weak continuous  frame    w.r.t.  a set $ \{\tau_\alpha\}_{\alpha\in \Omega}$ in $\mathcal{H}$  if 
 \begin{enumerate}[\upshape(i)]
 \item for each  $h \in \mathcal{H}$, both    maps   $\Omega \ni \alpha \mapsto\langle  h, x_\alpha\rangle \in \mathbb{K}$ and $\Omega \ni\alpha \mapsto \langle  h, \tau_\alpha \rangle \in\mathbb{K}$ are measurable,
 \item the map (we call as frame operator) $S_{x,\tau} : \mathcal{H} \ni h \mapsto \int_{\Omega}\langle  h, x_\alpha\rangle\tau_\alpha  \,d\mu(\alpha)\in \mathcal{H} $  is a well-defined bounded positive invertible operator.
 \end{enumerate}	
 Notions of frame bounds, optimal bounds, tight frame, Parseval frame, Bessel are similar  to the same  in Definition \ref{SEQUENTIAL2}.
 
 For fixed $ \Omega, \mathcal{H},$  and $ \{\tau_\alpha\}_{\alpha\in \Omega}$,   the set of all weak continuous  frames for $ \mathcal{H}$  w.r.t.  $ \{\tau_\alpha\}_{\alpha\in \Omega}$ is denoted by $ \mathscr{F}^w_\tau.$
 \end{definition}
 \begin{proposition}
 A set of vectors   $ \{x_\alpha\}_{\alpha\in \Omega}$  in   $\mathcal{H}$ is  a weak continuous  frame    w.r.t.  a set $ \{\tau_\alpha\}_{\alpha\in \Omega}$ in $\mathcal{H}$   if and only if there are  $ a, b, r  >0$ such that
 \begin{enumerate}[\upshape(i)]
 \item for each  $h \in \mathcal{H}$, both    maps   $\Omega \ni \alpha \mapsto\langle  h, x_\alpha\rangle \in \mathbb{K}$, $\Omega \ni\alpha \mapsto \langle  h, \tau_\alpha \rangle \in\mathbb{K}$ are measurable, 
 \item $ \|\int_{\Omega}\langle h,x_\alpha\rangle \tau_\alpha \,d\mu(\alpha)\|\leq r\|h\| , \forall h \in \mathcal{H},$
 \item $ a\|h\|^2\leq \int_{\Omega}\langle h,x_\alpha\rangle\langle \tau_\alpha, h \rangle \,d\mu(\alpha)\leq b\|h\|^2 , \forall h \in \mathcal{H},$
 \item $ \int_{\Omega}\langle h,x_\alpha\rangle \tau_\alpha\,d\mu(\alpha)=\int_{\Omega}\langle h,\tau_\alpha\rangle x_\alpha\,d\mu(\alpha),  \forall h \in \mathcal{H}.$
 \end{enumerate}
 If the space is over $ \mathbb{C},$ then          \text{\upshape{(iv)}} can be omitted.
 \end{proposition}

\begin{theorem}
Let $\{x_\alpha\}_{\alpha\in \Omega}, \{\tau_\alpha\}_{\alpha\in \Omega}$ be in $\mathcal{H}$. Define $A_\alpha: \mathcal{H} \ni h \mapsto \langle h, x_\alpha \rangle \in \mathbb{K} $, $\Psi_\alpha: \mathcal{H} \ni h \mapsto \langle h, \tau_\alpha \rangle \in \mathbb{K}, \forall \alpha \in \Omega $. Then   $(\{x_\alpha\}_{\alpha\in \Omega}, \{\tau_\alpha\}_{\alpha\in \Omega})$ is a weak continuous frame for  $\mathcal{H}$ if and only if  $(\{A_\alpha\}_{\alpha\in \Omega}, \{\Psi_\alpha\}_{\alpha\in \Omega})$ is a  weak continuous  operator-valued frame  in $\mathcal{B}(\mathcal{H},\mathbb{K})$.
\end{theorem} 
\begin{proposition}
If $(\{x_\alpha\}_{\alpha\in \Omega}, \{\tau_\alpha\}_{\alpha\in \Omega})$ is a weak continuous frame for  $\mathcal{H}$, then  every $ h \in \mathcal{H}$ can be written as 
$$h =\int_{\Omega}\langle h, S^{-1}_{x, \tau}\tau_\alpha\rangle  x_\alpha\,d\mu(\alpha)=\int_{\Omega}\langle h,\tau_\alpha \rangle S^{-1}_{x, \tau}  x_\alpha \,d\mu(\alpha)=\int_{\Omega}\langle h, S^{-1}_{x, \tau}x_\alpha\rangle  \tau_\alpha\,d\mu(\alpha)=\int_{\Omega}\langle h, x_\alpha\rangle S^{-1}_{x, \tau}  \tau_\alpha\,d\mu(\alpha).$$
\end{proposition}
\begin{proof}
For all $h,g  \in \mathcal{H}$, $\langle h, g\rangle =\langle S_{x, \tau}h, S^{-1}_{x, \tau}g\rangle=\int_{\Omega}\langle h, x_\alpha\rangle \langle \tau_\alpha,  S^{-1}_{x, \tau}g\rangle \,d\mu(\alpha)=\int_{\Omega}\langle h, x_\alpha\rangle \langle S^{-1}_{x, \tau}\tau_\alpha,  g\rangle \,d\mu(\alpha)=\langle\int_{\Omega}\langle h,x_\alpha \rangle S^{-1}_{x, \tau}\tau_\alpha  \,d\mu(\alpha), g\rangle$, $\langle h, g\rangle =\langle S_{x, \tau} S_{x, \tau}^{-1}h, g\rangle=\int_{\Omega}\langle  S^{-1}_{x, \tau}h, \tau_\alpha\rangle \langle x_\alpha, g \rangle \,d\mu(\alpha)=\int_{\Omega}\langle  h, S^{-1}_{x, \tau} \tau_\alpha\rangle \langle x_\alpha, g \rangle \,d\mu(\alpha)$ $=\langle \int_{\Omega}\langle h, S^{-1}_{x, \tau}\tau_\alpha\rangle  x_\alpha\,d\mu(\alpha), g\rangle $.
 \end{proof}
\begin{proposition}
Let $(\{x_\alpha\}_{\alpha\in \Omega},\{\tau_\alpha\}_{\alpha \in \Omega} )$ be a weak continuous frame for  $ \mathcal{H}$  with upper frame bound $b$.  If for some $ \alpha \in \Omega $ we have $\{\alpha\}$ is measurable and  $  \langle x_\alpha, x_\beta \rangle\langle \tau_\beta, x_\alpha \rangle \geq0, \forall \beta  \in \Omega$,  then $ \mu(\{\alpha\})\langle x_\alpha, \tau_\alpha \rangle\leq b$ for that $\alpha. $
\end{proposition}
\begin{definition}\label{WEAKDUALDEFINITIONSEQUENTIAL}
A weak continuous frame   $(\{y_\alpha\}_{\alpha\in\Omega}, \{\omega_\alpha\}_{\alpha\in \Omega})$  for  $\mathcal{H}$ is said to be a dual of weak continuous frame  $ ( \{x_\alpha\}_{\alpha\in \Omega}, \{\tau_\alpha\}_{\alpha\in \Omega})$ for  $\mathcal{H}$  if $ \int_{\Omega}\langle h, x_\alpha\rangle \omega_\alpha\,d\mu(\alpha)=\int_{\Omega}\langle h, \tau_\alpha\rangle y_\alpha\,d\mu(\alpha)=h, \forall h \in  \mathcal{H}$. The `weak continuous frame' $(\{\widetilde{x}_\alpha\coloneqq S_{x,\tau}^{-1}x_\alpha\}_{\alpha\in \Omega},\{\widetilde{\tau}_\alpha\coloneqq S_{x,\tau}^{-1}\tau_\alpha\}_{\alpha \in \Omega} )$, which is a `dual' of $ (\{x_\alpha\}_{\alpha\in \Omega}, \{\tau_\alpha\}_{\alpha\in \Omega})$ is called the canonical dual of $ (\{x_\alpha\}_{\alpha\in \Omega}, \{\tau_\alpha\}_{\alpha\in \Omega})$.
 \end{definition}
\begin{proposition}
Let $( \{x_\alpha\}_{\alpha\in \Omega},\{\tau_\alpha\}_{\alpha\in \Omega} )$ be a weak continuous frame for  $\mathcal{H}.$ If $ h \in \mathcal{H}$ has representation  $ h=\int_{\Omega}f(\alpha)x_\alpha\,d\mu(\alpha)= \int_{\Omega}g(\alpha)\tau_\alpha\,d\mu(\alpha), $ for some  measurable  $ f,g : \Omega \rightarrow \mathbb{K}$,  then 
$$ \int_{\Omega}f(\alpha)\overline{g(\alpha)} \,d\mu(\alpha)=\int_{\Omega}\langle h, \widetilde{\tau}_\alpha\rangle\langle \widetilde{x}_\alpha , h \rangle\,d\mu(\alpha)+\int_{\Omega}( f(\alpha)-\langle h, \widetilde{\tau}_\alpha\rangle)(\overline{g(\alpha)}-\langle \widetilde{x}_\alpha, h\rangle)\,d\mu(\alpha). $$
\end{proposition}
\begin{theorem}
Let $( \{x_\alpha\}_{\alpha\in \Omega},\{\tau_\alpha\}_{\alpha\in \Omega} )$ be a weak continuous frame for $ \mathcal{H}$ with frame bounds $ a$ and $ b.$ Then
\begin{enumerate}[\upshape(i)]
\item The canonical dual weak continuous frame of the canonical dual weak continuous frame  of $ (\{x_\alpha\}_{\alpha\in \Omega} $, $\{\tau_\alpha\}_{\alpha\in \Omega} )$ is itself.
\item$ \frac{1}{b}, \frac{1}{a}$ are frame bounds for the canonical dual of $ (\{x_\alpha\}_{\alpha\in \Omega},\{\tau_\alpha\}_{\alpha\in\Omega}).$
\item If $ a, b $ are optimal frame bounds for $( \{x_\alpha\}_{\alpha\in \Omega} , \{\tau_\alpha\}_{\alpha\in \Omega}),$ then $ \frac{1}{b}, \frac{1}{a}$ are optimal  frame bounds for its canonical dual.
\end{enumerate} 
\end{theorem}

\begin{definition}\label{WEAKORTHOGONALDEFINITIONSEQUENTIAL}
A weak continuous frame   $(\{y_\alpha\}_{\alpha\in \Omega},  \{\omega_\alpha\}_{\alpha\in \Omega})$  for  $\mathcal{H}$ is said to be orthogonal to a weak continuous frame   $( \{x_\alpha\}_{\alpha\in \Omega}, \{\tau_\alpha\}_{\alpha\in \Omega})$ for $\mathcal{H}$ if $\int_{\Omega}\langle h, x_\alpha\rangle \omega_\alpha\,d\mu(\alpha)= \int_{\Omega}\langle h, \tau_\alpha\rangle y_\alpha\,d\mu(\alpha)=0, \forall h \in  \mathcal{H}.$
\end{definition}

\begin{proposition}
Two orthogonal weak continuous frames  have a common dual weak continuous frame.	
\end{proposition}
\begin{proof}
Let  $ (\{x_\alpha\}_{\alpha\in \Omega}, \{\tau_\alpha\}_{\alpha\in \Omega}) $ and $ (\{y_\alpha\}_{\alpha\in \Omega}, \{\omega_\alpha\}_{\alpha\in \Omega}) $ be   orthogonal weak continuous frames for  $\mathcal{H}$. Define $ z_\alpha\coloneqq S_{x,\tau}^{-1}x_\alpha+S_{y,\omega}^{-1}y_\alpha,\rho_\alpha\coloneqq S_{x,\tau}^{-1}\tau_\alpha+S_{y,\omega}^{-1}\omega_\alpha, \forall \alpha \in \Omega$. For $h,g  \in \mathcal{H}$,  

\begin{align*}
\langle S_{z,\rho}h, g\rangle &=\int_{\Omega}\langle h, z_\alpha\rangle \langle \rho_\alpha , g\rangle\,d\mu(\alpha)\\
&=\int_{\Omega}\langle h, S_{x,\tau}^{-1}x_\alpha+S_{y,\omega}^{-1}y_\alpha\rangle \langle  S_{x,\tau}^{-1}\tau_\alpha+S_{y,\omega}^{-1}\omega_\alpha, g\rangle\,d\mu(\alpha)\\
&=\int_{\Omega}\langle S_{x,\tau}^{-1} h,x_\alpha \rangle \langle  \tau_\alpha,  S_{x,\tau}^{-1}g\rangle\,d\mu(\alpha)+\int_{\Omega}\langle  S_{x,\tau}^{-1}h, x_\alpha\rangle \langle  \omega_\alpha, S_{y,\omega}^{-1}g\rangle\,d\mu(\alpha)\\
&\quad +\int_{\Omega}\langle  S_{y,\omega}^{-1}h, y_\alpha\rangle \langle  \tau_\alpha, S_{x,\tau}^{-1}g\rangle\,d\mu(\alpha)+\int_{\Omega}\langle  S_{y,\omega}^{-1}h, y_\alpha\rangle \langle  \omega_\alpha, S_{y,\omega}^{-1}g\rangle\,d\mu(\alpha)\\
&=\langle S_{x,\tau}S_{x,\tau}^{-1}h, S_{x,\tau}^{-1}g\rangle +\left\langle\int_{\Omega}\langle  S_{x,\tau}^{-1}h, x_\alpha\rangle   \omega_\alpha\,d\mu(\alpha), S_{y,\omega}^{-1}g \right \rangle \\
&\quad +\left\langle \int_{\Omega}\langle  S_{y,\omega}^{-1}h, y_\alpha\rangle   \tau_\alpha\,d\mu(\alpha), S_{x,\tau}^{-1}g \right \rangle+\langle S_{y,\omega}S_{y,\omega}^{-1} h, S_{y,\omega}^{-1}g\rangle\\
&=\langle S_{x,\tau}^{-1}h, g\rangle+\langle 0, S_{y,\omega}^{-1}g\rangle +\langle 0, S_{x,\tau}^{-1}g\rangle+\langle S_{y,\omega}^{-1} h, g\rangle=\langle S_{x,\tau}^{-1}h+S_{y,\omega}^{-1}h, g\rangle.
\end{align*}
Therefore $S_{z,\rho}=S_{x,\tau}^{-1}+S_{y,\omega}^{-1}$ which tells that $ (\{z_\alpha\}_{\alpha\in \Omega}, \{\rho_\alpha\}_{\alpha\in \Omega}) $ is a  weak continuous frame for  $\mathcal{H}$. For duality, let  $h, g \in \mathcal{H}$. Then

\begin{align*}
\left\langle \int_{\Omega}\langle h, x_\alpha\rangle \rho_\alpha \,d\mu(\alpha), g\right\rangle &=\left\langle \int_{\Omega}\langle h, x_\alpha\rangle( S_{x,\tau}^{-1}\tau_\alpha+S_{y,\omega}^{-1}\omega_\alpha) \,d\mu(\alpha), g\right\rangle\\
&= \int_{\Omega}\langle h, x_\alpha\rangle \langle \tau_\alpha , S_{x,\tau}^{-1}g\rangle\,d\mu(\alpha)+\int_{\Omega}\langle h, x_\alpha\rangle \langle \omega_\alpha , S_{y, \omega}^{-1}g\rangle\,d\mu(\alpha)=\langle h, g\rangle +\langle 0, S_{y, \omega}^{-1}g\rangle,
\end{align*}

\begin{align*}
\left\langle \int_{\Omega}\langle h, \tau_\alpha\rangle z_\alpha \,d\mu(\alpha), g\right\rangle &=\left\langle \int_{\Omega}\langle h, \tau_\alpha\rangle(S_{x,\tau}^{-1}x_\alpha+S_{y,\omega}^{-1}y_\alpha) \,d\mu(\alpha), g\right\rangle\\
&= \int_{\Omega}\langle h, \tau_\alpha\rangle \langle x_\alpha , S_{x,\tau}^{-1}g\rangle\,d\mu(\alpha)+\int_{\Omega}\langle h, \tau_\alpha\rangle \langle y_\alpha , S_{y, \omega}^{-1}g\rangle\,d\mu(\alpha)=\langle h, g\rangle +\langle 0, S_{y, \omega}^{-1}g\rangle,
\end{align*}
and 

\begin{align*}
\left\langle \int_{\Omega}\langle h, y_\alpha\rangle \rho_\alpha \,d\mu(\alpha), g\right\rangle &=\left\langle \int_{\Omega}\langle h, y_\alpha\rangle( S_{x,\tau}^{-1}\tau_\alpha+S_{y,\omega}^{-1}\omega_\alpha) \,d\mu(\alpha), g\right\rangle\\
&= \int_{\Omega}\langle h, y_\alpha\rangle \langle \tau_\alpha , S_{x,\tau}^{-1}g\rangle\,d\mu(\alpha)+\int_{\Omega}\langle h, y_\alpha\rangle \langle \omega_\alpha , S_{y, \omega}^{-1}g\rangle\,d\mu(\alpha)=\langle 0, S_{y, \omega}^{-1}g\rangle+\langle h, g\rangle,
\end{align*}

\begin{align*}
\left\langle \int_{\Omega}\langle h, \omega_\alpha\rangle z_\alpha \,d\mu(\alpha), g\right\rangle &=\left\langle \int_{\Omega}\langle h, \omega_\alpha\rangle(S_{x,\tau}^{-1}x_\alpha+S_{y,\omega}^{-1}y_\alpha) \,d\mu(\alpha), g\right\rangle\\
&= \int_{\Omega}\langle h, \omega_\alpha\rangle \langle x_\alpha , S_{x,\tau}^{-1}g\rangle\,d\mu(\alpha)+\int_{\Omega}\langle h, \omega_\alpha\rangle \langle y_\alpha , S_{y, \omega}^{-1}g\rangle\,d\mu(\alpha)=\langle 0, S_{x, \tau}^{-1}g\rangle+\langle h, g\rangle.
\end{align*}
\end{proof}
\begin{proposition}
Let $ (\{x_\alpha\}_{\alpha\in \Omega}, \{\tau_\alpha\}_{\alpha\in \Omega}) $ and $ (\{y_\alpha\}_{\alpha\in \Omega}, \{\omega_\alpha\}_{\alpha\in \Omega}) $ be  two Parseval weak continuous frames for  $\mathcal{H}$ which are  orthogonal. If $A,B,C,D \in \mathcal{B}(\mathcal{H})$ are such that $ AC^*+BD^*=I_\mathcal{H}$, then  $ (\{Ax_\alpha+By_\alpha\}_{\alpha\in \Omega}, \{C\tau_\alpha+D\omega_\alpha\}_{\alpha\in \Omega}) $ is a  Parseval weak continuous frame for  $\mathcal{H}$. In particular,  if scalars $ a,b,c,d$ satisfy $a\bar{c}+b\bar{d} =1$, then $ (\{ax_\alpha+by_\alpha\}_{\alpha\in \Omega}, \{c\tau_\alpha+d\omega_\alpha\}_{\alpha\in \Omega}) $ is a  Parseval weak continuous frame for  $\mathcal{H}$.
\end{proposition} 
\begin{proof}
For all $h , g \in \mathcal{H}$,

\begin{align*}
&\langle S_{Ax+By,C\tau+D\omega}h, g\rangle =\int_{\Omega}\langle h,Ax_\alpha+By_\alpha\rangle\langle C\tau_\alpha+D\omega_\alpha, g\rangle \,d\mu(\alpha)\\
&=\int_{\Omega}\langle A^*h,x_\alpha\rangle\langle \tau_\alpha, C^*g\rangle \,d\mu(\alpha)+\int_{\Omega}\langle A^*h,x_\alpha\rangle\langle \omega_\alpha, D^*g\rangle \,d\mu(\alpha)\\
&\quad+\int_{\Omega}\langle B^*h,y_\alpha\rangle\langle \tau_\alpha, C^*g\rangle \,d\mu(\alpha)+\int_{\Omega}\langle B^*h,y_\alpha\rangle\langle \omega_\alpha, D^*g\rangle \,d\mu(\alpha)\\
&=\langle A^*h, C^*g\rangle+\langle 0, D^*g\rangle+\langle 0, C^*g\rangle+\langle B^*h, D^*g\rangle\\
&=\langle CA^*h, g\rangle+\langle DB^*h, g\rangle=\langle (CA^*+DB^*)h, g\rangle=\langle h, g\rangle.
\end{align*}
\end{proof}
\begin{definition}
Two weak continuous frames  $(\{x_\alpha\}_{\alpha\in \Omega},  \{\tau_\alpha\}_{\alpha\in \Omega}) $ and $(\{y_\alpha\}_{\alpha\in \Omega},\{\omega_\alpha\}_{\alpha\in \Omega})$  for $ \mathcal{H}$ are called disjoint if $(\{x_\alpha\oplus y_\alpha\}_{\alpha\in \Omega},\{\tau_\alpha\oplus\omega_\alpha\}_{\alpha\in \Omega})$ is a weak continuous frame for $\mathcal{H}\oplus\mathcal{H} $.
\end{definition} 
\begin{proposition}
If $(\{x_\alpha\}_{\alpha\in \Omega},\{\tau_\alpha\}_{\alpha\in \Omega} )$  and $ (\{y_\alpha\}_{\alpha\in \Omega}, \{\omega_\alpha\}_{\alpha\in \Omega} )$  are  orthogonal weak continuous  frames  for $\mathcal{H}$, then  they  are disjoint. Further, if both $(\{x_\alpha\}_{\alpha\in \Omega},\{\tau_\alpha\}_{\alpha\in \Omega} )$  and $ (\{y_\alpha\}_{\alpha\in \Omega}, \{\omega_\alpha\}_{\alpha\in \Omega} )$ are  Parseval weak, then $(\{x_\alpha\oplus y_\alpha\}_{\alpha \in \Omega},\{\tau_\alpha\oplus \omega_\alpha\}_{ \alpha\in \Omega})$ is Parseval weak.
\end{proposition}
\begin{proof}
Let $ h \oplus g , u \oplus v\in \mathcal{H}\oplus \mathcal{H}$. Then 
\begin{align*}
&\langle S_{x\oplus y, \tau\oplus \omega}(h\oplus g),  u \oplus v\rangle =\int_{\Omega}\langle h\oplus g, x_\alpha \oplus y_\alpha \rangle \langle \tau_\alpha \oplus \omega_\alpha, u \oplus v\rangle \,d\mu(\alpha)\\
&=\int_{\Omega}(\langle h,x_\alpha\rangle + \langle g,y_\alpha\rangle)(\langle \tau_\alpha,u\rangle + \langle \omega_\alpha,v\rangle) \,d\mu(\alpha)=\int_{\Omega}\langle h,x_\alpha\rangle \langle \tau_\alpha,u\rangle  \,d\mu(\alpha)\\
&\quad +\int_{\Omega}\langle h,x_\alpha\rangle  \langle \omega_\alpha,v\rangle \,d\mu(\alpha)+\int_{\Omega}\langle g,y_\alpha\rangle\langle \tau_\alpha,u\rangle  \,d\mu(\alpha)+\int_{\Omega} \langle g,y_\alpha\rangle \langle \omega_\alpha,v\rangle \,d\mu(\alpha)\\
&=\langle S_{x,\tau}h, u\rangle +\langle 0, v\rangle+\langle 0, u\rangle+\langle S_{y,\omega}g, v\rangle=\langle S_{x,\tau}h\oplus S_{y,\omega}g,u\oplus v\rangle \\
&=\langle (S_{x,\tau}\oplus S_{y,\omega})(h \oplus g ),u\oplus v\rangle \quad \implies S_{x\oplus y, \tau\oplus \omega}=S_{x,\tau}\oplus S_{y,\omega}.
\end{align*}
\end{proof}  

\textbf{Similarity}
\begin{definition}
A weak continuous frame  $ (\{y_\alpha\}_{\alpha\in \Omega},\{\omega_\alpha\}_{\alpha\in \Omega})$ for $ \mathcal{H}$ is said to be  similar to  a weak continuous frame $ (\{x_\alpha\}_{\alpha\in \Omega},\{\tau_\alpha\}_{\alpha\in \Omega})$ for $ \mathcal{H}$ if there are invertible  $ T_{x,y}, T_{\tau,\omega} \in \mathcal{B}(\mathcal{H})$ such that $ y_\alpha=T_{x,y}x_\alpha, \omega_\alpha=T_{\tau,\omega}\tau_\alpha,   \forall \alpha \in \Omega.$
\end{definition} 
\begin{proposition}
Let $ \{x_\alpha\}_{\alpha \in \Omega}\in \mathscr{F}^w_\tau$  with frame bounds $a, b,$  let $T_{x,y} , T_{\tau,\omega}\in \mathcal{B}(\mathcal{H})$ be positive, invertible, commute with each other, commute with $ S_{x, \tau}$, and let $y_\alpha=T_{x,y}x_\alpha , \omega_\alpha=T_{\tau,\omega}\tau_\alpha,  \forall \alpha \in \Omega.$ Then $ \{y_\alpha\}_{\alpha\in \Omega}\in \mathscr{F}^w_\tau$, $S_{y,\omega}=T_{\tau,\omega}S_{x, \tau}T_{x,y}$, and $ \frac{a}{\|T_{x,y}^{-1}\|\|T_{\tau,\omega}^{-1}\|}\leq S_{y, \omega} \leq b\|T_{x,y}T_{\tau,\omega}\|$. Assuming that $ (\{x_\alpha\}_{\alpha\in\Omega},\{\tau_\alpha\}_{\alpha\in \Omega})$ is Parseval weak continuous, then $(\{y_\alpha\}_{\alpha\in \Omega},  \{\omega_\alpha\}_{\alpha\in\Omega})$ is Parseval weak continuous if and only if   $ T_{\tau, \omega}T_{x,y}=I_\mathcal{H}.$  
\end{proposition}     
 \begin{proposition}
 Let $ \{x_\alpha\}_{\alpha\in \Omega}\in \mathscr{F}^w_\tau,$ $ \{y_\alpha\}_{\alpha\in \Omega}\in \mathscr{F}^w_\omega$ and   $y_\alpha=T_{x, y}x_\alpha , \omega_\alpha=T_{\tau,\omega}\tau_\alpha,  \forall \alpha \in \Omega$, for some invertible $T_{x,y}, T_{\tau,\omega}\in \mathcal{B}(\mathcal{H}).$ Then 
 $  S_{y,\omega}=T_{\tau,\omega}S_{x, \tau}T_{x,y}^*.$ Assuming that $ (\{x_\alpha\}_{\alpha\in \Omega},\{\tau_\alpha\}_{\alpha\in \Omega})$ is  Parseval weak continuous frame, then $ (\{y_\alpha\}_{\alpha\in \Omega},\{\omega_\alpha\}_{\alpha\in \Omega})$ is  Parseval weak continuous frame if and only if $T_{\tau,\omega}T_{x,y}^*=I_\mathcal{H}.$
 \end{proposition}    
 \begin{proof}
 For all $ h,g \in \mathcal{H}$, 
 
 \begin{align*}
 \langle S_{y,\omega}h, g\rangle &=\int_{\Omega}\langle h,y_\alpha \rangle \langle \omega_\alpha , g \rangle \,d\mu(\alpha)=\int_{\Omega}\langle h, T_{x,y}x_\alpha \rangle \langle T_{\tau,\omega}\tau_\alpha , g \rangle \,d\mu(\alpha)\\
 &=\int_{\Omega}\langle  T_{x,y}^*h, x_\alpha \rangle \langle\tau_\alpha, T_{\tau,\omega}^* g \rangle \,d\mu(\alpha)  =\langle S_{x,\tau}(T^*_{x,y}h), T^*_{\tau,\omega}g\rangle=\langle T_{\tau,\omega}S_{x,\tau}T^*_{x,y}h, g\rangle.
 \end{align*}
 \end{proof}
 
\begin{remark}
For every weak continuous frame  $(\{x_\alpha\}_{\alpha \in \Omega}, \{\tau_\alpha\}_{\alpha \in \Omega}),$ each  of `weak continuous frames'  $( \{S_{x, \tau}^{-1}x_\alpha\}_{\alpha \in \Omega}, \{\tau_\alpha\}_{\alpha \in \Omega})$,    $( \{S_{x, \tau}^{-1/2}x_\alpha\}_{\alpha \in \Omega}, \{S_{x,\tau}^{-1/2}\tau_\alpha\}_{\alpha \in \Omega}),$ and  $ (\{x_\alpha \}_{\alpha \in \Omega}, \{S_{x,\tau}^{-1}\tau_\alpha\}_{\alpha \in \Omega})$ is a Parseval  weak continuous frame which is similar to  $ (\{x_\alpha\}_{\alpha \in \Omega} , \{\tau_\alpha\}_{\alpha \in \Omega} ).$  
\end{remark}

\section{Acknowledgments}
 The first author thanks the National Institute of Technology Karnataka (NITK), Surathkal for giving financial support and the present work of the second author was partially supported by National Board for Higher Mathematics (NBHM), Ministry of Atomic Energy, Government of India (Reference No.2/48(16)/2012/NBHM(R.P.)/R\&D 11/9133).
 \bibliographystyle{plain}
 \bibliography{reference.bib}
 
 \end{document}